\documentclass[11pt]{amsart}
\usepackage{cite}
\usepackage{amsmath}
\usepackage{graphicx,wrapfig}
\usepackage{amssymb,amsgen,color}
\usepackage{amsfonts}
\usepackage{hyperref}

\usepackage{epsfig,epstopdf,color,bm}

\numberwithin{equation}{section}
\newcommand{\ep}{\varepsilon}
\newcommand{\bx}{{\bf x} }
\newcommand{\p}{\partial}
\newcommand{\og}{\omega}
\newcommand{\Og}{\Omega}
\newcommand{\fl}[2]{\frac{#1}{#2}}
\newcommand{\be}{\begin{equation}}
\newcommand{\ee}{\end{equation}}
\newcommand{\nn}{\nonumber}

\def\({\left(}
\def\){\right)}
\def\<{\left\langle}
\def\>{\right\rangle}

 \newtheorem{theorem}{Theorem}[section]

\newtheorem{lemma}[theorem]{Lemma}
\newtheorem{remark}[theorem]{Remark}

\begin{document}

\title[Numerical methods for the logarithmic Schr\"odinger
equation]{Regularized numerical methods for the logarithmic Schr\"odinger equation}
\author[W. Bao]{Weizhu Bao}
\address[W. Bao]{Department of Mathematics,
National University of Singapore, Singapore 119076}
\email{matbaowz@nus.edu.sg}
\urladdr{http://www.math.nus.edu.sg/~bao/}
\author[R. Carles]{R\'emi Carles}
\address[R. Carles]{Univ Rennes, CNRS\\ IRMAR - UMR 6625\\ F-35000
  Rennes\\  France}
\email{Remi.Carles@math.cnrs.fr}
\urladdr{http://carles.perso.math.cnrs.fr/}
\author[C. Su]{Chunmei Su}
\address[C. Su]{Department of Mathematics,
University of Innsbruck, Innsbruck 6020, Austria}
\email{sucm13@163.com}
\author[Q. Tang]{Qinglin Tang}
\address[Q. Tang]{School of Mathematics,
Sichuan University, Chengdu 610064, People's Republic of China}
\email{qinglin\_tang@scu.edu.cn}
\urladdr{http://math.scu.edu.cn/info/1013/3088.htm}
\thanks{This work was partially supported by the Ministry
of Education of Singapore grant
R-146-000-223-112 (MOE2015-T2-2-146) (W. Bao), and  by  the
Fundamental Research Funds for the Central Universities (Q. Tang).} 
\maketitle

\begin{abstract}
We present and analyze two numerical methods for the logarithmic
Schr\"odinger equation (LogSE) consisting of
 a regularized splitting method and a regularized conservative Crank-Nicolson finite
difference method (CNFD). In order to avoid numerical  blow-up and/or to suppress round-off error due to the logarithmic
nonlinearity in the LogSE, a regularized logarithmic Schr\"odinger equation (RLogSE)
with a small regularized parameter $0<\ep\ll1$ is adopted to
approximate the LogSE with linear convergence rate  $O(\ep)$.
Then we use the Lie-Trotter splitting integrator
to solve the RLogSE and establish its error bound  $O(\tau^{1/2}\ln(\ep^{-1}))$ with $\tau>0$ the time step,
which implies an error bound at $O(\ep+\tau^{1/2}\ln(\ep^{-1}))$
for the LogSE by the Lie-Trotter splitting method.
In addition, the CNFD is also applied to discretize the RLogSE,
which conserves the mass and energy in the discretized level.
Numerical results are reported to confirm our error bounds
and to demonstrate rich and complicated dynamics of the LogSE.
\end{abstract}

\section{Introduction}
We consider the logarithmic Schr\"odinger equation (LogSE) which was
originally introduced as a model of nonlinear wave mechanics (cf. \cite{BiMy76})
\be\label{LSE}
\left\{
\begin{aligned}
&i\p_t u(\bx,t)+\Delta u(\bx,t)=\lambda u(\bx,t)\,\ln(|u(\bx,t)|^2),\quad \bx\in \Omega, \quad t>0,\\
& u(\bx,0)=u_0(\bx),\quad \bx\in \Omega,
\end{aligned}
\right.
\ee
where $t$ is time, $\bx=(x_1,\ldots,x_d)^T\in\mathbb{R}^d$ ($d=1,2,3$)
is the spatial coordinate, $u:=u(\bx,t)\in\mathbb{C}$ is the dimensionless wave function, $\lambda\in
\mathbb{R}\backslash\{0\}$ is a dimensionless real constant of
the nonlinear interaction strength, and
$\Omega=\mathbb{R}^d$ or $\Omega\subset\mathbb{R}^d$ is a bounded
domain with homogeneous Dirichlet boundary condition or periodic
boundary condition posted on the boundary. The LogSE now arises from
different applications, such as  quantum mechanics \cite{yasue},
quantum optics \cite{BiMy76,BiMy79,hansson,KEB00,buljan}, nuclear
physics \cite{Hef85,HeRe80}, Bohmian mechanics \cite{DFGL03}, effective quantum gravity\cite{Zlo10}, theory of
superfluidity and  Bose-Einstein condensation \cite{BEC}.
\smallbreak

We emphasize that the nonlinearity $z\mapsto z\ln|z|^2$ is not locally
Lipschitz continuous due to the singularity of the logarithm at the origin, and therefore even the well-posedness of the Cauchy problem for \eqref{LSE} is not completely obvious. We
refer to \cite{CaHa80, CaGa-p,GLN10} for a study of
the Cauchy problem  by compactness methods in a suitable functional
framework (which depends
on the sign of $\lambda$ based on a priori estimate), and to \cite{Hayashi-p} for
an alternative proof relying on
the strong convergence of
suitable approximate solutions when $\lambda<0$.
\smallbreak

Similar to the more usual cubic nonlinear Schr\"odinger equation, the LogSE
conserves the {\sl mass}, {\sl momentum} and {\sl energy}
\cite{CazCourant}, which are defined, respectively, as follows:
\begin{align}
&M(u(\cdot,t)):=\|u(\cdot,t)\|^2_{L^2(\Omega)}\equiv
  \|u_0\|^2_{L^2(\Omega)}=M(u_0),\label{mass}\\
& P(u(\cdot,t)):=\mathrm{Im} \int_\Omega \overline{u}(\bx,t)\nabla
  u(\bx,t)d\bx\equiv \mathrm{Im} \int_\Omega \overline {u_0}(\bx)\nabla
  u_0(\bx)d\bx=P(u_0),\quad t\ge0,\label{mom}\\
&E(u(\cdot,t)):=\|\nabla u(t)\|^2_{L^2(\Omega)}+\lambda \int_{\Omega}|u(\bx,t)|^2\ln(|u(\bx,t)|^2)d\bx\nonumber\\
&\qquad\quad\ \ \,\,\, \equiv\|\nabla u_0\|^2_{L^2(\Omega)}+\lambda \int_{\Omega}|u_0(\bx)|^2\ln(|u_0(\bx)|^2)d\bx=E(u_0). \label{energy}
\end{align}
At this stage, the only difference with the cubic nonlinear Schr\"odinger
equation is  the expression of the energy.
We emphasize that the second term in the energy, referred to as {\sl
  interaction energy}, has no definite sign, since
\[
\int_{|u|>1} |u(\bx,t)|^2\ln(|u(\bx,t)|^2)d\bx\ge 0,\quad
\text{while}\quad \int_{|u|<1} |u(\bx,t)|^2\ln(|u(\bx,t)|^2)d\bx\le 0.
\]
Therefore, it is not obvious to guess which sign of $\lambda$ leads to
which type of dynamics. In the case $\lambda<0$, no solution is
dispersive, as proven in \cite{cazenave1983}. This is reminiscent of
the \emph{focusing} nonlinear Schr\"odinger equation, where, however,
small solutions are always dispersive. In addition, unlike what
happens in the case of the  nonlinear Schr\"odinger equation, every
solution is global in time: there is no such thing as finite time blow
up for LogSE. 
In the case $\lambda<0$, stationary
solutions are available, called \emph{Gaussons} (as noticed in \cite{BiMy79}; see below), which
turn out to be orbitally stable, but not
stable in the usual sense of Lyapunov
\cite{cazenave1983,cazenave1982} (see also \cite{Ar16} for another
proof, and  \cite{ squassina2015,
  tanaka2017} for other particular solutions). In the case
$\lambda>0$, every solution is (global and) dispersive with a non-standard rate (a
logarithmic perturbation of the Schr\"odinger rate), and after a
time-dependent
rescaling related to this dispersion, the large time behavior of the
renormalized solution exhibits a universal Gaussian profile --- a
phenomenon which is rather unique in the context of Hamiltonian
dispersive PDEs, see \cite{CaGa-p}.
\smallbreak

Note that unlike the more standard nonlinear Schr\"odinger equation,
 a large set
of explicit solutions is associated to \eqref{LSE} in the case
$\Omega=\mathbb R^d$. This important feature was
noticed already from the introduction of this model \cite{BiMy76}: if $u_0$ is
Gaussian, $u(\cdot,t)$ is Gaussian for all time, and solving
\eqref{LSE} amounts to solving ordinary differential equations.
For the convenience of the reader,  we briefly recall the formulas
given in \cite{BCST}.
In the one-dimensional case, if
\be
\label{ini-gaus}
  u_0(x) = b_0 e^{-\fl{a_0}{2} x^2+ivx}, \qquad x\in \mathbb{R},
\ee
with $a_0,b_0\in \mathbb C$ satisfying $\alpha_0:=\mathrm{Re}\, a_0>0$ and $v\in\mathbb{R}$ being constants,
then the solution of \eqref{LSE} is given by \cite{Ar16,BCST,CaGa-p}
\be\label{Gaus}
  u(x,t) = \frac{b_0}{\sqrt{r(t)}}e^{i(vx-v^2t)+Y(x-2vt,t)},\qquad x\in{\mathbb R}, \quad t\ge0,
\ee
where
\be
Y(x,t)=-i\phi(t)-\alpha_0
  \frac{x^2}{2r(t)^2}+i\frac{\dot r(t)}{r(t)}\frac{x^2}{4},
  \qquad x\in{\mathbb R}, \quad t\ge0,
\ee
with $\phi:=\phi(t)$ and $r:=r(t)>0$ being the solutions of the ODEs
\begin{align}
\label{phi}
&  \dot \phi= \frac{\alpha_0}{r^2} +\lambda \ln|b_0|^2-\lambda \ln
  r,\quad \phi(0)=0,\\
  \label{r}
& \ddot r = \frac{4\alpha_0^2}{r^3}+\frac{4\lambda \alpha_0}{r},\quad
  r(0)=1,\quad \dot{r}(0)= -2\,\mathrm{Im}\,a_0.
\end{align}
In the multi-dimensional case, one can actually tensorize such
one-dimensional solutions due to the property $\ln |ab|=\ln|a|+\ln|b|$.
In the case $\lambda<0$, if $a_0=-\lambda>0$, then
$r(t)\equiv 1$, which generates a {\sl moving Gausson} when $v\not =0$
and a {\sl static Gausson} when $v=0$ of the LogSE; and
 if $0<a_0\ne-\lambda$,
the function $r$ is  (time)
periodic (in agreement with the absence of dispersive
effects), which generates a {\sl breather} of the LogSE \cite{BCST}.   In the case $\lambda>0$, the
large time behavior of $r$ does not depend on its initial data,
$r(t)\sim 2t\sqrt{\lambda \alpha_0 \ln t}$ as $t\to \infty$ (see
\cite{CaGa-p}). The general dynamics is
rather well understood in the case $\lambda>0$ (see \cite{CaGa-p}), but,
aside from the
explicit Gaussian solutions, the dynamical properties in the case
$\lambda <0$ constitute a vast open problem: is the dynamics
comparable to the one, say, of the cubic nonlinear Schr\"odinger
equation, or is it drastically different? The numerical simulations
presented in this paper tend to suggest that the dynamics associated
to \eqref{LSE} is quite rich, and reveals phenomena absent (or at least
unknown) in the case of the cubic nonlinear Schr\"odinger equation.
\smallbreak

There is a long list of references on numerical approaches for solving
the nonlinear Schr\"odinger (NLS) equation with power-like
nonlinearity
\begin{equation}\label{NLSE1}
  i\p_t u(\bx,t)+\Delta u(\bx,t)=\lambda |u(\bx,t)|^{2\sigma}u(\bx,t),\quad \bx\in \mathbb{R}^d, \quad t>0
\end{equation}
with $\sigma\in \mathbb N$, or with nonlocal Hartree-type nonlinearity
\begin{equation}\label{NLSE2}
    i\p_t u(\bx,t)+\Delta u(\bx,t)=\lambda \left(|\bx|^{-\gamma}\ast
      |u(\bx,t)|^2\right)u(\bx,t),\quad \bx\in \mathbb{R}^d, \quad t>0
\end{equation}
with $\gamma>0$, such as finite
difference method \cite{ABB2013,akrivis1993, chang1999},  finite element
method \cite{akrivis, karakashian},  relaxation method
\cite{besse2004},  and  time-splitting
pseudospectral method \cite{ABB2013, bao2002, bao2003, robinson1993, taha1984,
  markowich1999, thalhammer2009}, etc.
For the error analysis of the
time-splitting or split-step method for the NLS equation, we refer to
\cite{besse, debussche2009, lubich2008}
and the references therein; for the error estimates of the finite difference method, we refer to \cite{guo1986, you1983, bao2013}; and for the error bound of the finite element method, we refer to \cite{akrivis, karakashian, henning2017}.
However, few numerical methods have been proposed and/or analyzed for the LogSE due to the singularity at the origin of the logarithmic nonlinearity.
\smallbreak

In order to avoid (numerical) blow-up and/or to suppress round-off error due to the logarithmic nonlinearity in the LogSE, a regularized logarithmic Schr\"odinger equation (RLogSE)
with a small regularized parameter $0<\ep\ll1$ was
introduced  in \cite{BCST} as
\be\label{RLSE}
\left\{
\begin{aligned}
&i\p_t u^\ep(\bx,t)+\Delta u^\ep(\bx,t)=\lambda u^\ep(\bx,t)\,\ln(\ep+|u^\ep(\bx,t)|)^2,\quad \bx\in \Omega, \quad t>0,\\
&u^\ep(\bx,0)=u_0(\bx),\quad \bx\in \Omega.
\end{aligned}
\right.
\ee
For any fixed $0<\ep\ll1$, the nonlinearity is now locally Lipschitz
continuous (the singularity at the origin disappears).
Remarkably enough, \eqref{RLSE} enjoys similar conservations as the
original model \eqref{LSE}, i.e., the mass \eqref{mass} and momentum \eqref{mom} as
well as the regularized energy defined as
\begin{align}\label{energyr}
&  E^\ep\left(u^\ep(\cdot,t)\right) = \int_{\Omega}\big[|\nabla u^\ep(\bx,t)|^2+2\lambda
  \ep |u^\ep(\bx,t)|+\lambda
  |u^\ep(\bx,t)|^2\ln(\ep+|u^\ep(\bx,t)|)^2\big.\nonumber\\
&\qquad\qquad\qquad\qquad\big.-\lambda \ep^2
\ln\left(1+|u^\ep(\bx,t)|/\ep\right)^2\big]d\bx\equiv E^\ep (u_0),\quad t\ge0.
\end{align}
The regularized model \eqref{RLSE} was proven to approximate
the LogSE \eqref{LSE} linearly in $\ep$ for bounded $\Omega$, i.e.,
\[
\sup_{t\in [0,T]}\|u^\ep(t) -u(t)\|_{L^2(\Omega)}=O(\ep),\quad \forall\ 
T>0,
\]
 and with an error
$O(\ep^{\fl{4}{4+d}})$ in the case of the whole space $\Omega=\mathbb
R^d$, provided that the first two momenta of $u_0$ belong to
$L^2(\mathbb R^d)$
\cite{BCST}. In addition, $E^\ep(u^\ep)=E(u) +O(\ep)$.
Therefore, it is sensible to analyze various numerical methods
associated to \eqref{RLSE}, provided that we have as precise as
possible a control of the dependence of the various constants upon
$\ep$. Then by using the triangle inequality, we can obtain error estimates
of different numerical methods for the LogSE \eqref{LSE}.


Very recently, a semi-implicit finite
difference method was proposed and analyzed for
\eqref{RLSE} and thus for \eqref{LSE} \cite{BCST}.
As we know, there are many efficient and accurate numerical methods
for the nonlinear Schr\"{o}dinger equation \eqref{NLSE1}
such as the  time-splitting spectral method
\cite{bao2002,bao2003,carles2013,carles2017,besse, gauckler, ignat, lubich2008, faou2011} and the conservative Crank-Nicolson
finite difference (CNFD) method \cite{bao2013,BTX13}.
The main aim of this paper is to present and analyze
the  time-splitting spectral method
 and  the CNFD method \eqref{RLSE} and thus for \eqref{LSE}.
We can establish  rigorous
error estimates for the Lie-Trotter splitting
under a much weaker constraint on the time step $\tau$,
$\tau\lesssim 1/(\ln(\ep^{-1}))^2$, instead of
$\tau\lesssim
\sqrt\ep e^{-CT|\ln(\ep)|^2}$ for some $C$ independent of $\ep$,
which is needed for the semi-implicit finite difference method
\cite{BCST}.

The rest of the paper is organized as follows. In Section~\ref{sec:lie}, we
propose a regularized Lie-Trotter splitting method and establish
rigorous error estimates. In Section~\ref{sec:CNFD}, a conservative
finite difference method is adapted to
the RLogSE. Numerical tests are reported for both
methods in term of accuracy under different regularities of the initial data. 
In addition, the splitting method is applied to study the long
time dynamics in Section~\ref{sec:long-time} and some very interesting and
complicated dynamical phenomena are presented to demonstrate
the rich dynamics of the Logarithmic Schr\"odinger equation including
interactions of Gaussons. Throughout the paper, $C$
represents a generic constant independent of $\ep$, $\tau$ and the
function $u$, and $C(c)$ means that $C$ depends on $c$.

\section{A regularized Lie-Trotter splitting method}
\label{sec:lie}
In this section, we study the approximation property of a semi-discretization in time
for solving the regularized model \eqref{RLSE}
 for $d=1, 2, 3$.
The numerical integrator we consider is a Lie-Trotter
splitting method \cite{mclachlan, descombes, gauckler}.  For simplicity of notation, we set
$\lambda=1$.

\subsection{Lie-Trotter splitting method} The operator splitting methods for the time integration of \eqref{RLSE} are based on the splitting
\[\p_t u^\ep=A(u^\ep)+B(u^\ep),\]
where
\[A(u^\ep)=i\Delta u^\ep,\quad B(u^\ep)=-i \varphi^\ep(u^\ep),\quad
  \varphi^\ep(z)=z\ln(\ep+|z|)^2,\]
and the solution of the subproblems
\be\label{lp}
\left\{
\begin{aligned}
&i\p_t v(\bx,t)=-\Delta v(\bx,t),\quad \bx\in\Omega,\quad t>0,\\
&v(\bx,0)=v_0(\bx),
\end{aligned}\right.
\ee
\be\label{nlp}
\left\{
\begin{aligned}
&i\p_t \og(\bx,t)=\varphi^\ep(\og(\bx,t)),\quad \bx\in\Omega,\quad t>0,\\
&\og(\bx,0)=\og_0(\bx),
\end{aligned}\right.
\ee
where $\Omega=\mathbb{R}^d$ or $\Omega\subset \mathbb{R}^d$ is a bounded domain with homogeneous Dirichlet or periodic boundary condition on the boundary.
The associated evolution operators are given by
\be\label{ABs}
v(\cdot,t)=\Phi_A^t(v_0)=e^{it\Delta}v_0,\quad
\og(\cdot,t)=\Phi_B^t(\og_0)=\og_0e^{-it\ln(\ep+|\og_0|)^2},\quad t\ge0.
\ee
Regarding the exact flow $\Phi_A^t$ and $\Phi_B^t$, we have the following properties.
\begin{lemma}
For \eqref{lp}, we have the isometry relation
\be\label{Ap}
\|\Phi_A^t(v_0)\|_{H^s(\Omega)}=\|v_0\|_{H^s(\Omega)},\quad s\in \mathbb{N},\quad t\ge 0.
\ee
For \eqref{nlp}, if $\og_0\in H^2(\Omega)$, then
\be\label{Bp}
\begin{split}
&\|\Phi_B^t(\og_0)\|_{L^2(\Omega)}=\|\og_0\|_{L^2(\Omega)},\quad
\|\nabla \Phi_B^t(\og_0)\|_{L^2(\Omega)}\le (1+2t)\|\nabla \og_0\|_{L^2(\Omega)},\\
&\|\Phi_B^t(\og_0)\|_{H^2(\Omega)}\le C(\|\og_0\|_{H^2(\Omega)})(1+t+t^2)/\ep.
\end{split}
\ee
\end{lemma}
\emph{Proof.} By direct calculation, we get
\[\nabla \Phi_B^t(\omega_0)=e^{-it\ln(\ep+|\omega_0|)^2}\left[\nabla
\omega_0-\fl{2it\omega_0}{\ep+|\omega_0|}\nabla |\omega_0|\right], \,\,\, \mathrm{with} \,\,\, \nabla |\omega_0|=\fl{\omega_0\nabla
\overline{\omega_0}+\overline{\omega_0}\nabla \omega_0}{2|\omega_0|},\]
which gives immediately, since $\left\lvert \nabla|\omega_0|\right\rvert\le |\nabla \omega_0|$,
\[\|\nabla \Phi_B^t(\og_0)\|_{L^2(\Omega)}\le (1+2t)\|\nabla \og_0\|_{L^2(\Omega)}.\]
Noticing that
\begin{align*}
\fl{\p^2\Phi_B^t(\og_0)}{\p x_k \p x_j}&=\fl{\p^2\og_0}{\p x_k \p x_j}-\fl{2it}{\ep+|\og_0|}\left(\fl{\p\og_0}{\p x_k}\fl{\p |\og_0|}{\p x_j} +\fl{\p\og_0}{\p x_j}\fl{\p |\og_0|}{\p x_k}+\og_0\fl{\p^2|\og_0|}{\p x_k \p x_j}\right)\\
&\quad+\fl{2it-4t^2}{(\ep+|\og_0|)^2} \fl{\p |\og_0|}{\p x_j}\fl{\p |\og_0|}{\p x_k}\og_0,
\end{align*}
where
\begin{align*}
\left|\fl{\p^2 |\og_0|}{\p x_k\p x_j}\right| &=\left|\fl{1}{|\og_0|}\mathrm{Re}\left(\og_0\fl{\p^2\overline{\og_0}}{\p x_k\p x_j}+\fl{\p \og_0}{\p x_j}\fl {\p\overline{\og_0}}{\p x_k}\right)-\fl{1}{|\og_0|^2}\fl{\p |\og_0|}{\p x_k}\mathrm{Re}\left(\og_0 \fl{\p \overline{\og_0}}{\p x_j}\right)\right|\\
&\le \left|\fl{\p^2\og_0}{\p x_k\p x_j}\right|+\fl{2}{|\og_0|}\left|\fl{\p\og_0}{\p x_j}\right|\left|\fl{\p\og_0}{\p x_k}\right|,
\end{align*}
this yields that
\begin{align*}
\left|\fl{\p^2\Phi_B^t(\og_0)}{\p x_k \p x_j}\right|
&\le (1+2t)\left|\fl{\p^2\og_0}{\p x_k \p x_j}\right|+\fl{10t+4t^2}{\ep+|\og_0|}\left|\fl{\p\og_0}{\p x_j}\right|\left|\fl{\p\og_0}{\p x_k}\right|\\
&\le (1+2t)\left|\fl{\p^2\og_0}{\p x_k \p x_j}\right|+\fl{5t+2t^2}{\ep}\left(\left|\fl{\p\og_0}{\p x_j}\right|^2+\left|\fl{\p\og_0}{\p x_k}\right|^2\right).
\end{align*}
Thus
\begin{align*}
\left\|\fl{\p^2\Phi_B^t(\og_0)}{\p x_k \p x_j}\right\|_{L^2(\Omega)}
&\le (1+2t)\left\|\fl{\p^2\og_0}{\p x_k \p x_j}\right\|_{L^2(\Omega)}+\fl{5t+2t^2}{\ep}\left(\left\|\fl{\p\og_0}{\p x_j}\right\|^2_{L^4(\Omega)}+\left\|\fl{\p\og_0}{\p x_k}\right\|^2_{L^4(\Omega)}\right),
\end{align*}
which completes the proof by recalling that
$H^2(\Omega)\hookrightarrow W^{1,4} (\Omega)$, since $d\le 3$.
\hfill $\square$ \bigskip

We consider the Lie-Trotter splitting
\be\label{LT}
u^{\ep, n+1}=\Phi^\tau(u^{\ep, n})=\Phi_A^\tau(\Phi_B^\tau(u^{\ep, n})),\quad u^{\ep,0}=u_0,
\ee
for a time step $\tau>0$. Thus for $u^{\ep,n}\in H^1(\Omega)$, it follows from the isometry property that the splitting method conserves the mass
\[\|u^{\ep,n+1}\|_{L^2(\Omega)}=\|u^{\ep,n}\|_{L^2(\Omega)}\equiv \|u^{\ep,0}\|_{L^2(\Omega)}= \|u_0\|_{L^2(\Omega)} ,\quad n\ge0,\]
and furthermore \eqref{Bp} gives $u^{\ep,n}\in H^1(\Omega)$ with
\be\label{unp}
\|u^{\ep,n+1}\|_{H^1(\Omega)}\le (1+2\tau)\| u^{\ep,n}\|_{H^1(\Omega)}\le e^{2n\tau}\| u_0\|_{H^1(\Omega)}, \quad n\ge0.
\ee

\subsection{Error estimates}
In this section, we carry out the error analysis of the Lie-Trotter splitting \eqref{LT}.
\begin{theorem}\label{thmlt}
Let $T>0$. Assume that the solution of \eqref{RLSE} satisfies
$u^\ep\in  L^\infty(0,T;H^1(\Omega))$ for $d=1$ or
$u^\ep\in L^\infty(0,T;H^2(\Omega))$  for $d=2, 3$. There exists
$\ep_0>0$ depending on $\|u^\ep\|_{L^\infty(0,T; H^1(\Omega))}$ for $d=1$,
and $\|u^\ep\|_{L^\infty(0,T; H^2(\Omega))}$ for $d=2, 3$, such that
when $\ep\le \ep_0$ and $n\tau\le T$, we have
\[\|u^{\ep,n}-u^\ep(t_n)\|_{L^2(\Omega)}\le C\left(T,
  \|u^\ep\|_{L^\infty([0,T];H^1(\Omega))}\right)\ln(\ep^{-1})\tau^{1/2},\]
where $C(\cdot,\cdot)$ is independent of $\ep>0$.
\end{theorem}

\begin{remark}
  As established in \cite[Theorem~2.2]{BCST}, for an arbitrarily large
fixed $T>0$, the above assumptions are satisfied as soon as $u_0\in
H^1(\Omega)$ if $d=1$, and $u_0\in H^2(\Omega)$ if $d=2,3$. More
precisely, for $k=1,2$,
\begin{equation*}
\sup_{t\in [0,T]}  \|u^\ep(t)\|_{H^k(\Omega)}\le C\left(\|u_0\|_{H^k(\Omega)}\right),
\end{equation*}
for a constant $C$ depending on $\|u_0\|_{H^k(\Omega)}$, \emph{but not
  on $0<\ep\ll1$}.
\end{remark}
The above result provides a convergence of order $1/2$, with a
constant mildly singular in $\ep$ (a logarithm). We will see in
Remark~\ref{rem:better-order} that it is possible to establish
the convergence of order $1$, which is rather natural for a
Lie-Trotter scheme, but the price to pay is a much more singular
dependence with respect to $\ep$. And numerical
  observations show that the convergence rate degenerates to $1/2$ for
  solutions belonging to $H^1$ in 1D (cf. first figure in Fig. \ref{fig:case2-L2}).
\smallbreak

Before giving the proof, we introduce the
following lemma, which is a variant of
\cite[Lemma~9.3.5]{CazCourant}, established in \cite{BCST}.
\begin{lemma}\label{pre}
For any $z_1$, $z_2\in\mathbb{C}$, we have
\[\left|\mathrm{Im}\left((\varphi^\ep(z_1)-\varphi^\ep(z_2))
(\overline{z_1}-\overline{z_2})\right)\right|\le 2|z_1-z_2|^2.\]
\end{lemma}

\begin{lemma}[Local error]\label{local}
Assume $u_0\in H^1(\Omega)$ for $d=1$ or $u_0\in H^2(\Omega)$ for
$d=2, 3$. Let $\Psi^t$ denote the exact flow of \eqref{RLSE}, i.e.,
$u^\ep(t)=\Psi^t(u_0)$. Then for $\tau\le 1$, there exists $\ep_0>0$
depending on $\|u_0\|_{H^1(\Omega)}$ for $d=1$ and
$\|u_0\|_{H^2(\Omega)}$ for $d=2, 3$ such that when $\ep\le \ep_0$, we
have
\[\|\Psi^\tau(u_0)-\Phi^\tau(u_0)\|_{L^2(\Omega)}\le C \|u_0\|_{H^1(\Omega)}\ln(\ep^{-1})\tau^{3/2}.\]
\end{lemma}
\emph{Proof.}
It can be obtained from the definition that
\begin{align*}
&i\p_t(\Psi^tu_0)+\Delta (\Psi^t u_0)=\varphi^\ep(\Psi^t u_0),\\
&i\p_t(\Phi^tu_0)+\Delta (\Phi^t u_0)=\Phi_A^t(\varphi^\ep(\Phi_B^tu_0)).
\end{align*}
Denoting $\mathcal E^tu_0=\Psi^tu_0-\Phi^tu_0$, we have
\be\label{erq}
i\p_t (\mathcal E^tu_0)+\Delta (\mathcal E^tu_0)=\varphi^\ep(\Psi^tu_0)-\Phi_A^t(\varphi^\ep(\Phi_B^tu_0)).
\ee
Denote by
\begin{equation*}
  (f,g)=\int_\Omega f \overline g \,d\bx
\end{equation*}
the $L^2$ inner product.
Multiplying \eqref{erq} by $\overline{\mathcal E^t u_0}$, integrating in space
and taking the imaginary part, the term corresponding to the Laplacian
vanishes (in the case with a boundary, we use the Dirichlet
boundary condition or periodic boundary condition), and  Lemma~\ref{pre} yields
\begin{align*}
\fl{1}{2}\fl{d}{dt}\|\mathcal E^tu_0\|_{L^2(\Omega)}^2&=\mathrm{Im}\left(
\varphi^\ep(\Psi^tu_0)-\Phi_A^t(\varphi^\ep(\Phi_B^tu_0)), \mathcal E^tu_0\right)\\
&=\mathrm{Im}\left[
\left(\varphi^\ep(\Psi^tu_0)-\varphi^\ep(\Phi^tu_0), \mathcal E^tu_0\right)+\left(\varphi^\ep(\Phi^tu_0)-\Phi_A^t(\varphi^\ep(\Phi_B^tu_0)), \mathcal E^tu_0\right)\right]\\
&\le 2\|\mathcal E^tu_0\|_{L^2(\Omega)}^2+\|
\varphi^\ep(\Phi^tu_0)-\Phi_A^t(\varphi^\ep(\Phi_B^tu_0))\|_{L^2(\Omega)}\|\mathcal
  E^tu_0\|_{L^2(\Omega)},
\end{align*}
which implies
\begin{align}
\fl{d}{dt}\|\mathcal E^tu_0\|_{L^2(\Omega)}&\le 2\|\mathcal E^tu_0\|_{L^2(\Omega)}+\|
\varphi^\ep(\Phi^tu_0)-\Phi_A^t(\varphi^\ep(\Phi_B^tu_0))\|_{L^2(\Omega)}\nn\\
&\le 2\|\mathcal E^tu_0\|_{L^2(\Omega)}+\|
\varphi^\ep(\Phi^tu_0)-\varphi^\ep(\Phi_B^tu_0)\|_{L^2(\Omega)}\nn\\
&\quad+\|\varphi^\ep(\Phi_B^tu_0)-\Phi_A^t(\varphi^\ep(\Phi_B^tu_0))\|_{L^2(\Omega)}.
\label{tmp1}
\end{align}
If $\Omega=\mathbb{R}^d$, for any $f\in H^1(\Omega)$, we compute
\begin{align}
\left\|f-\Phi_A^tf\right\|_{L^2(\Omega)}&=
\left\|(1-e^{-it|\xi|^2})\widehat{f}(\xi)\right\|_{L^2(\Omega)}=2\left\|\sin\left(t|\xi|^2/2\right)
\widehat{f}(\xi)\right\|_{L^2(\Omega)}\nn\\
&\le \sqrt{2t}\left\||\xi|\widehat{f}(\xi)\right\|_{L^2(\Omega)}\le \sqrt{2t}\,\|f\|_{H^1(\Omega)}.\label{ld}
\end{align}
When $\Omega$ is a bounded domain, \eqref{ld} can be similarly
obtained via the discrete Fourier transform. At this stage, one could
argue that the above estimate can be improved, by removing the square
root, the price to pay being an $H^2$-norm instead of an
$H^1$-norm. It turns out that this approach eventually yields an extra $1/\ep$
factor in the error estimate, which we want to avoid here; see
Remark~\ref{rem:better-order}.
Recalling that
$\varphi^\ep(\Phi_B^tu_0)=u_0\ln(\ep+|u_0|)^2e^{-it\ln(\ep+|u_0|)^2}$, we compute
\[\nabla \varphi^\ep(\Phi_B^tu_0)=e^{-it\ln(\ep+|u_0|)^2}\left[
\nabla u_0\ln(\ep+|u_0|)^2+\fl{2u_0\nabla |u_0|}{\ep+|u_0|}\left(1-it\ln(\ep+|u_0|)^2\right)\right],\]
which yields
\[|\nabla \varphi^\ep(\Phi_B^tu_0)|\le 2|\nabla u_0|\left(1+(1+2t)\max\left\{|\ln(\ep)|, \left|\ln\left(\ep+\|u_0\|_{L^\infty(\Omega)}\right)\right|\right\}\right).\]
This implies
\[\|\varphi^\ep(\Phi_B^tu_0)\|_{H^1(\Omega)}\le C\ln(\ep^{-1})(1+t)\|u_0\|_{H^1(\Omega)},\]
when $\ep\lesssim 1/\|u_0\|_{L^\infty(\Omega)}$.
To check this property, we recall that
 $ H^1(\Omega)\hookrightarrow L^\infty(\Omega)$ for $d=1$, and $
 H^2(\Omega)\hookrightarrow L^\infty(\Omega)$ for $d=2, 3$.
It follows from \eqref{Ap} and \eqref{Bp} that
\[\|\Phi^tu_0\|_{H^2(\Omega)}=\|\Phi_B^tu_0\|_{H^2(\Omega)}\le C(\|u_0\|_{H^2(\Omega)})(1+t+t^2)/\ep.\]
Hence by Sobolev imbedding, when $t\le 1$,
we have $\|\Phi^tu_0\|_{L^\infty(\Omega)}$,
$\|\Phi_B^tu_0\|_{L^\infty(\Omega)}\le
\fl{C(\|u_0\|_{H^2(\Omega)})}{\ep}$ for $d=2, 3$ and
$\|\Phi^tu_0\|_{L^\infty(\Omega)}$,
$\|\Phi_B^tu_0\|_{L^\infty(\Omega)}\le C\|u_0\|_{H^1(\Omega)}$ for
$d=1$. In particular, the property $\ep\lesssim
1/\|u_0\|_{L^\infty(\Omega)}$ is always satisfied.

 Hence  we have
\be\label{tmp2}
\begin{aligned}
\|\varphi^\ep(\Phi_B^tu_0)-\Phi_A^t(\varphi^\ep(\Phi_B^tu_0))\|_{L^2(\Omega)}&\le \sqrt{2t}
\|\varphi^\ep(\Phi_B^tu_0)\|_{H^1(\Omega)}\\
&\le C\ln(\ep^{-1})\sqrt{t}(1+t)\|u_0\|_{H^1(\Omega)},
\end{aligned}
\ee
for $\ep\le \ep_1$, with $\ep_1$ depending on $\|u_0\|_{H^1(\Omega)}$ for $d=1$ and $\|u_0\|_{H^2(\Omega)}$ for $d=2, 3$. Next we claim that for $v(\bx)$, $w(\bx)$ satisfying $|v(\bx)|, |w(\bx)|\le C_1/\ep$, it can be established that
\[|\varphi^\ep(v(\bx))-\varphi^\ep(w(\bx))|\le C\ln(\ep^{-1})|v(\bx)-w(\bx)|,\]
when $\ep$ is sufficiently small.
Assuming, for example, $0\le|w(\bx)|\le|v(\bx)|$, then
\begin{align}
|\varphi^\ep(v(\bx))-\varphi^\ep(w(\bx))|&=2\left|(v(\bx)-w(\bx))\ln(\ep+|v(\bx)|)
+w(\bx)\ln\left(1+\fl{|v(\bx)|-|w(\bx)|}{\ep+|w(\bx)|}\right)\right|\nn\\
&\le 2|v(\bx)-w(\bx)||\ln(\ep+|v(\bx)|)|+\fl{2|w(\bx)|}{\ep+|w(\bx)|}|v(\bx)-w(\bx)|\nn\\
&\le C\ln(\ep^{-1})|v(\bx)-w(\bx)|,\label{vp}
\end{align}
when $\ep\le \min(C_1,1)$.
Thus, we obtain
\begin{align}
\|\varphi^\ep(\Phi^tu_0)-\varphi^\ep(\Phi_B^tu_0)\|_{L^2(\Omega)}
&\le C\ln(\ep^{-1})
\|\Phi^tu_0-\Phi_B^tu_0\|_{L^2(\Omega)}\nn\\
&\le C\ln(\ep^{-1})\sqrt{2t}\|\Phi_B^tu_0\|_{H^1(\Omega)}\nn\\
&\le C\ln(\ep^{-1})\sqrt{2t}(1+2t)\|u_0\|_{H^1(\Omega)},\label{tmp3}
\end{align}
when $\ep\le \ep_2$ with $\ep_2$ depending on $\|u_0\|_{H^2(\Omega)}$ for $d=2, 3$ and $\|u_0\|_{H^1(\Omega)}$ for $d=1$.
Combining \eqref{tmp1}, \eqref{tmp2} and \eqref{tmp3}, we get
\[\fl{d}{dt}\|\mathcal E^tu_0\|_{L^2(\Omega)}\le 2\|\mathcal E^t u_0\|_{L^2(\Omega)}+C\ln(\ep^{-1})\sqrt{t}(1+t)\|u_0\|_{H^1(\Omega)}.\]
Applying the Gronwall's inequality, when $\tau\le 1$, we have
\[\|\mathcal E^\tau u_0\|_{L^2(\Omega)}\le C\ln(\ep^{-1})\sqrt{\tau}(1+\tau)(e^{2\tau}-1)\|u_0\|_{H^1(\Omega)}\le C\ln(\ep^{-1})\tau^{3/2}\|u_0\|_{H^1(\Omega)},\]
when $\ep\le \ep_0=\min\{\ep_1, \ep_2\}$ depending on
$\|u_0\|_{H^1(\Omega)}$ for $d=1$ and $\|u_0\|_{H^2(\Omega)}$ for
$d=2, 3$.
\hfill $\square$ \bigskip

Furthermore, we also need the following lemma concerning on the stability property.
\begin{lemma}[Stability]\label{stab}
Let $f$, $g\in L^2(\Omega)$. Then for all $\tau>0$, we have
\[\|\Phi^\tau(f)-\Phi^\tau(g)\|_{L^2(\Omega)}\le (1+2\tau)\|f-g\|_{L^2(\Omega)}.\]
\end{lemma}
\emph{Proof.}
Noticing that $\Phi_A^\tau$ is a linear isometry on $H^s(\Omega)$, we obtain that
\[
\|\Phi^\tau(f)-\Phi^\tau(g)\|_{L^2(\Omega)}=\|\Phi_B^\tau(f)-\Phi_B^\tau(g)\|_{L^2(\Omega)}.\]
We claim that for any $\bx\in\Omega$, we have
\[|\Phi_B^\tau(f)(\bx)-\Phi_B^\tau(g)(\bx)|\le (1+2\tau)|f(\bx)-g(\bx)|.\]
Assuming, for example, $|f(\bx)|\le|g(\bx)|$, then by inserting a term
$f(\bx)e^{-i\tau\ln(\ep+|g(\bx)|)^2}$, we can get that
\begin{align*}
|\Phi_B^\tau(f)(\bx)-\Phi_B^\tau(g)(\bx)|&=\left|f(\bx)e^{-i\tau\ln(\ep+|f(\bx)|)^2}-
g(\bx)e^{-i\tau\ln(\ep+|g(\bx)|)^2}\right|\\
&=\Big|f(\bx)-g(\bx)+f(\bx)\Big(e^{2i\tau\ln(\fl{\ep+|g(\bx)|}{\ep+|f(\bx)|})}-1\Big)\Big|\\
&\le|f(\bx)-g(\bx)|+2|f(\bx)|\Big|\sin\Big(\tau\ln\Big(\fl{\ep+|g(\bx)|}{\ep+|f(\bx)|}\Big)\Big)\Big|\\
&\le |f(\bx)-g(\bx)|+2\tau|f(\bx)|\ln\Big(1+\fl{|g(\bx)|-|f(\bx)|}{\ep+|f(\bx)|}\Big)\\
&\le(1+2\tau)|f(\bx)-g(\bx)|.
\end{align*}
When $|f(\bx)|\ge |g(\bx)|$,  the same inequality is obtained by
exchanging $f$ and $g$ in the above computation. Thus the proof is completed.
\hfill $\square$

\begin{proof}[Proof of Theorem~\ref{thmlt}] It can be easily concluded
  from \eqref{unp} that $u^{\ep,n}\in H^1(\Omega)$ and
  $\|u^{\ep,n}\|_{H^1(\Omega)}\le e^{2T} \|u_0\|_{H^1(\Omega)}$.
The triangle inequality, \eqref{unp}, Lemmas~\ref{local} and \ref{stab} yield
\begin{align*}
&\hspace{-3mm}\|u^{\ep,n}-u^\ep(t_{n})\|_{L^2(\Omega)}\\
&=\|\Phi^\tau(u^{\ep,n-1})-\Psi^\tau(u^\ep(t_{n-1}))\|_{L^2(\Omega)}\\
&\le\|\Phi^\tau(u^{\ep,n-1})-\Phi^\tau(u^\ep(t_{n-1}))\|_{L^2(\Omega)}+
\|\Phi^\tau(u^\ep(t_{n-1}))-
\Psi^\tau(u^\ep(t_{n-1}))\|_{L^2(\Omega)}\\
&\le (1+2\tau)\|u^{\ep,n-1}-u^\ep(t_{n-1})\|_{L^2(\Omega)}+\|\mathcal E^\tau(u^\ep(t_{n-1}))\|_{L^2(\Omega)}\\
&\le (1+2\tau)\|u^{\ep,n-1}-u^\ep(t_{n-1})\|_{L^2(\Omega)}+C\|u^\ep(t_{n-1})\|_{H^1(\Omega)}\ln(\ep^{-1})\tau^{3/2}\\
&\le C\|u^\ep\|_{L^\infty(0,T; H^1(\Omega))}\ln(\ep^{-1})\tau^{3/2}(1+1+2\tau)+(1+2\tau)^2\|u^{\ep,n-2}-u^\ep(t_{n-2})\|_{L^2(\Omega)}\\
&\le\cdots\\
&\le C\|u^\ep\|_{L^\infty(0,T; H^1(\Omega))}\ln(\ep^{-1})\tau^{3/2}\left[1+(1+2\tau)+\cdots+(1+2\tau)^{n-1}\right]\\
&\quad+(1+2\tau)^{n}\|u^{\ep,0}-u_0\|_{L^2(\Omega)}\\
&\le C\|u^\ep\|_{L^\infty(0,T; H^1(\Omega))}\ln(\ep^{-1})\tau^{1/2}(1+2\tau)^{n}\\
&\le C\|u^\ep\|_{L^\infty(0,T; H^1(\Omega))} e^{2T}\ln(\ep^{-1})\tau^{1/2},
\end{align*}
where we have used $u^{\ep,0}=u_0$, see \eqref{LT}. This completes the proof.
\end{proof}

\begin{remark}\label{rem:better-order}
Noticing that for any $f\in H^2(\Omega)$ and $t\ge 0$, we have \cite{besse}
\[\|f-\Phi_A^t f\|_{L^2(\Omega)}\le t\|f\|_{H^2(\Omega)},\]
and by tedious calculation, one can get that
\[\|\varphi^\ep(f)\|_{H^2(\Omega)}\le C \ln(\ep^{-1})\|f\|_{H^2(\Omega)}+C\ep^{-1}\|\nabla f\|_{L^4(\Omega)}^2,\]
it can be concluded from \eqref{tmp1} that
\begin{align*}
\fl{d}{dt}\|\mathcal E^tu_0\|_{L^2(\Omega)}&\le
2\|\mathcal E^tu_0\|_{L^2(\Omega)}+\|\varphi^\ep(\Phi^tu_0)-\varphi^\ep(\Phi_B^tu_0)\|_{L^2(\Omega)}\nn\\
&\quad+\|\varphi^\ep(\Phi_B^tu_0)-\Phi_A^t(\varphi^\ep(\Phi_B^tu_0))\|_{L^2(\Omega)}\nn\\
&\le 2\|\mathcal E^tu_0\|_{L^2(\Omega)}+C \ln(\ep^{-1})\|\Phi^t u_0-\Phi_B^t u_0\|_{L^2(\Omega)}+t
\|\varphi^\ep(\Phi_B^tu_0)\|_{H^2(\Omega)}\\
&\le 2\|\mathcal E^tu_0\|_{L^2(\Omega)}+C t\ln(\ep^{-1})\|\Phi_B^t u_0\|_{H^2(\Omega)}+C t\ep^{-1}\|\nabla\Phi_B^tu_0 \|_{L^4(\Omega)}^2\\
&\le 2\|\mathcal E^tu_0\|_{L^2(\Omega)}+C(\|u_0\|_{H^2(\Omega)}) \ep^{-1}\ln(\ep^{-1})t+C t\ep^{-1}\|\nabla u_0 \|_{L^4(\Omega)}^2\\
&\le 2\|\mathcal E^tu_0\|_{L^2(\Omega)}+C(\|u_0\|_{H^2(\Omega)}) \ep^{-1}\ln(\ep^{-1})t,
\end{align*}
when $\ep$ is sufficiently small. Hence by using similar arguments, we can get that
\[\|u^{\ep,n}-u^\ep(t_n)\|_{L^2(\Omega)}\le C(T, \|u^\ep\|_{L^\infty([0,T];H^2(\Omega))})\ep^{-1}\ln(\ep^{-1})\tau,\]
when $\ep\le c$, where $c>0$ depends on $\|u^\ep\|_{L^\infty(0,T;
  H^1(\Omega))}$ for $d=1$ and $\|u^\ep\|_{L^\infty(0,T;
  H^2(\Omega))}$ for $d=2, 3$. This approach yields a better
convergence rate for fixed $\ep>0$, but with a terrible dependence
upon $\ep$, as $\ep$ is intended to go to zero to recover the solution
of \eqref{LSE}. Following Theorem~\ref{thmlt}, we get a reasonable
numerical approximation of the solution $u$ to \eqref{LSE} provided
that $\tau\ll 1/(\ln(\ep^{-1}))^2$ and $\ep\ll 1$ (for $u^\ep$ to
approximate $u$), while the above
estimate requires the stronger condition $\tau\ll \ep/\ln(\ep^{-1})$,
and still $\ep\ll 1$.
\end{remark}

\begin{remark}
For the other Lie-Trotter splitting
\[u^{\ep,n+1}= \Phi_B^{\tau}\left(\Phi_A^\tau(u^{\ep,n})\right)=\Phi_A^\tau(u^{\ep,n})-i\int_0^\tau \varphi^\ep\left(\Phi_B^s \Phi_A^\tau (u^{\ep,n})\right)ds,\]
unfortunately, we cannot get the similar error estimate as in
Theorem~\ref{thmlt}, since the proof involves $(\varphi^\ep)'$ or even
$(\varphi^\ep)''$, which yields negative powers of $\ep$ in  the local
error. In fact, by using the standard
arguments via the Lie commutator as in \cite{lubich2008}, we can get
the error bound
\[\|u^{\ep,n}-u^\ep(t_n)\|_{L^2(\Omega)}\le C(T, \|u^\ep\|_{L^\infty([0,T];H^2(\Omega))})\ep^{-1}\tau,\]
when $\ep\le c$ with $c$ depending on $\|u^\ep\|_{L^\infty(0,T;
  H^2(\Omega))}$.
\end{remark}
\begin{remark}[Strang splitting]
  When considering a Strang splitting,
    \begin{equation}
    \label{ST1}
u^{\ep,n+1}= \Phi_B^{\tau/2}\left(\Phi_A^\tau
    \left(\Phi_B^{\tau/2}(u^{\ep,n})\right)\right),
  \end{equation}
  or
  \begin{equation}
  \label{ST2}
    u^{\ep,n+1}= \Phi_A^{\tau/2}\left(\Phi_B^\tau
    \left(\Phi_A^{\tau/2}(u^{\ep,n})\right)\right),
  \end{equation}
one would expect to face similar singular factors as above. It turns
out that the analysis is even more intricate than expected, and we
could not get any reasonable estimate in that case, that is,
improving Theorem~\ref{thmlt} in terms of order for fixed $\ep$,
without (too much) singularity in $\ep$. This can be
understood as a remain of the singularity of the logarithm at the
origin, yielding too many negative powers of $\ep$ in the case of
\eqref{RLSE}. Strang splitting usually provides
better error estimates by invoking higher regularity which, in our
case, implies extra negative powers of $\ep$.
\end{remark}

\section{A regularized Crank-Nicolson finite difference method}
\label{sec:CNFD}
In this section, we introduce a conservative Crank-Nicolson finite difference (CNFD) method  for solving the regularized model \eqref{RLSE}. For simplicity of
notation, we  only present the numerical method for
the  RLogSE \eqref{RLSE} in 1D, as  extensions to higher dimensions are
straightforward. When $d=1$, we truncate the RLogSE on a bounded
computational interval $\Og=(a,b)$ with periodic boundary
condition (here $|a|$ and $b$ are chosen large enough such that the
truncation error is negligible):
\be\label{RLSE1d}
\left\{
\begin{aligned}
&i\p_t u^\ep(x,t)+\p_{xx} u^\ep(x,t)=\lambda u^\ep(x,t)\,\ln(\ep+|u^\ep(x,t)|)^2,
\quad x \in \Omega, \quad t>0,\\
&u^\ep(x,0)=u_0(x),\quad x\in\Omega;
\quad u^\ep(a,t)=u^\ep(b,t), \quad u_x^\ep(a,t)=u_x^\ep(b,t),  \quad t\ge0,
\end{aligned}
\right.
\ee

Choose a mesh size $h:=\Delta x=(b-a)/M$ with $M$ being a positive integer and
a time step $\tau:=\Delta t>0$ and denote the grid points and time steps as
$$x_j:=a+jh,\quad j=0,1,\cdots,M;\quad t_k:=k\tau,\quad k=0,1,2,\dots$$
Let $u^{\ep,k}_j$  be the approximation of $u^{\ep}(x_j,t_k)$, and
denote $u^{\ep,k}=(u^{\ep,k}_0, u^{\ep,k}_1, \ldots, u^{\ep,k}_M)^T\in
\mathbb{C}^{M+1}$ as
the numerical solution vectors at $t=t_k$. Define the standard finite difference operators
\[
\delta_t^+ u_j^k=\fl{u_j^{k+1}-u_j^{k}}{\tau},\quad
\delta_x^+ u_j^k=\fl{u_{j+1}^k-u_j^k}{h},\quad
\delta_x^2u_j^k=\fl{u_{j+1}^k-2u_j^k+u_{j-1}^k}{h^2}.
\]
Denote
$$X_M=\left\{v=\left(v_0,v_1,\ldots,v_M\right)^T\  | \ v_0=v_M, v_{-1}=v_{M-1}, \right\} \subseteq \mathbb{C}^{M+1},$$
equipped with inner products and norms defined as (recall that
$u_0=u_M$ by periodic boundary condition)
\be\label{norm}
\begin{split}
&(u, v)=h\sum\limits_{j=0}^{M-1}u_j \overline{v_j}, \quad \|u\|_{L^2}^2=(u, u),\quad |u|_{H^1}^2=(\delta_x^+u, \delta_x^+u),\\
&\|u\|_{H^1}=\|u\|_{L^2}+|u|_{H^1},\quad \|u\|_{L^\infty}=\sup\limits_{0\le j\le M}|u_j|.
\end{split}
\ee
Then we have for $u$, $v\in X_M$,
\be\label{innpX_M}
(-\delta_x^2 u,v)=(\delta_x^+ u,\delta_x^+ v)=(u,-\delta_x^2 v).
\ee
Following the general CNFD form for the nonlinear Schr\"odinger equation as in \cite{glassey,bao2013}, we can get the CNFD discretization as
\be
\label{CNFD}
\left\{
\begin{aligned}
&i\delta_t^+ u_j^{\ep,k}=-\fl{1}{2}\delta_x^2(u_j^{\ep,k}+u_j^{\ep,k+1})+G_\ep(u_j^{\ep,k+1}, u_j^{\ep,k}),\quad j=0,\cdots,M-1,\\
&u_j^{\ep,0}=u_0(x_j), \quad j=0,\cdots,M; \quad u_0^{\ep,k+1}=u_M^{\ep,k+1},\quad u_{-1}^{\ep,k+1}=u_{M-1}^{\ep,k+1},
\end{aligned}
\right.
\quad k\ge 0.
\ee
Here, $G_\ep(z_1, z_2)$ is defined for $z_1$, $z_2\in\mathbb{C}$ as
\[G_\ep(z_1,z_2):=\int_0^1 f_\ep(\theta|z_1|^2+(1-\theta)|z_2|^2)d\theta\cdot\fl{z_1+z_2}{2}
=\fl{F_\ep(|z_1|^2)-F_\ep(|z_2|^2)}{|z_1|^2-|z_2|^2}\cdot \fl{z_1+z_2}{2},\]
with
\begin{align*}
f_\ep(\rho)&=\lambda\ln(\ep+\sqrt{\rho})^2,\\
F_\ep(\rho)&=\int_0^\rho f_\ep(s)ds=2\lambda(\rho-\ep^2)\ln(\ep+\sqrt{\rho})-\lambda \rho+2\ep\lambda\sqrt{\rho}.
\end{align*}
Then following the analogous arguments of the CNFD method for NLS \cite{bao2013,glassey}, we can get the conservation properties in the discretized level
\begin{align*}
&M_h(u^{\ep,k}):=\|u^{\ep,k}\|_{L^2}^2\equiv M_h(u^{\ep,0}),\\
&E^\ep_h(u^{\ep,k}):=|e^{\ep,k}|_{H^1}^2+h\sum\limits_{j=0}^{M-1} F_\ep(|u_j^{\ep,k}|^2)\equiv
E_h^\ep(u^{\ep,0}).
\end{align*}

\section{Numerical results}

In this section, we first test the order of accuracy of the regularized Lie-Trotter splitting (LTSP) scheme \eqref{LT}, the Strang-splitting (STSP) scheme \eqref{ST1} and the
CNFD  scheme \eqref{CNFD}. Then we apply the Strang-splitting method to investigate some long time dynamics of the LogSE. In practical computation, we impose the periodic boundary condition on $\Og=(a,b)$ for the RLogSE \eqref{RLSE}.

For the Lie-Trotter and Strang splitting methods, we employ the Fourier pseudo-spectral discretization \cite{bao2002,bao2003} for the spatial variable. Let $M$ be a positive even integer and denote $h=(b-a)/M$ and the grid points $x_j=a+jh$ ($0\le j\le M-1$). Denote
by $u^{M,k}$ the discretized solution vector over the grid points $x_j$
($0\le j\le M-1$) at time $t=t_k=k\tau$. Let $\mathcal{F}_M$ and $\mathcal{F}^{-1}_M$ denote the discrete Fourier transform and its inverse, respectively. With this notation, $\Phi_A^\tau(u^{M,k})$ \eqref{ABs} can be obtained by
\[\Phi_A^\tau(u^{M,k})=\mathcal{F}^{-1}_M(e^{-i\tau (\mu^M)^2}\mathcal{F}_M(u^{M,k})),\]
where
\[\mu^M=\fl{2\pi}{b-a}\left[0,1,\cdots,\left(\fl{M}{2}-1\right),-\fl{M}{2},\cdots,-1\right],\]
and the multiplication of two vectors is taken as point-wise. Moreover, $\Phi_B^\tau$ can be directly written in physical space.

\subsection{Accuracy test}
\label{sec:Accuracy}
Here, we fix $\lambda=-1$ and $d=1$. We  compare the   LTSP \eqref{LT}, STSP  \eqref{ST1} and  CNFD \eqref{CNFD} schemes for the following two  initial set-ups:

\medskip

{\it Case I}. We consider the smooth Gaussian-type data \eqref{ini-gaus} as
\begin{equation}
\label{ini-data-case1}
u_0(x)=\sqrt[4]{-\lambda/\pi}e^{ivx+\fl{\lambda}{2}x^2}, \quad x\in\mathbb{R},
\end{equation}
where $v$ is a real constant.
Indeed, with this initial data and $\Og=\mathbb{R}$, $\phi$ and $r$ in \eqref{phi}-\eqref{r} can be obtained explicitly and the LogSE \eqref{LSE} admits a moving Gausson solution \eqref{Gaus} with velocity $v$\cite{BiMy79,CaGa-p}.

\medskip

{\it Case II}. We consider the datum in $H^\vartheta(\Og)$ as
\begin{equation}
\label{ini-data-case2}
u^{M,0}=\fl{u_\vartheta^M}{\|u_\vartheta^M\|},\quad
u_\vartheta^M:=\mathcal{F}_M^{-1}\left(|\mu^M|^{-\vartheta}\mathcal{F}_M(\mathcal{U}^{M})\right),\quad
\big(|\mu^M|^{-\vartheta} \big)_l=\left\{
\begin{array}{ll}
|\mu^M_l|^{-\vartheta}, & {\rm if}\;\;\mu_l^M\ne0,\\[0.5em]
0, & {\rm if}\;\; \mu_l^M=0,
\end{array}
\right.
\end{equation}
where
\[
\mathcal{U}^{M}:={\rm rand}(M,1)+i\ {\rm rand}(M,1) \in \mathbb{C}^{M},
\]
with rand$(M,1)$ returning $M$ uniformly distributed random numbers between $0$ and $1$.
For typical initial values, see Fig. \ref{fig:ini_data-case2}.

\begin{figure}[htbp!]
\begin{center}
\includegraphics[width=2.8in,height=2.0in]{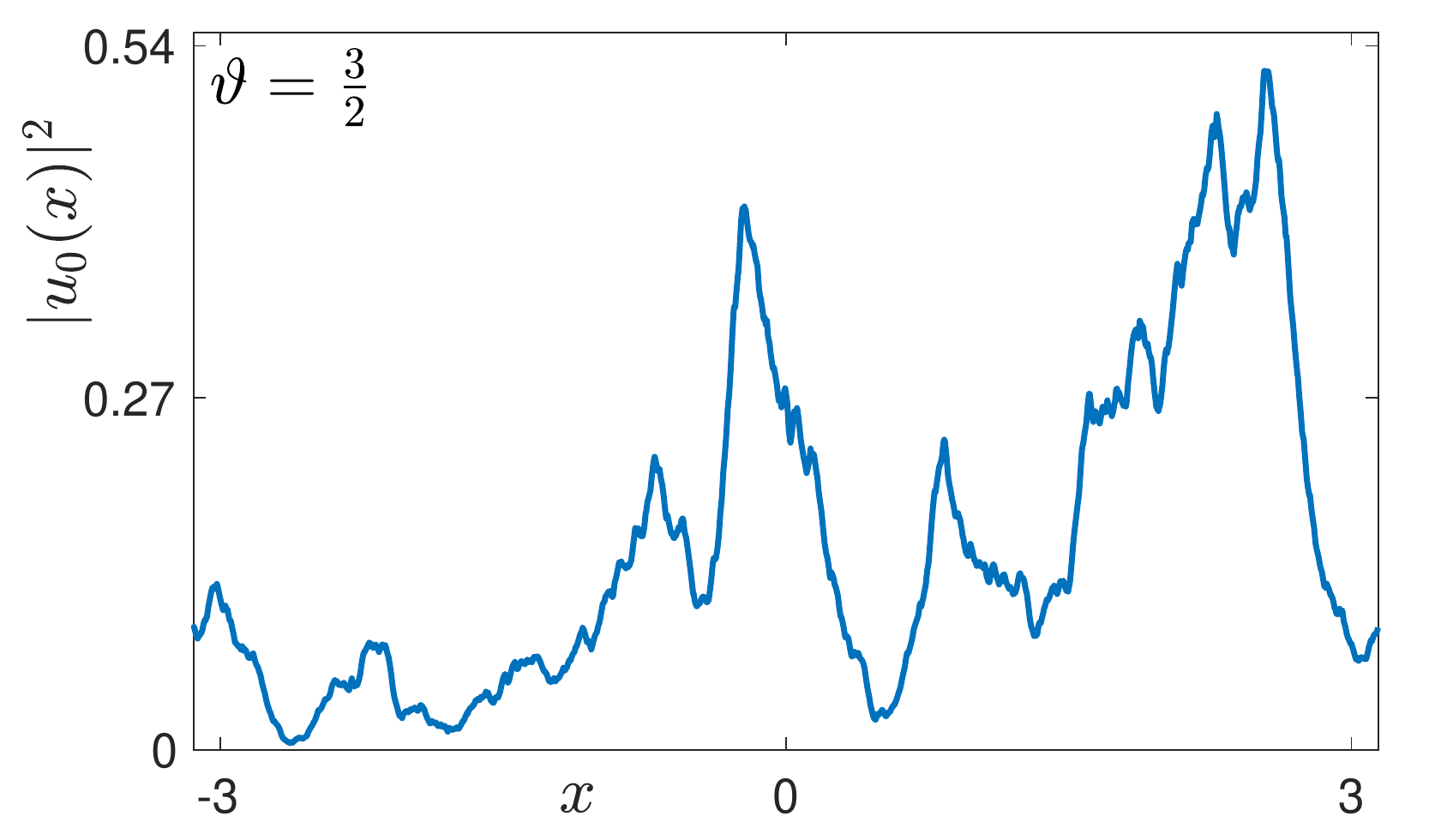}\hspace{0.6cm}
\includegraphics[width=2.8in,height=2.0in]{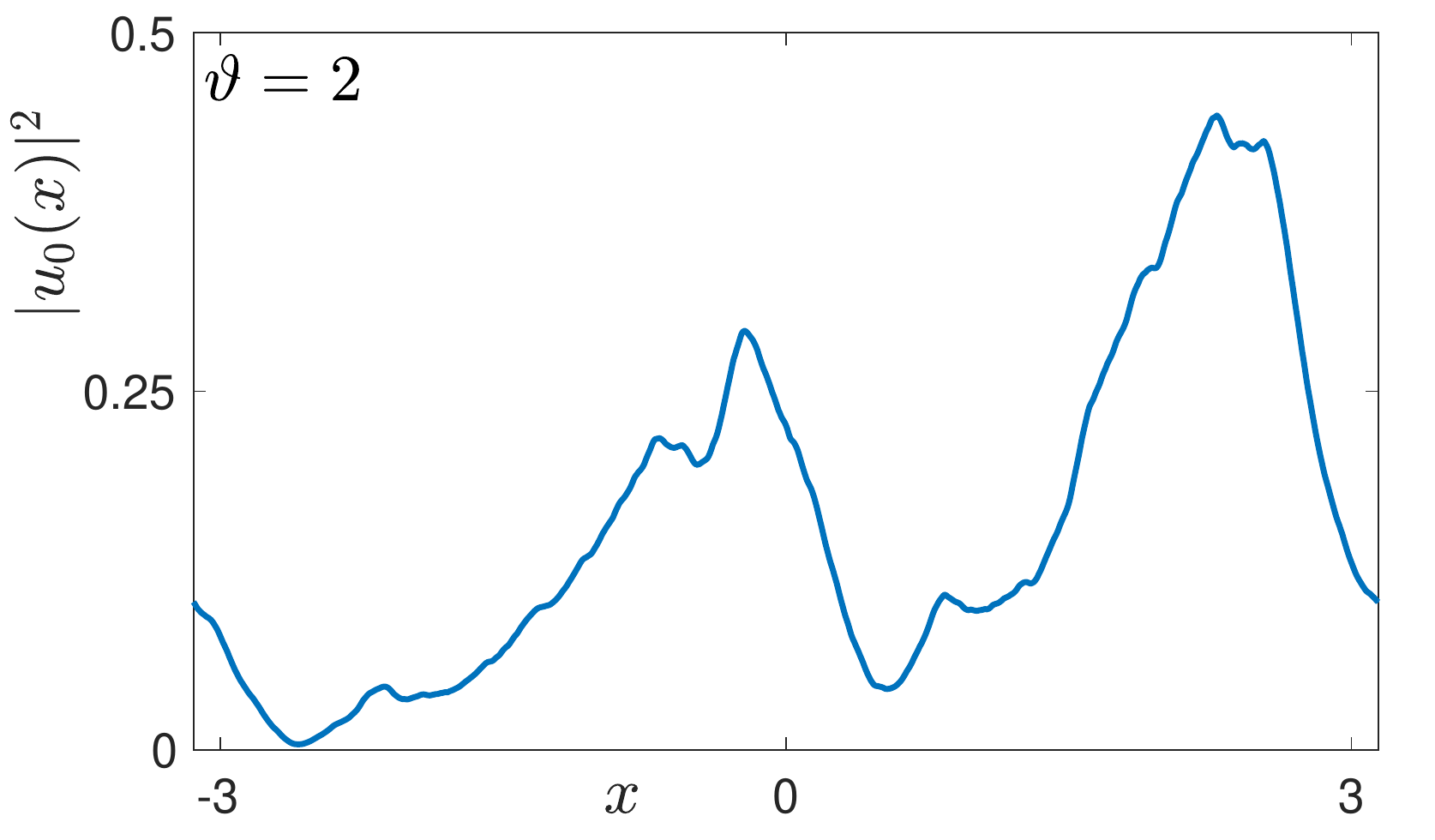}
\end{center}
\caption{{Initial data \eqref{ini-data-case2} for $\vartheta=\fl{3}{2}$ (left) and $\vartheta=2$ (right).}}
\label{fig:ini_data-case2}
\end{figure}

The RLogSE \eqref{RLSE} is then solved by CNFD,  LTSP and STSP on the domain
$\Og=[-16, 16]$ and $\Og=[-\pi, \pi]$ for {\it Case I} and {\it Case II}, respectively.
To quantify the numerical errors, we introduce the error function
\begin{equation}
\label{errfun}
e^{\varepsilon}(t_k)=u^{\varepsilon}(\cdot,t_k)-u^{\varepsilon,k},
\end{equation}
where $u^{\varepsilon}$ is the exact solution of the RLogSE \eqref{RLSE}, while $u^{\varepsilon,k}$ is the
numerical solution obtained by the CNFD,  LTSP or STSP.

\bigskip

{\bf Example 1}. We consider the initial data {\it Case I} \eqref{ini-data-case1} with $v=1$.  The `exact' solution $u^{\varepsilon}$ in \eqref{errfun} is obtained
numerically by the STSP with $\tau=\tau_e=:10^{-6}$ and $h=h_e=:\fl{1}{2^{8}}$. For LTSP and STSP, we fix $h=h_e$ and
vary $\tau=\tau_k^j=:\fl{10^{1-j}}{10+k}$ for $j=1,2,3$ and $k=0,\cdots,90$. For CNFD, we vary  the mesh size and time step
simultaneously under ratio $\tau=\fl{2}{5}h=\tau_j=:\fl{2^{-j}}{5}$ for $j=0,\cdots,7$. Fig. \ref{fig:case1} shows the errors
 $\|e^{\varepsilon}(1)\|_{H^1}$ vs time step $\tau$ under different $\varepsilon$ for CNFD,  LTSP and STSP schemes.
It clearly shows that LTSP/STSP is first/second-order convergent in time while CNFD is second-order convergent in both space and time.
In addition, for other initial datum smooth enough (not shown here for brevity), all methods show their classical orders of convergence.
The same conclusion applies to $\|e^{\varepsilon}(1)\|_{L^2}$ and  $\|e^{\varepsilon}(1)\|_{L^\infty}$. Here the norms $\|\cdot\|_{L^2}$, $\|\cdot\|_{H^1}$ and $\|\cdot\|_{L^\infty}$ are defined as \eqref{norm}.

\begin{figure}[h!]
\begin{center}
\includegraphics[width=3.0in,height=2.0in]{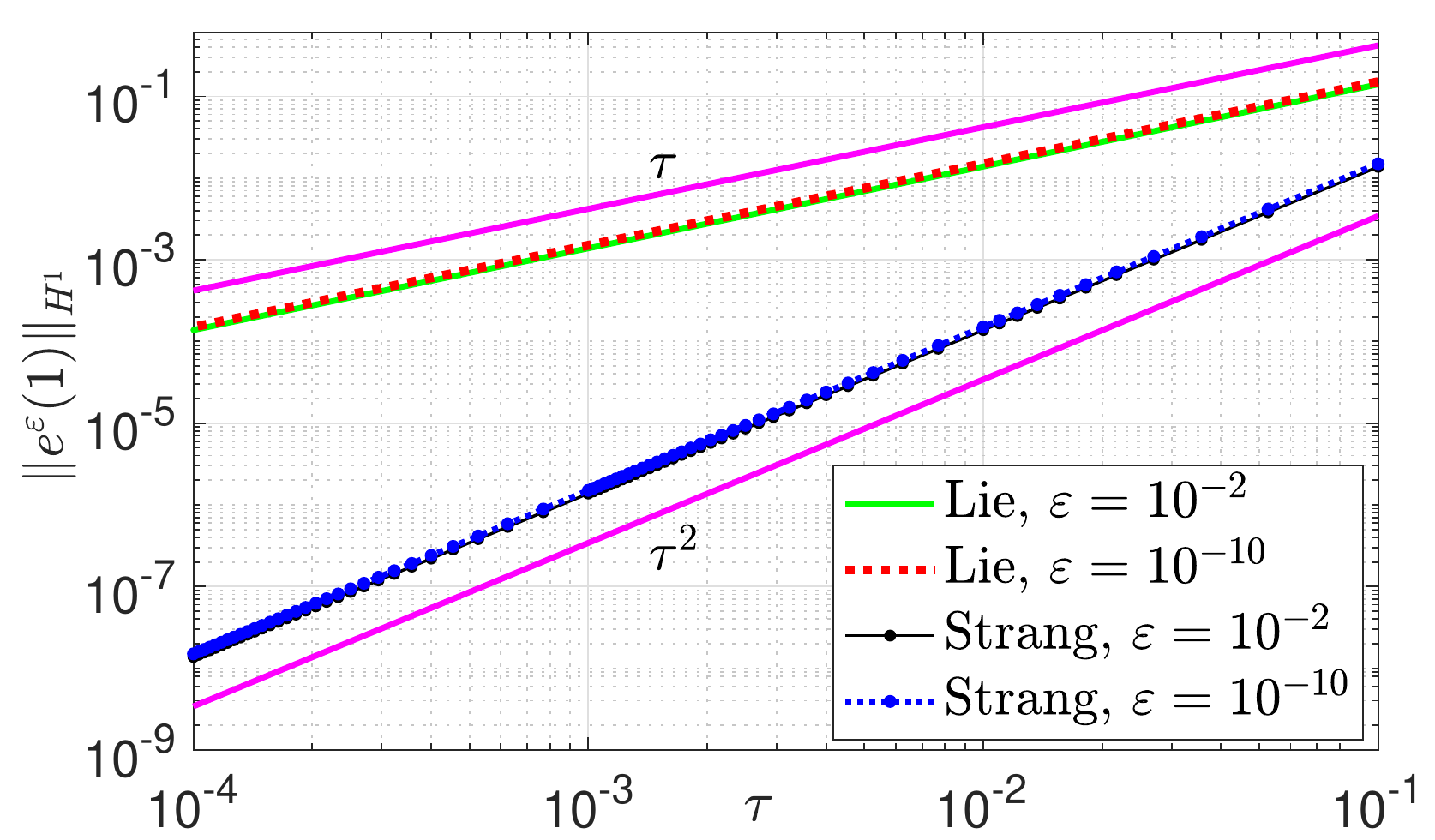}
\hspace{0.1cm}
\includegraphics[width=3.0in,height=2.0in]{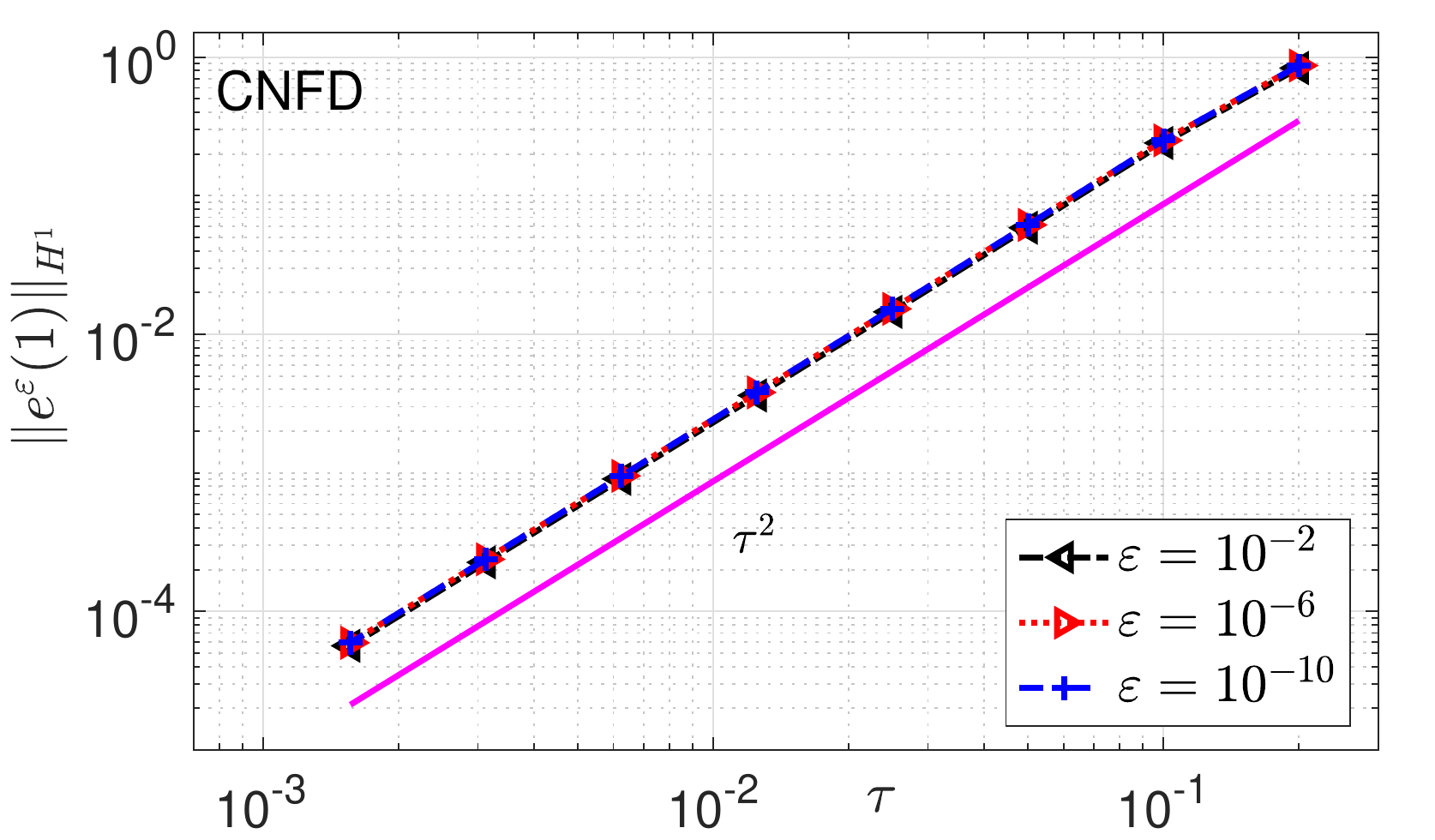}
\end{center}
\caption{{Errors $\|e^\ep(1)\|_{H^1}$ of LTSP \& STSP (left) and CNFD (right) for {\it Case I}. }}
\label{fig:case1}
\end{figure}

\bigskip

{\bf Example 2.}
We consider the initial data {\it Case II} \eqref{ini-data-case2}.  The `exact' solution $u^{\varepsilon}$ in \eqref{errfun} is obtained
numerically by the STSP with $\tau=\tau_e=:10^{-6}$ and $h=h_e=:\fl{\pi}{2^{15}}$. For all the methods, we fix $h=h_e$
and vary $\tau=\tau_k^j=:\fl{10^{1-j}}{10+k}$ for $j=1,2,3$ and $k=0,\cdots,90$.  The errors  $\|e^{\varepsilon}(1)\|_{L^2}$
\and  $\|e^{\varepsilon}(1)\|_{H^1}$ of the schemes LTSP, STSP and CNFD for the initial value \eqref{ini-data-case2}
 with different values of $\vartheta$ are illustrated in Figs. \ref{fig:case2-L2}  and \ref{fig:case2-H1}, respectively. From these figures we can see that:
 (i) For smaller values of $\vartheta$, i.e., when the initial data is not smooth enough, all the errors show a zigzag behavior due to
 happy error cancelation or accumulation occurring. Order reduction occurs for all methods in this case.
 (ii) In $L^2$ norm, the LTSP is half-order convergent  for  $H^1$ initial datum (cf. $\vartheta=1$ in Fig.  \ref{fig:case2-L2}),
 which confirms the conclusion in Theorem \ref{thmlt}. Meanwhile, it is first-order convergent for  $\vartheta\ge2$, which is in line with Remark \ref{rem:better-order}.
(iii) The STSP is second-order convergent in $L^2$ norm for  $\vartheta\ge4$, while the CNFD recovers its second-order convergence only when $\vartheta\ge5$.
(iv) For all the methods, to recover their  classical orders of convergence, the initial data is required to be more regular by one additional order when errors are
measured in $H^1$ norm than in  $L^2$ norm.

\begin{figure}[htbp!]
\begin{center}
\includegraphics[width=3.0in,height=2.0in]{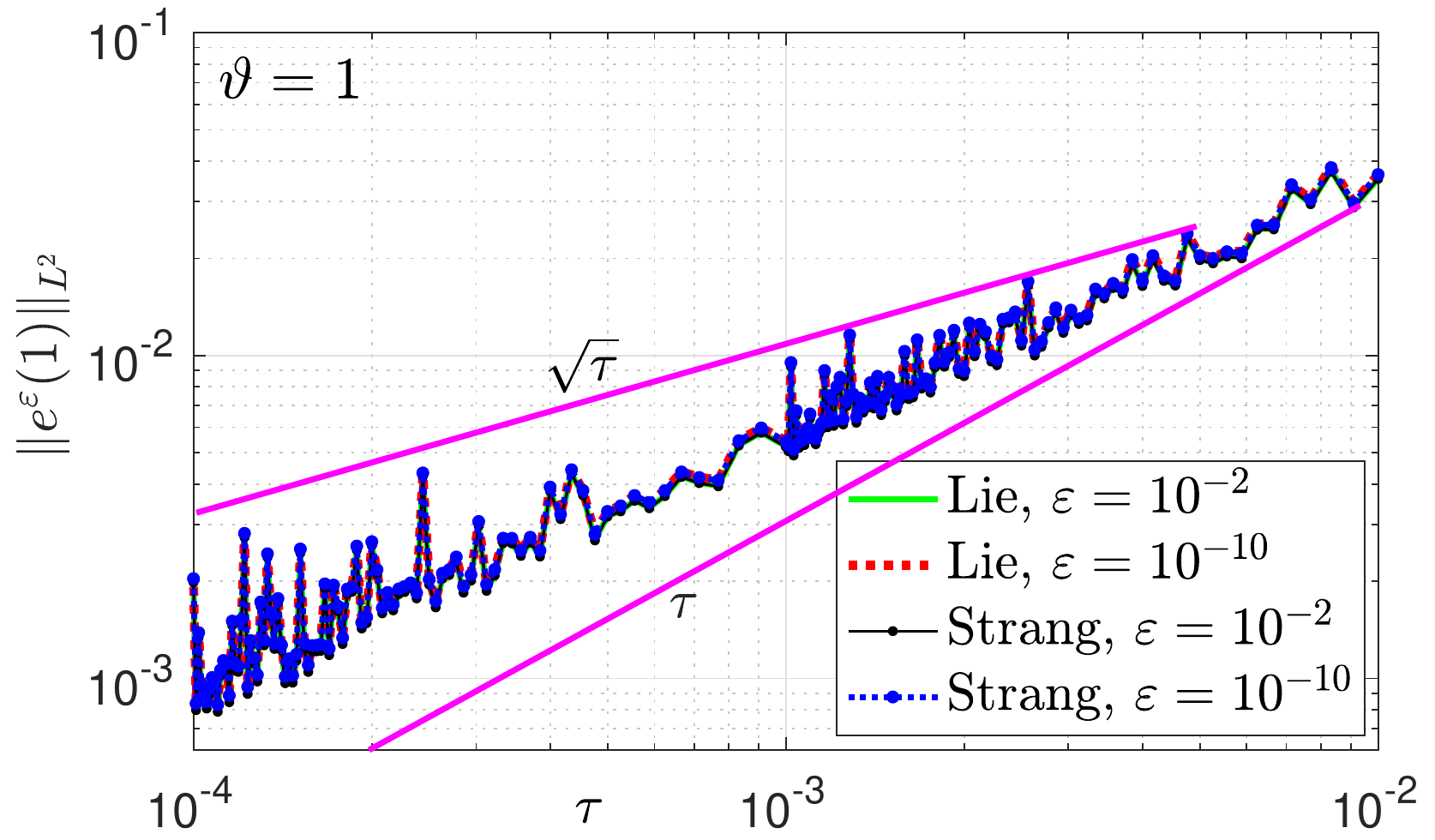}\hspace{0.2cm}
\includegraphics[width=3.0in,height=2.0in]{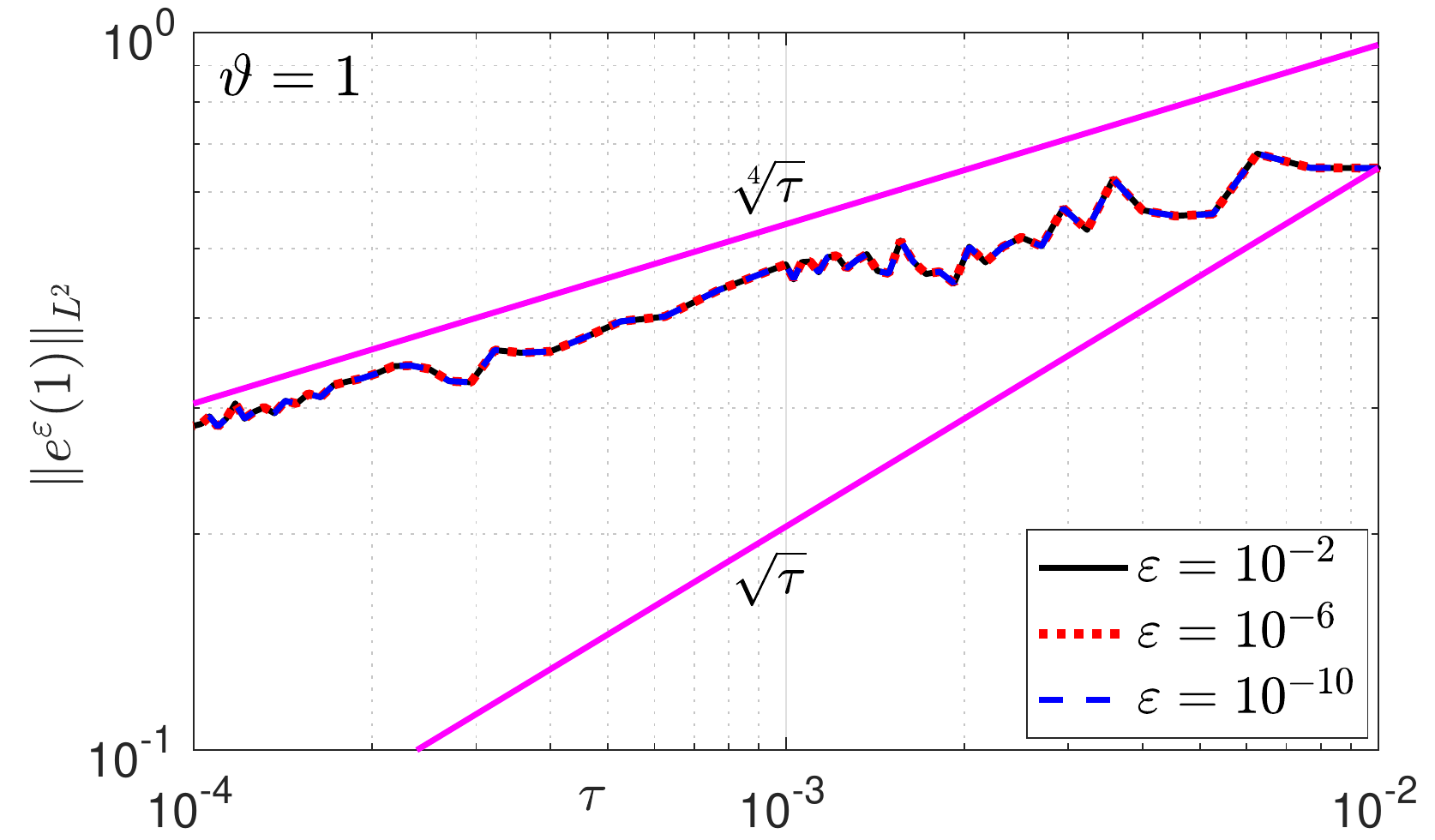}\\
\includegraphics[width=3.0in,height=2.0in]{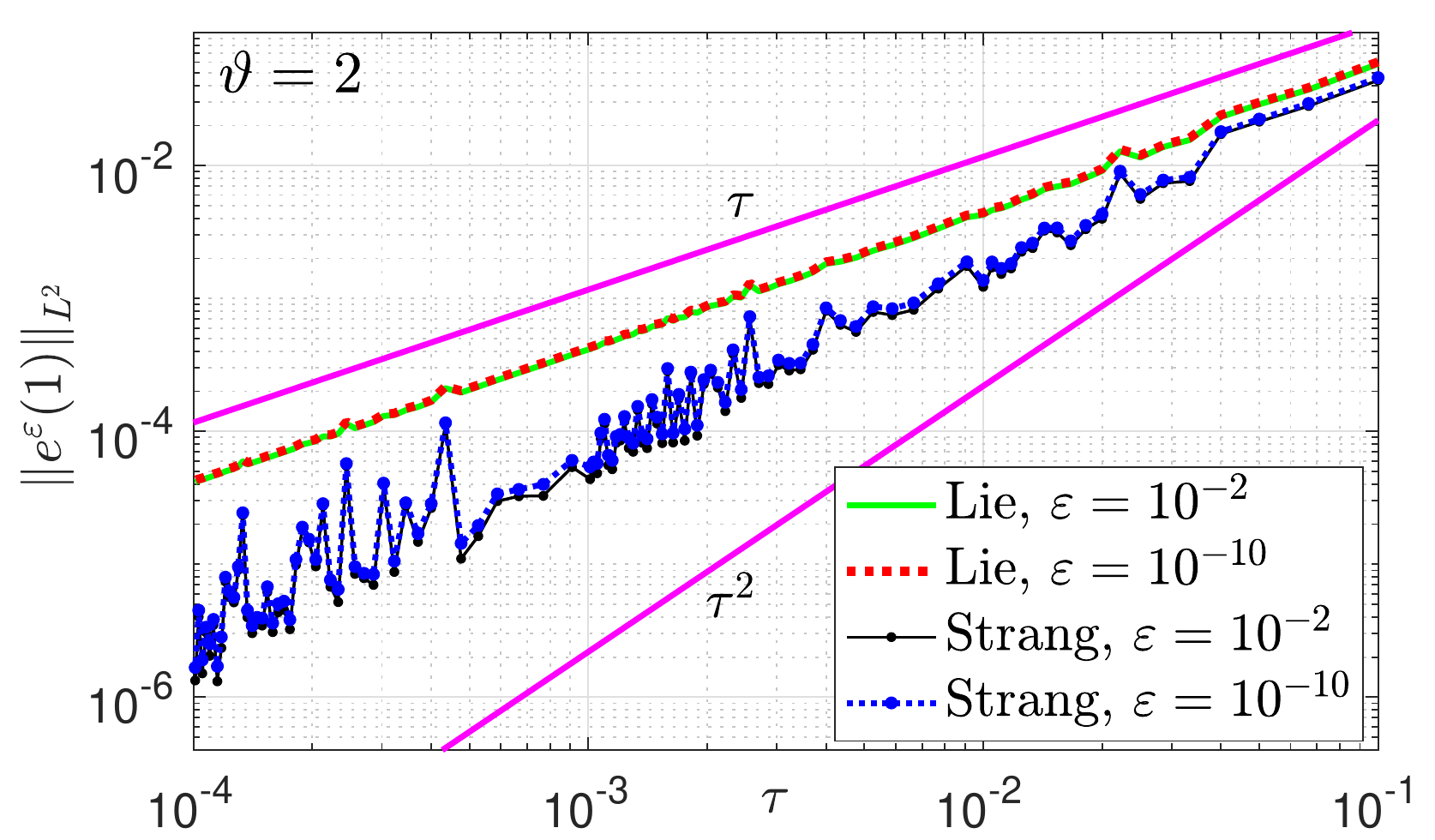}\hspace{0.2cm}
\includegraphics[width=3.0in,height=2.0in]{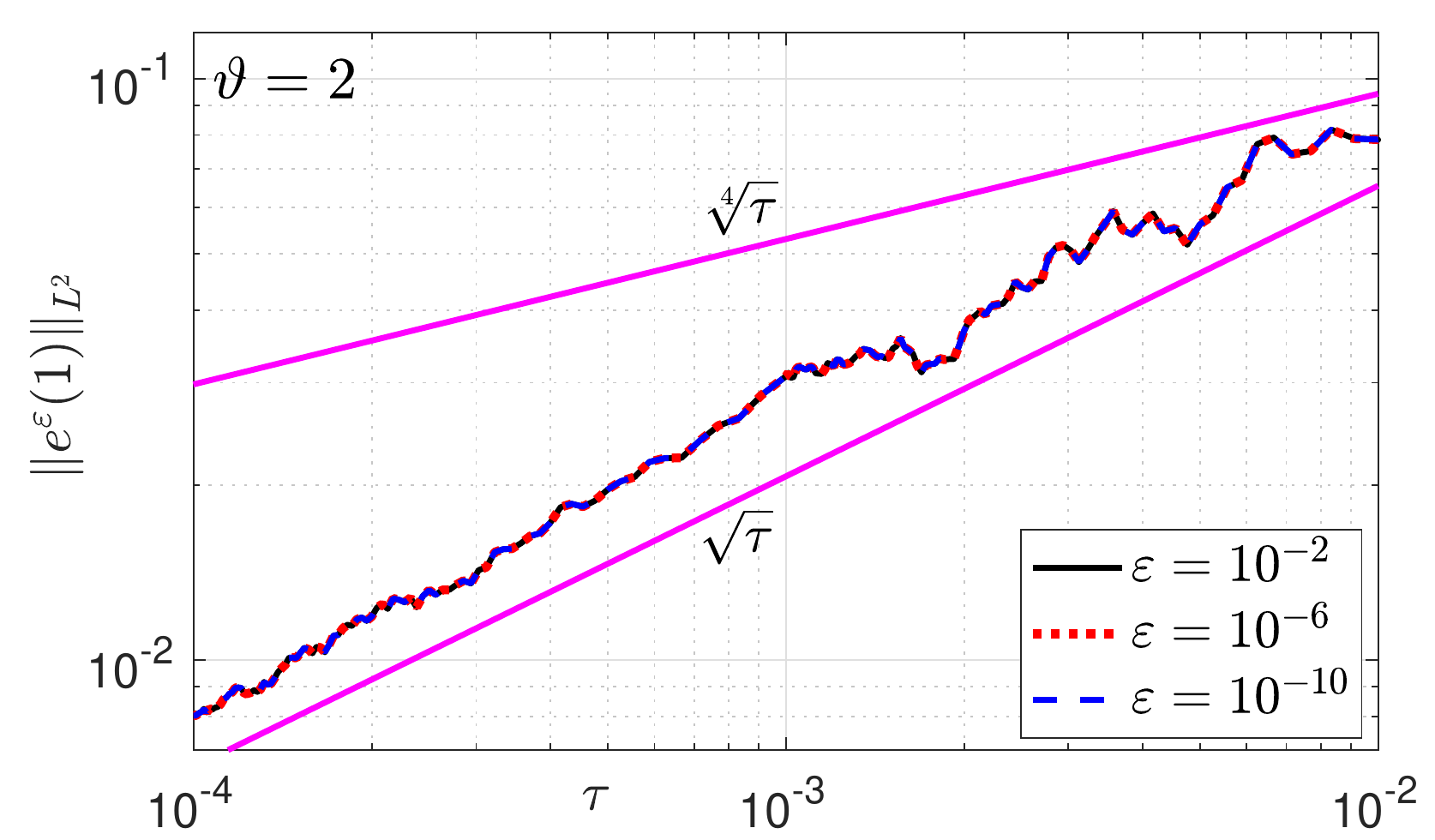}\\
\includegraphics[width=3.0in,height=2.0in]{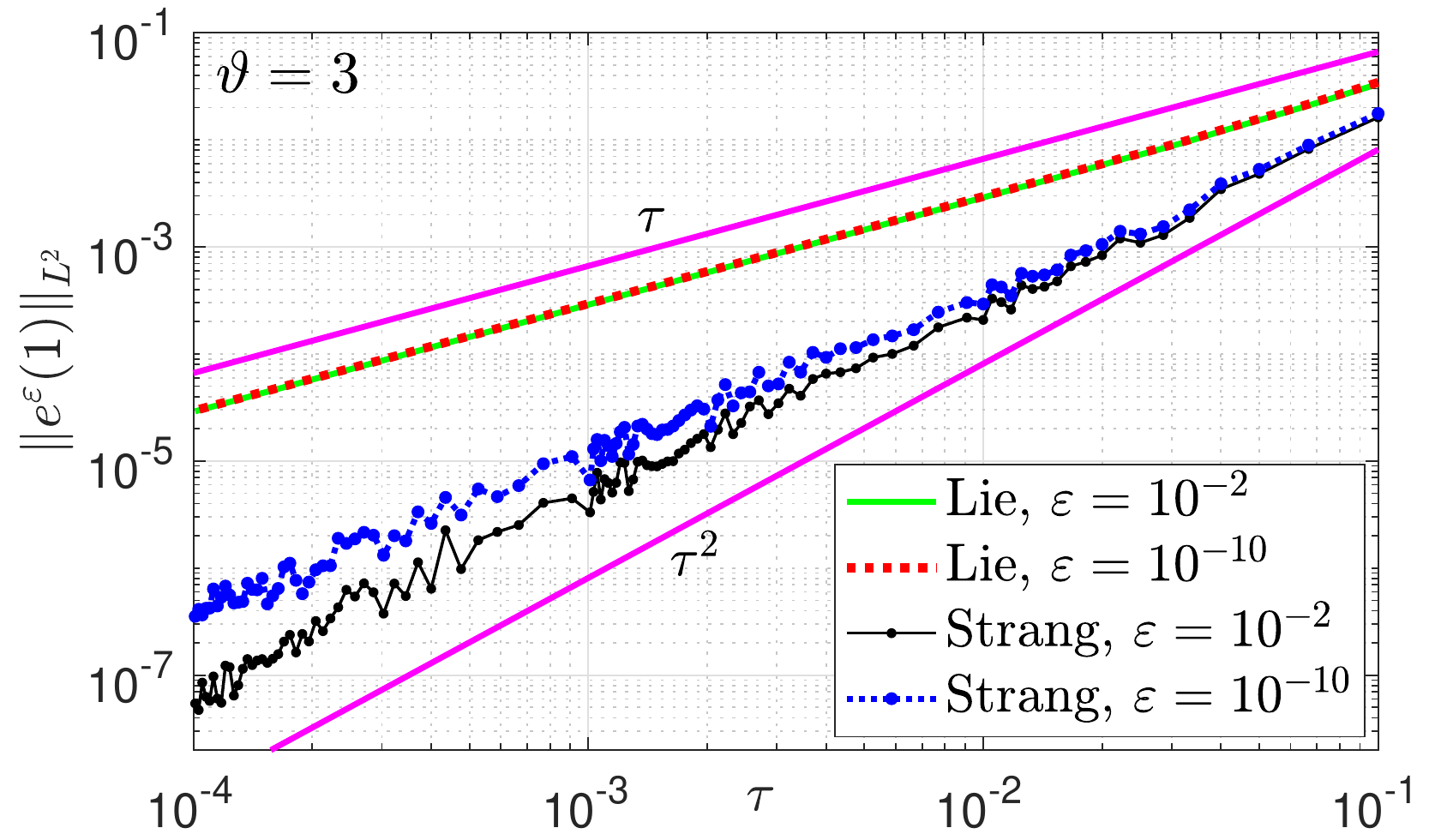}\hspace{0.2cm}
\includegraphics[width=3.0in,height=2.0in]{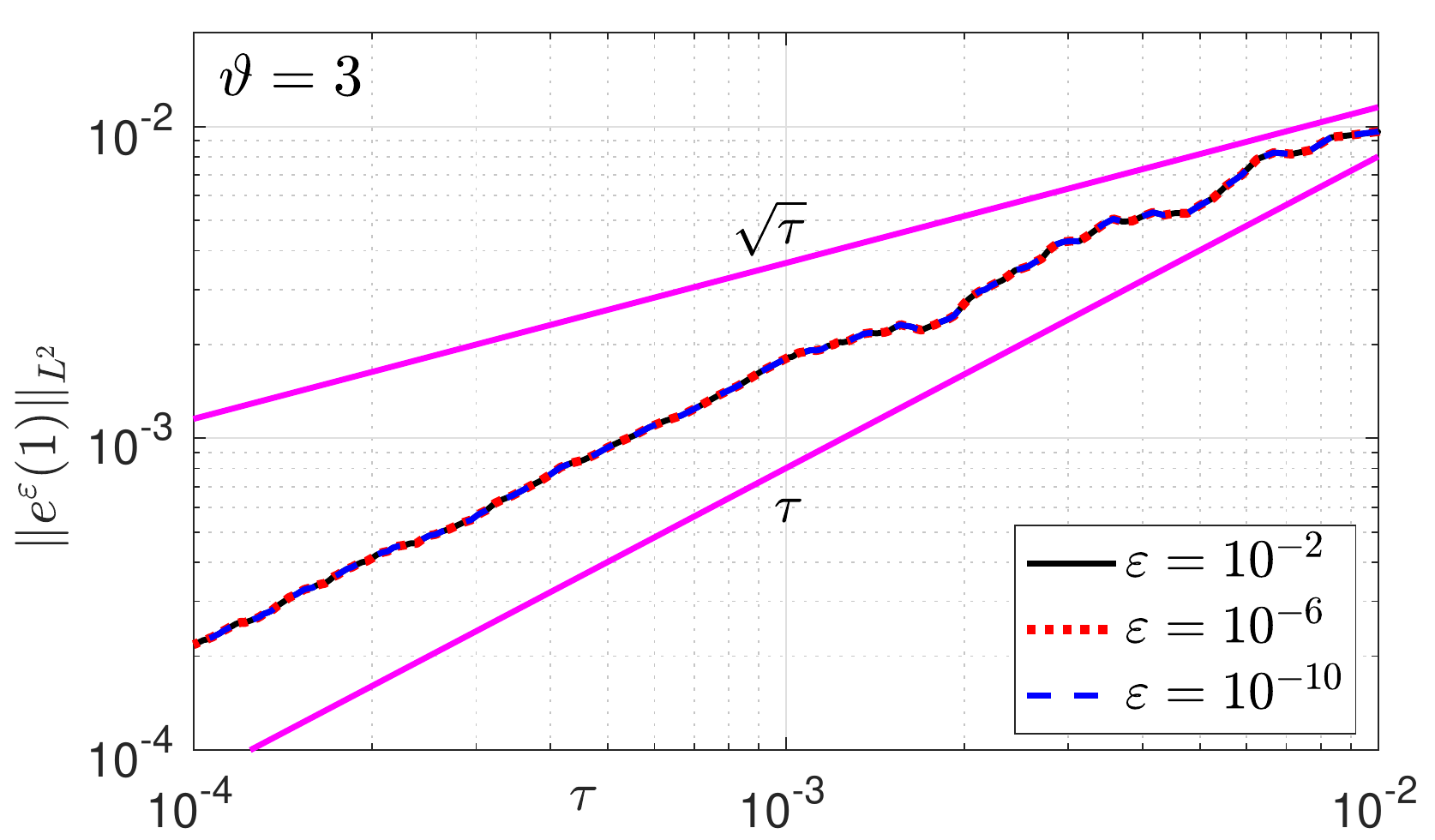}\\
\includegraphics[width=3.0in,height=2.0in]{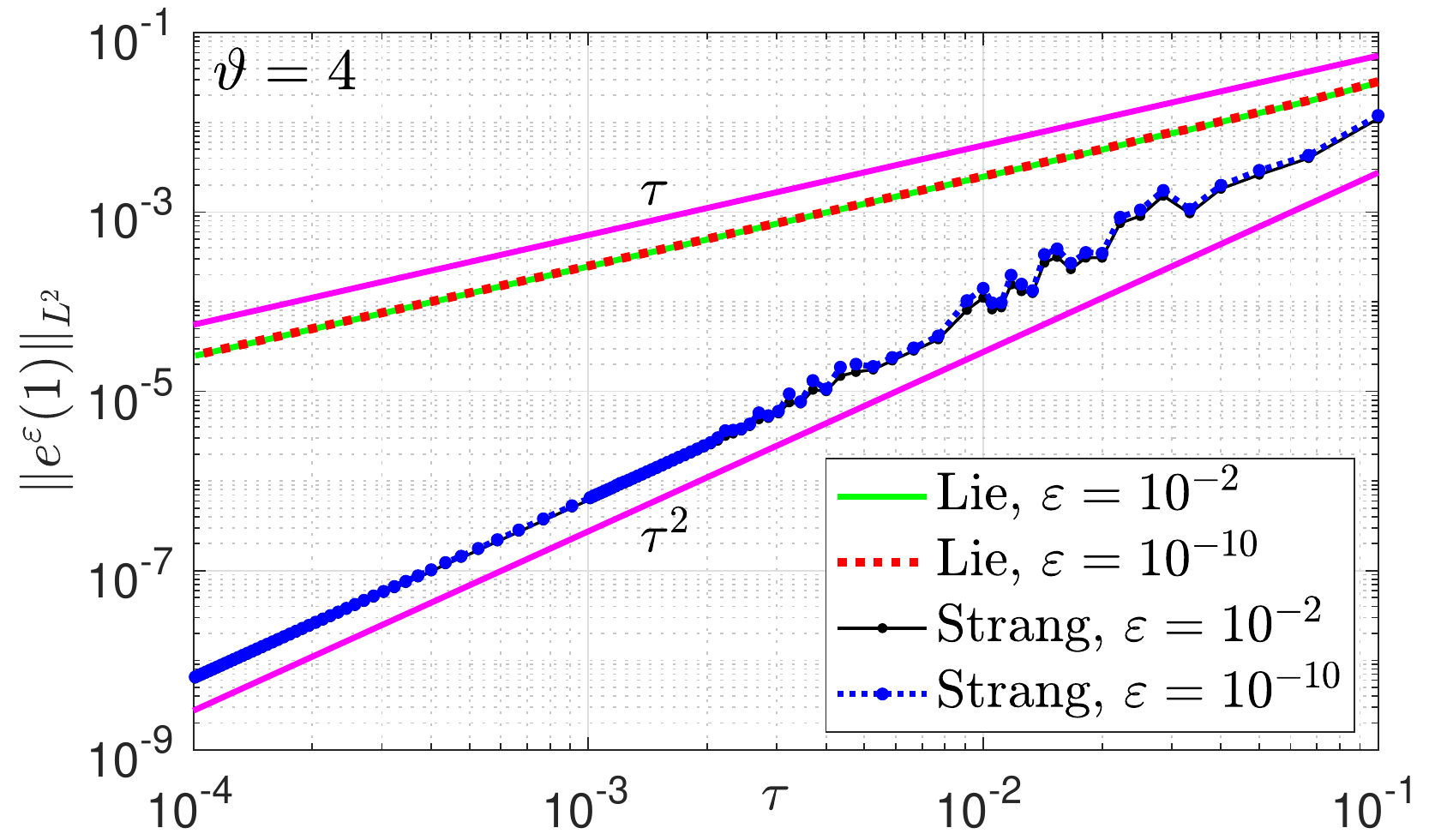}\hspace{0.2cm}
\includegraphics[width=3.0in,height=2.0in]{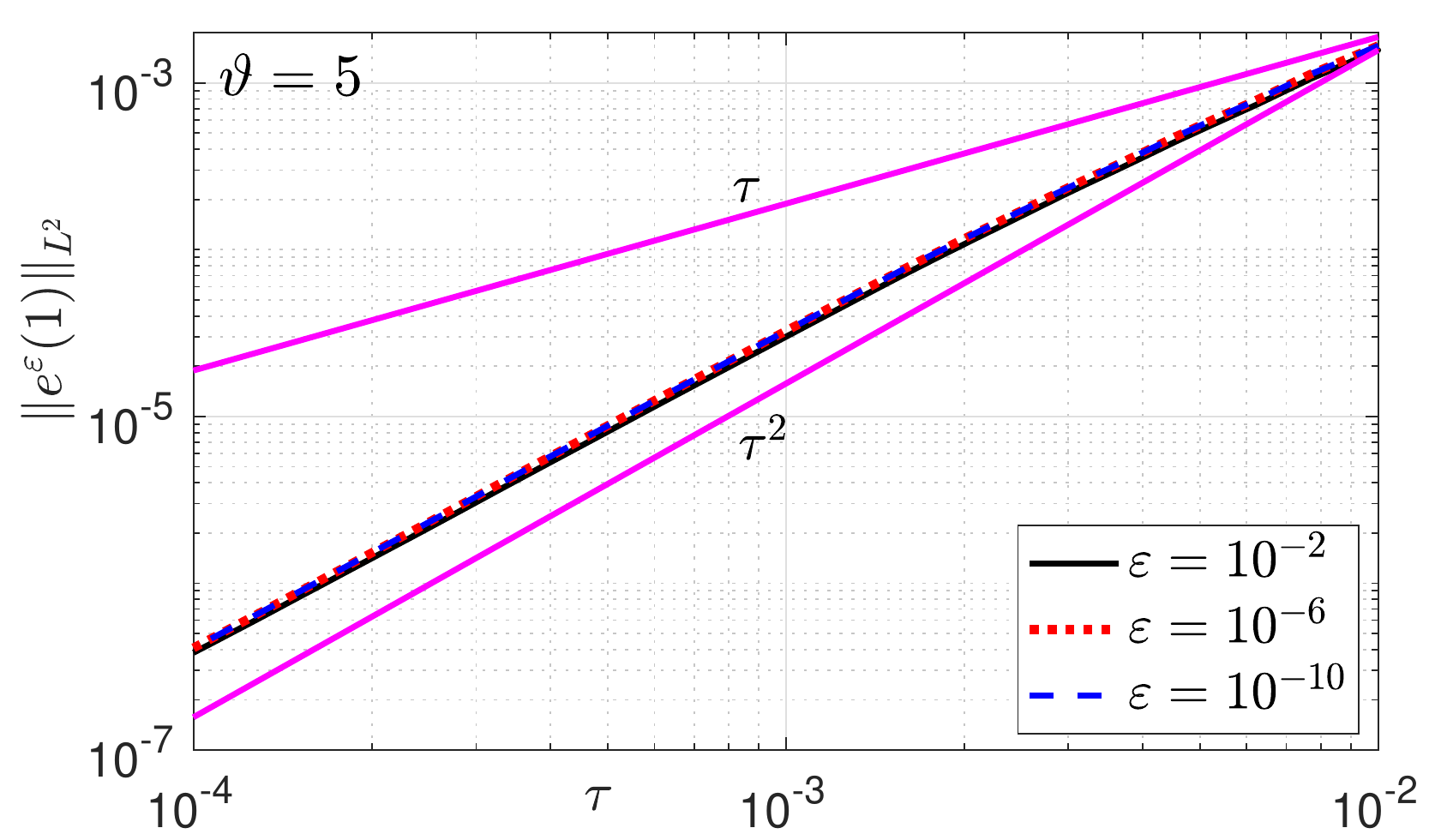}
\end{center}
\caption{{Errors $\|e^\ep(1)\|_{L^2}$ for the LTSP \& STSP (left) and CNFD (right)  in {\it Case II} for different
values of $\vartheta$ in \eqref{ini-data-case2}.}}
\label{fig:case2-L2}
\end{figure}

\begin{figure}[htbp!]
\begin{center}
\includegraphics[width=3.0in,height=2.0in]{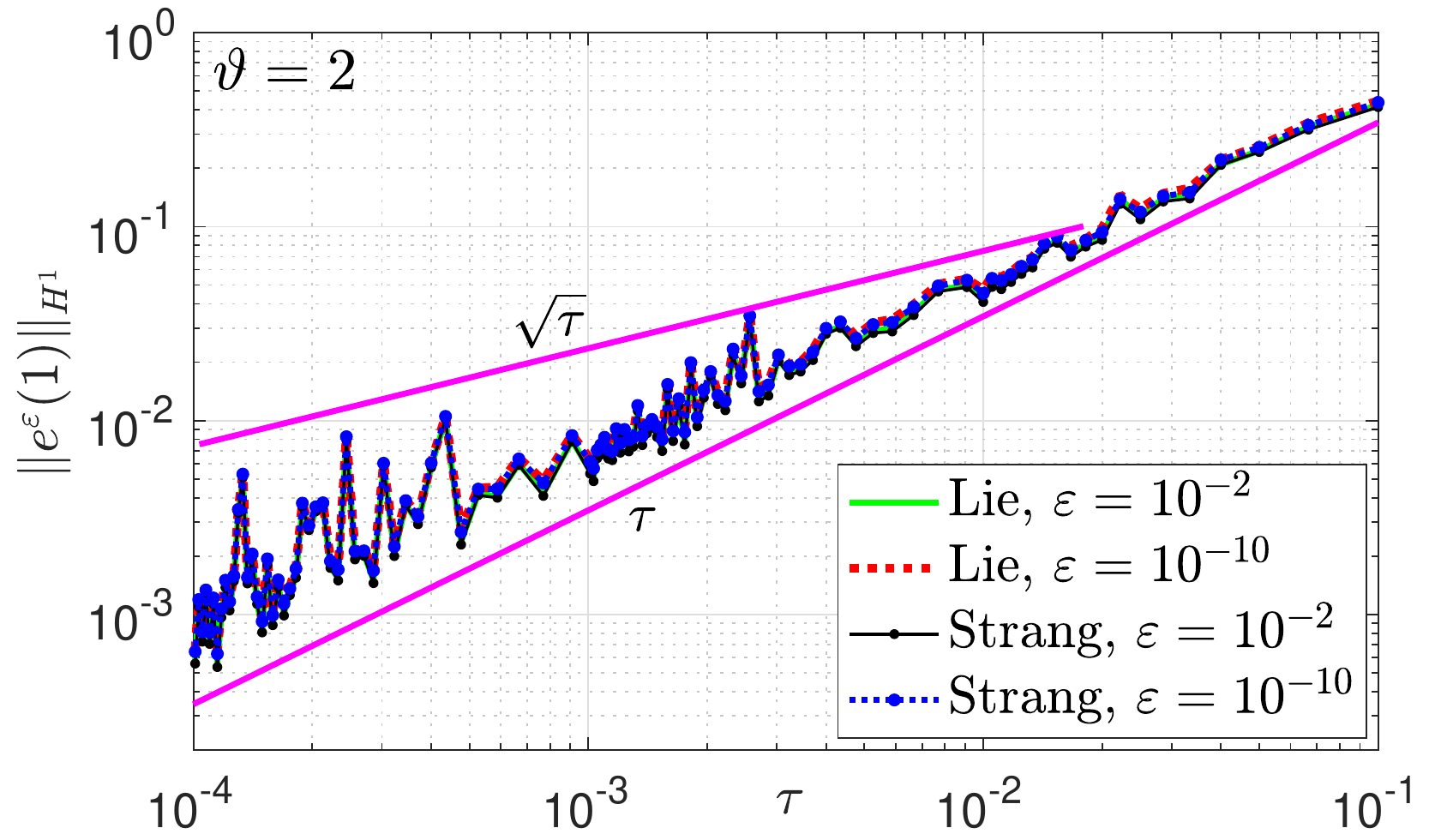}\hspace{0.2cm}
\includegraphics[width=3.0in,height=2.0in]{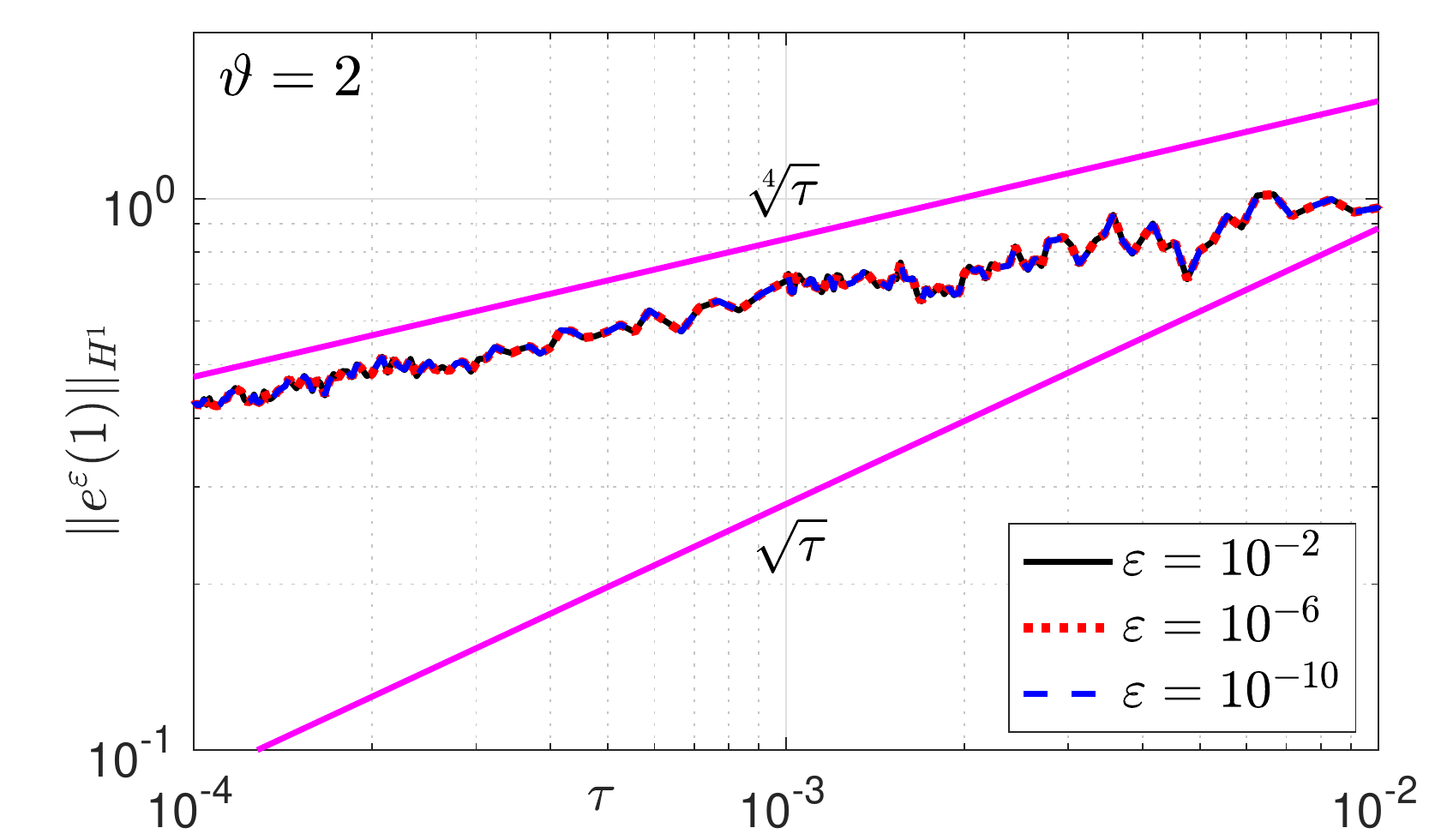}\\
\includegraphics[width=3.0in,height=2.0in]{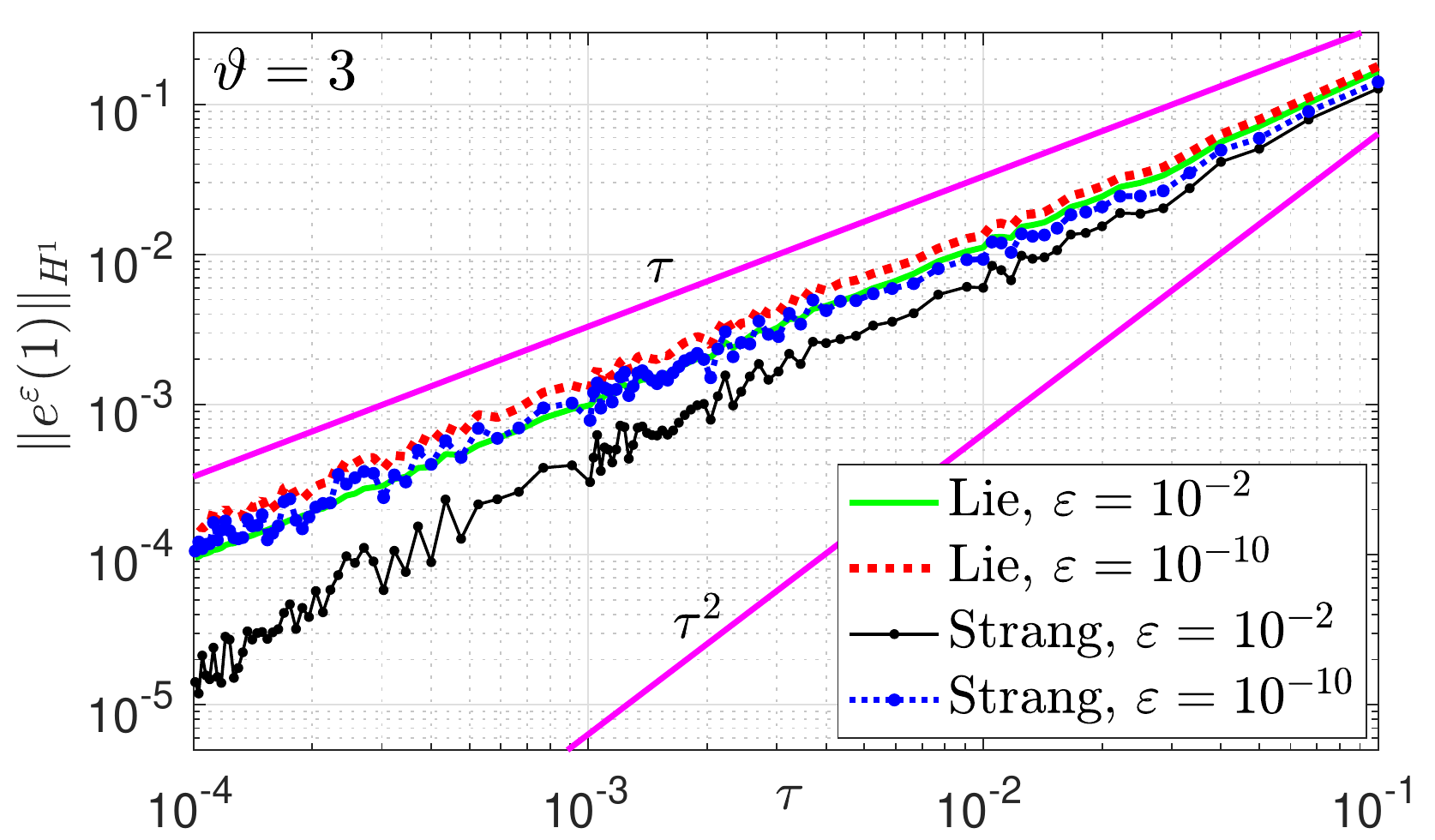}\hspace{0.2cm}
\includegraphics[width=3.0in,height=2.0in]{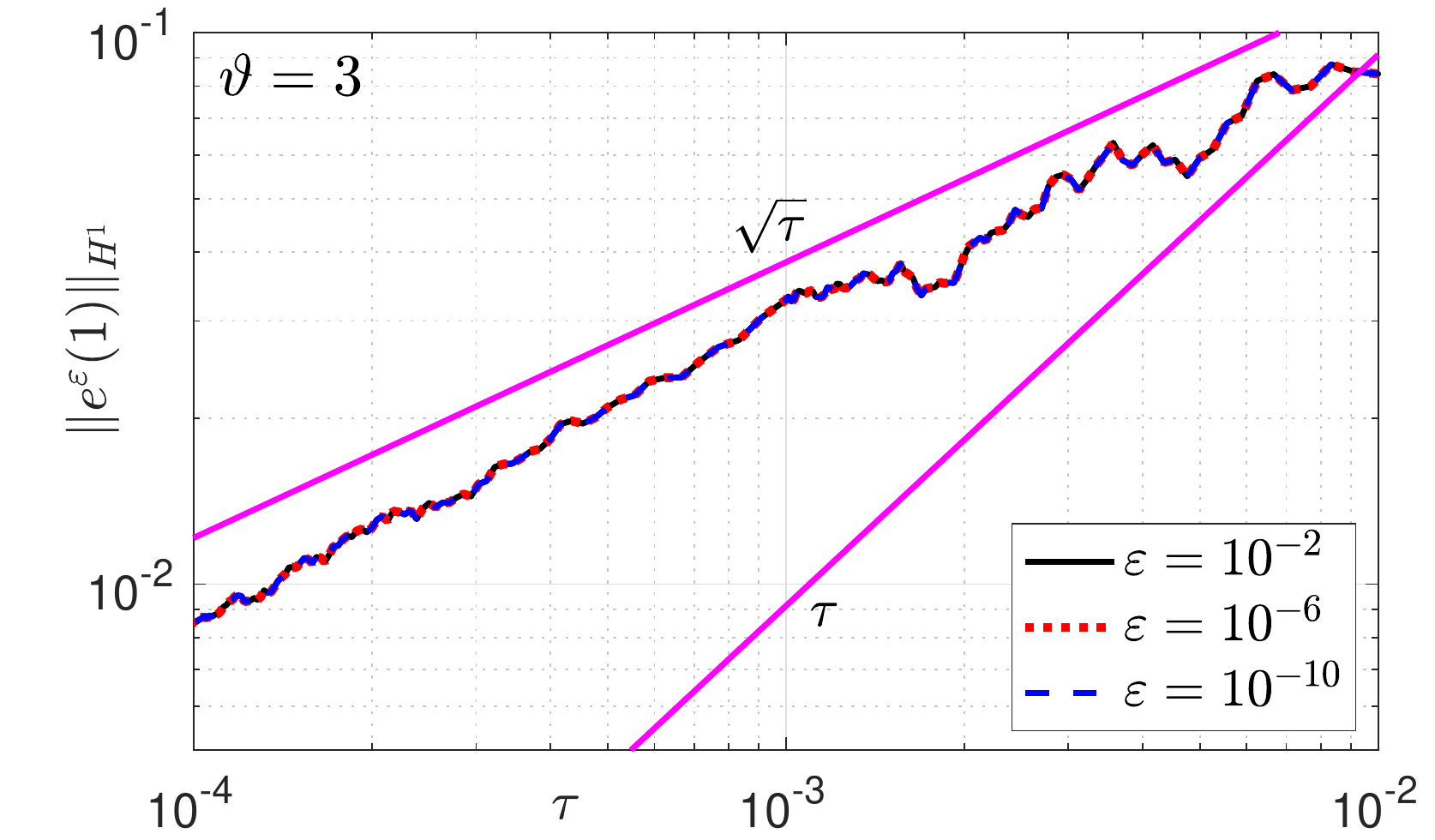}\\
\includegraphics[width=3.0in,height=2.0in]{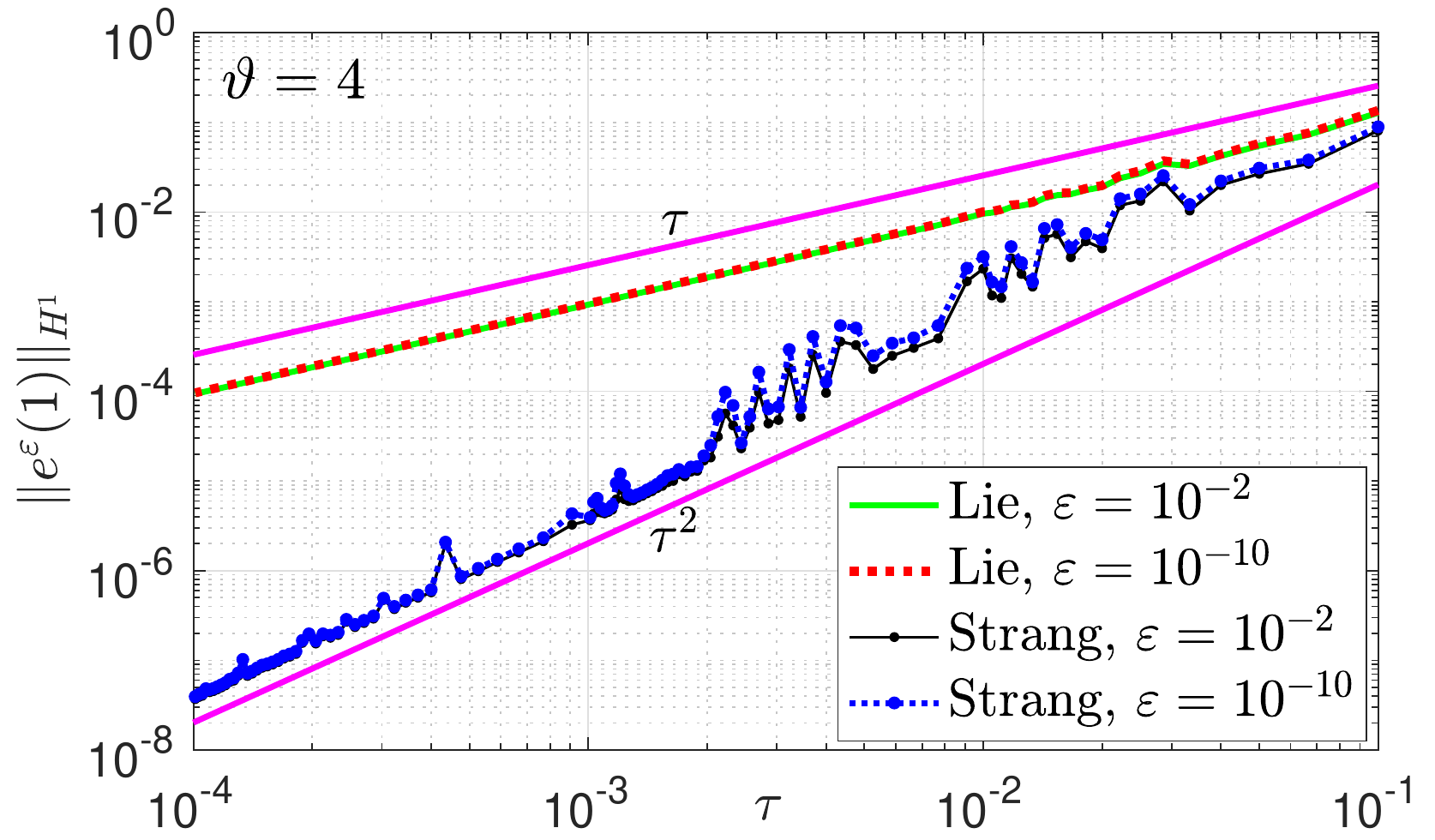}\hspace{0.2cm}
\includegraphics[width=3.0in,height=2.0in]{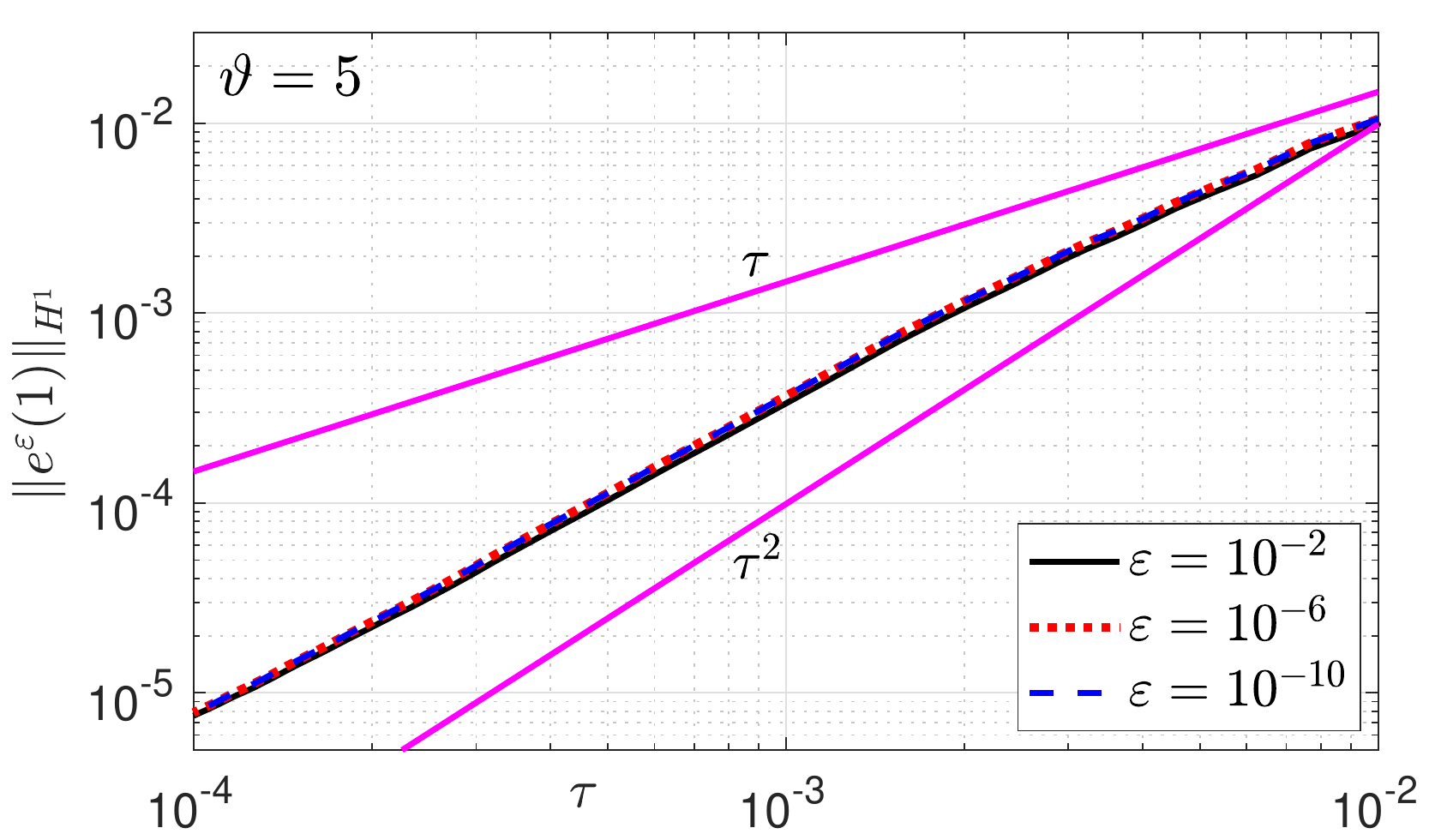}\\
\includegraphics[width=3.0in,height=2.0in]{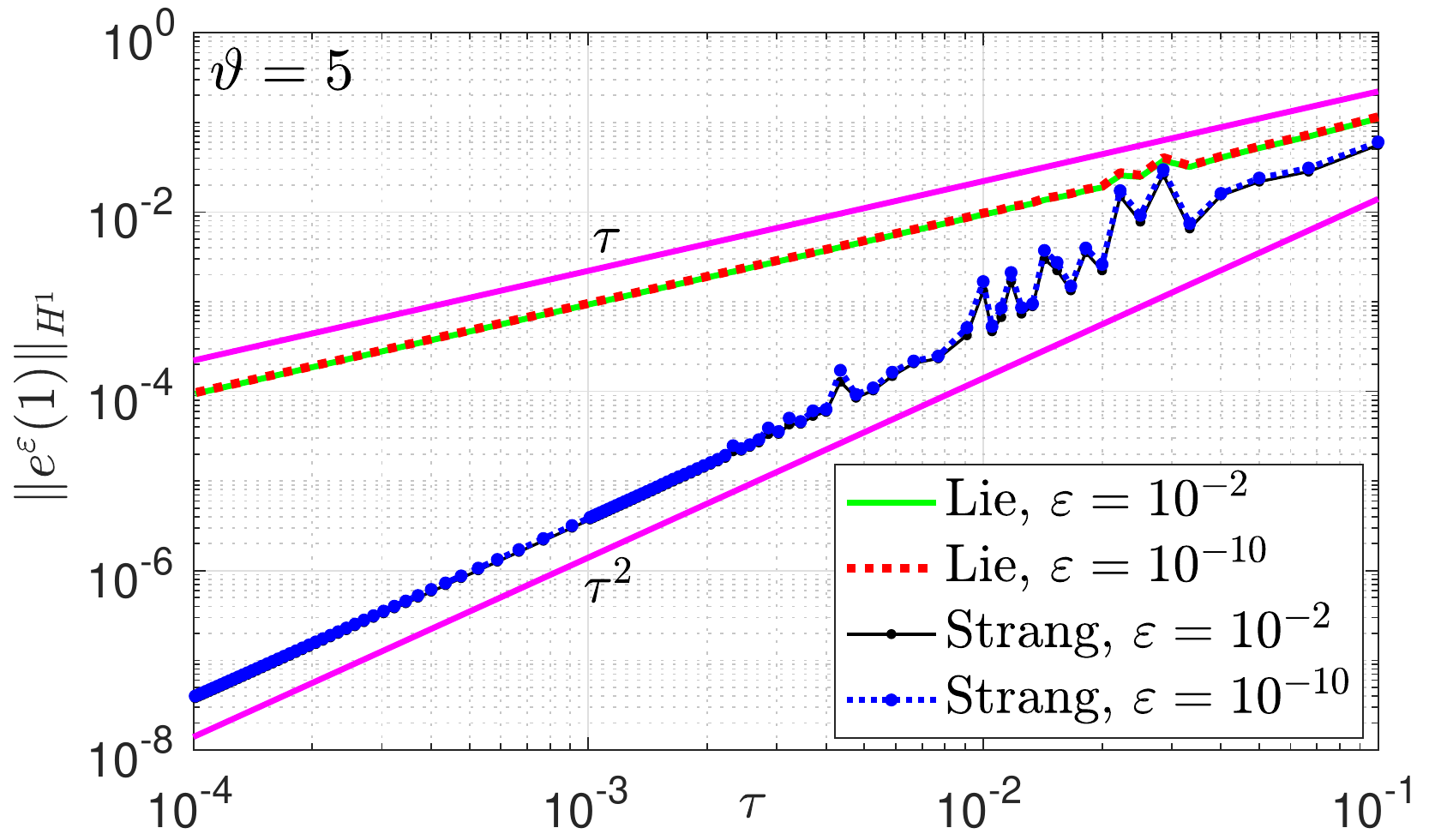}\hspace{0.2cm}
\includegraphics[width=3.0in,height=2.0in]{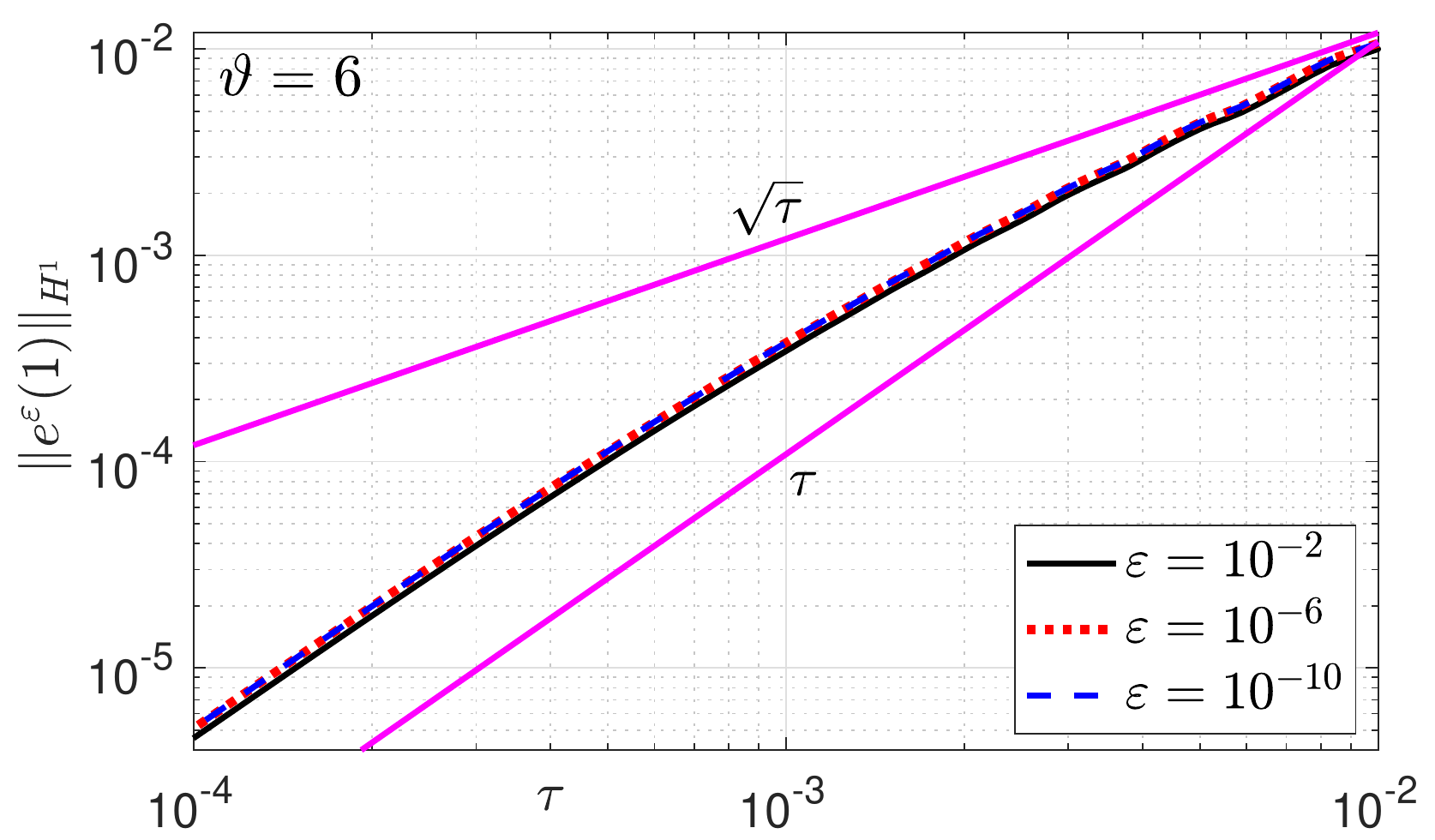}
\end{center}
\caption{{Errors $\|e^\ep(1)\|_{H^1}$ for the LTSP \& STSP (left) and CNFD (right) in {\it Case II} for different values
of $\vartheta$ in \eqref{ini-data-case2}.}}
\label{fig:case2-H1}
\end{figure}

\subsection{Applications for long time dynamics}
\label{sec:long-time}
In this section, we apply the STSP method to investigate long time dynamics of LogSE with Gaussian-type initial datum in 1D.
 To this end, we fix $\varepsilon=10^{-15}$, $\Og=[-L,L]$, $h=\fl{1}{16}$ and $\tau=0.001$. The initial data is chosen as
\be\label{ini_set}
u_0(x)=\sum_{k=1}^N b_k e^{-\fl{a_k}{2}(x-x_k)^2+iv_k x},\qquad x\in\mathbb{R},
\ee
where $b_k$, $a_k$, $x_k$ and $v_k$ are real constants,
i.e, the initial data is the sum of $N$ Gaussons \eqref{ini-gaus} with velocity $v_k$ and initial location $x_k$.

\medskip

\noindent{\bf Example 3}. Here, we let $\lambda=-1$, $L=1000$.  We set $N=2$
 in \eqref{ini_set}
and consider the following cases:
\begin{itemize}
\item[](i).    $x_1=-x_2=-5$, $v_k=0$, $b_k=a_k=1$ ($k=1,2$);
\item[] (ii).   $x_1=-x_2=-3$,  $v_k=0$,    $b_k=a_k=1$ ($k=1,2$);
\item[] (iii). $v_1=v_2=2$,  $x_1=-x_2=-30$,   $b_k=a_k=1$ ($k=1,2$);
\item[] (iv).  $v_1=v_2=15$,  $x_1=-x_2=-30$,   $b_k=a_k=1$ ($k=1,2$);
\item[] (v). $v_1=1$,  $v_2=0$, $x_1=-40$,   $x_2=0$,   $b_2=2b_1=1$,   $a_k=1$ ($k=1,2$);
\item[] (vi). $v_1=4$,  $v_2=0$, $x_1=-40$,    $x_2=0$,   $b_2=2b_1=1$,   $a_k=1$ ($k=1,2$);
\item[] (vii).  $v_1=25$, $v_2=0$,   $x_1=-100$, $x_2=0$,    $b_2=2b_1=1$,  $a_k=1$ ($k=1,2$);
\item[] (viii). $v_1=10$, $v_2=0$,     $x_1=-50$,  $x_2=30$, $b_2=2b_1=1$,  $a_1=1.2$,  $a_2=0.8$.
\end{itemize}
Fig. \ref{fig:ex3-casei-iv} shows the evolution of $\sqrt{|u^\ep(x,t)|}$, $E^\varepsilon_{\rm kin}(t)$, $E^\varepsilon_{\rm int}(t)$ and $E^\varepsilon(t)$ as well as the plot of $|u^\ep(x,t)|$ at different time for
 Cases (i)-(iv).  While Fig. \ref{fig:ex3-casev-viii} illustrates those for  Cases (v)-(viii). Here, the kinetic  energy $E^\varepsilon_{\rm kin}$
 and interaction energy $E^\varepsilon_{\rm int}$
 are defined  as:
 \begin{align*}
&  E^\ep_{\rm kin} (t)= \int_{\Omega}\big |\nabla u^\ep(\bx,t)|^2 d\bx,\qquad\quad
E^\ep_{\rm int} (t)= E^\ep (t)-E^\ep_{\rm kin} (t).
\end{align*}

From these figures we can see that:
(1) The total energy is well conserved.
(2) For static Gaussons (i.e., $v_k=0$ in \eqref{ini_set}), if they were initially well-separated,
the two Gaussons will stay stable as separated static Gaussons (cf. Fig. \ref{fig:ex3-casei-iv} Case i) with density profile unchanged.
When they get closer, the two Gaussons  contact and undergo attractive interactions.  They move to each other, collide and stick together later.
Shortly, the two Gaussons  separate and swing like pendulum. Small outgoing solitary waves are emitted during the separation of the Gaussons.  This wave-emitting  `pendulum motion' becomes faster as time goes on (cf. Fig. \ref{fig:ex3-casei-iv} Case ii).
(3) For moving Gaussons, the two Gaussons  are transmitted  completely through each other and move separately at last.
Before they meet, the two Gaussons basically  move at constant velocities and preserve their profiles in density exactly if $a_k=\lambda$ (While they will move like
breathers if $a_k\ne\lambda$ (cf. Fig. \ref{fig:ex3-casev-viii} Case viii)).
During the interaction, there occurs  oscillation. Generally, the larger the relative velocity between the two Gaussons is,
the stronger the oscillation is (cf. Fig. \ref{fig:ex3-casei-iv} Cases iii-iv \& Fig. \ref{fig:ex3-casev-viii} Cases vii-viii). After
collision, the velocities of Gaussons change (cf. Fig. \ref{fig:ex3-casev-viii} Cases v-vi).  The two Gaussons  oscillate like breathers and
separate completely at last (cf. Fig. \ref{fig:ex3-casei-iv} Cases iii-iv \& Fig. \ref{fig:ex3-casev-viii} Cases vii-viii).
In addition, small waves are emitted if the relative velocity is small (cf. Fig. \ref{fig:ex3-casei-iv} Cases ii-iv and Fig. \ref{fig:ex3-casev-viii} Cases v-vi \& viii).
(4) For two Gaussons with large relative velocity, their dynamics and interaction  are similar to those of the bright solitons in the cubic nonlinear Schr\"{o}dinger equation \cite{BTX13}.
 While this is not true for the Gaussons with small relative velocity, in which case additional solitary waves are emitted after collision in the LogSE.

\begin{figure}[htbp]
\begin{center}
\includegraphics[width=2.2in,height=1.7in]{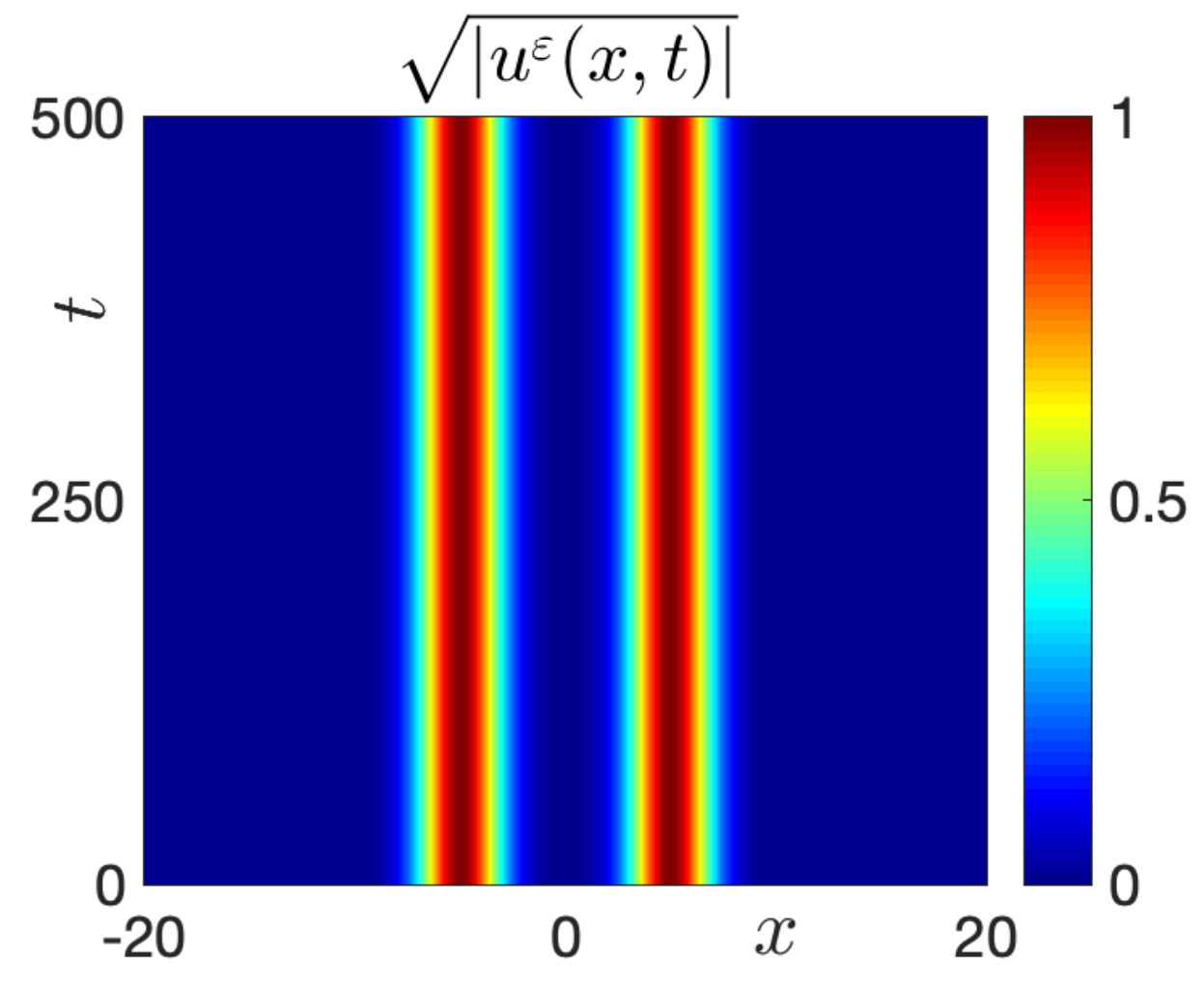}
\hspace{0.2cm}
\includegraphics[width=1.7in,height=1.55in]{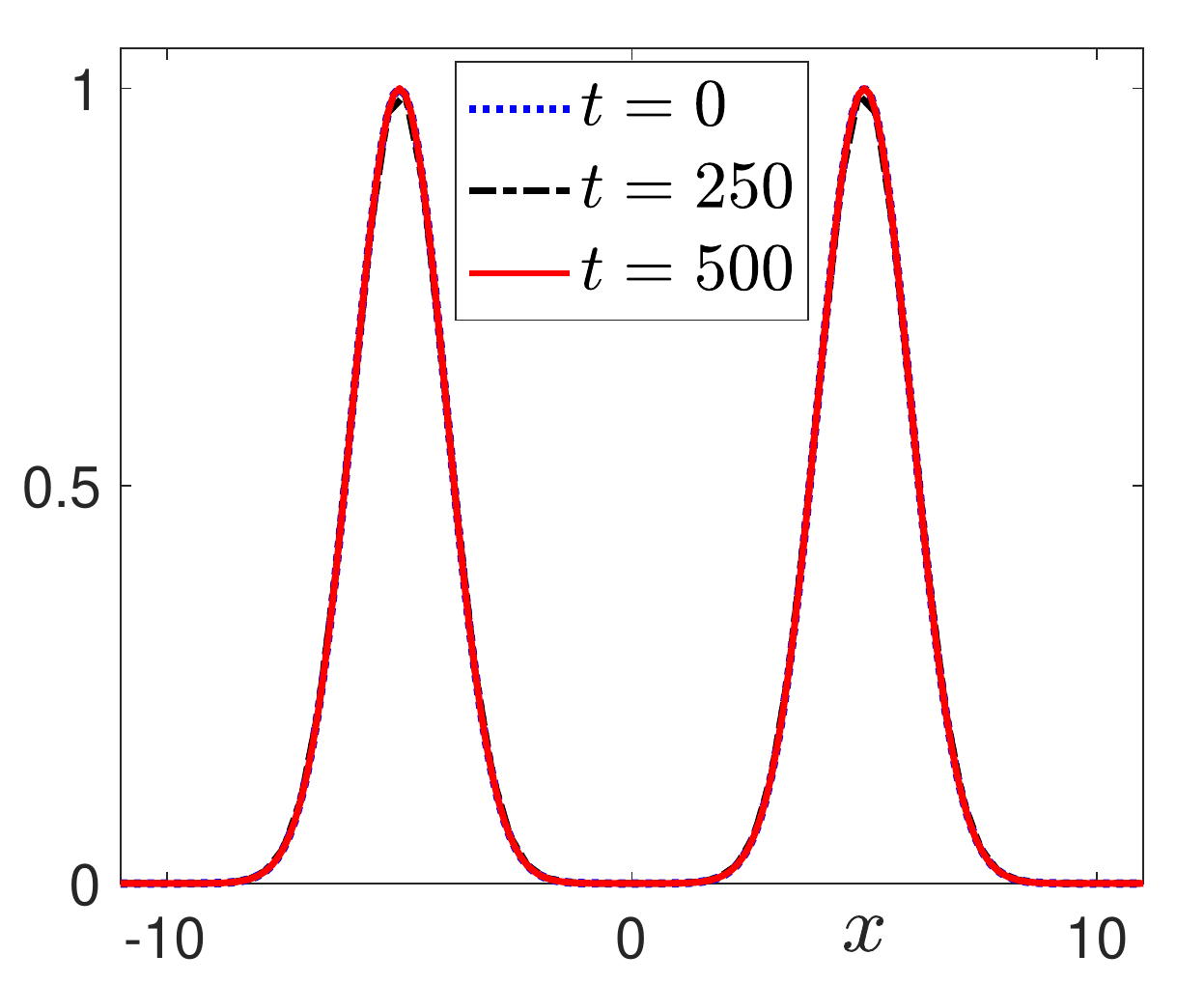}
\hspace{0.2cm}
\includegraphics[width=1.7in,height=1.55in]{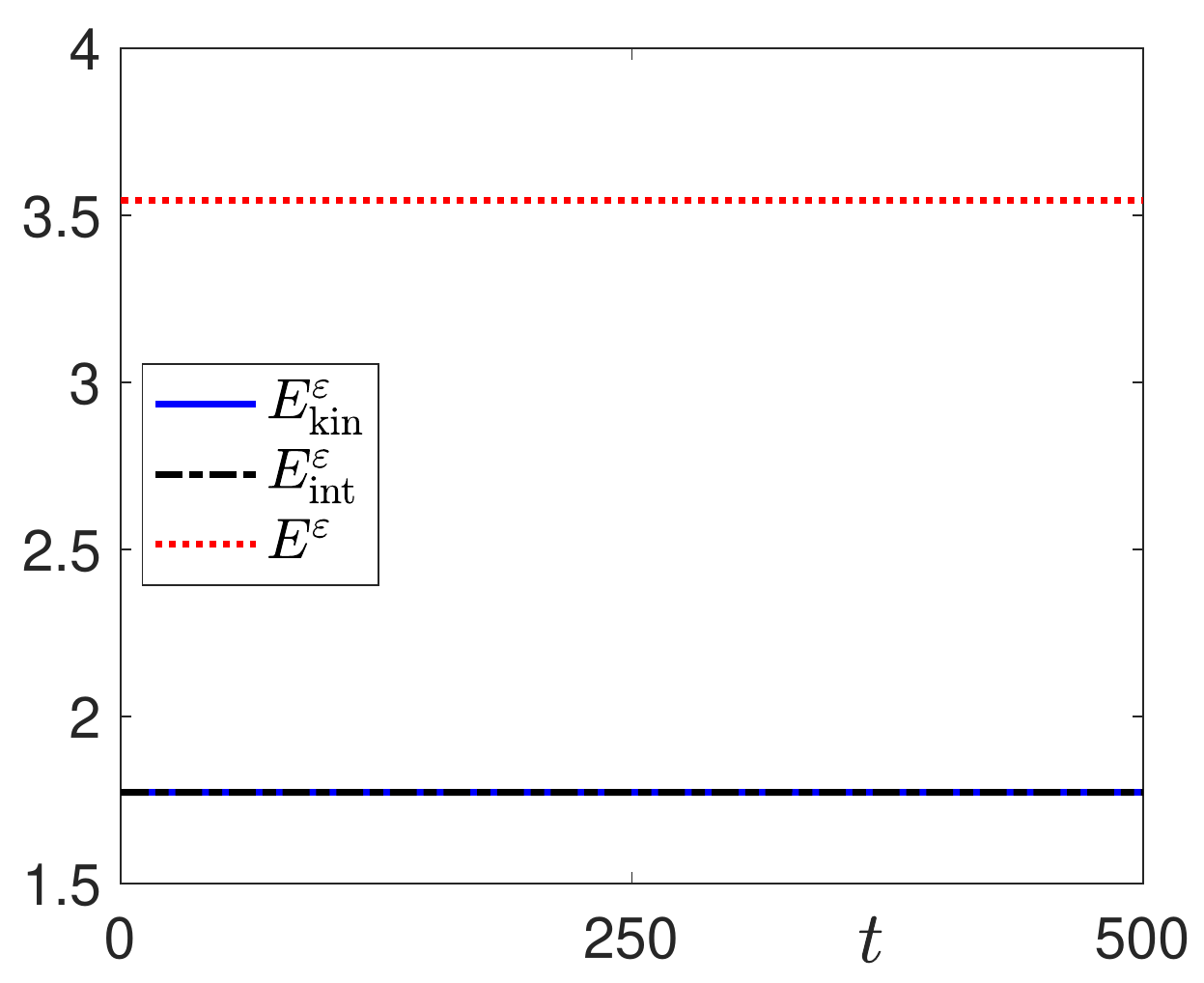}\\
\vspace{0.3cm}
\includegraphics[width=2.2in,height=1.7in]{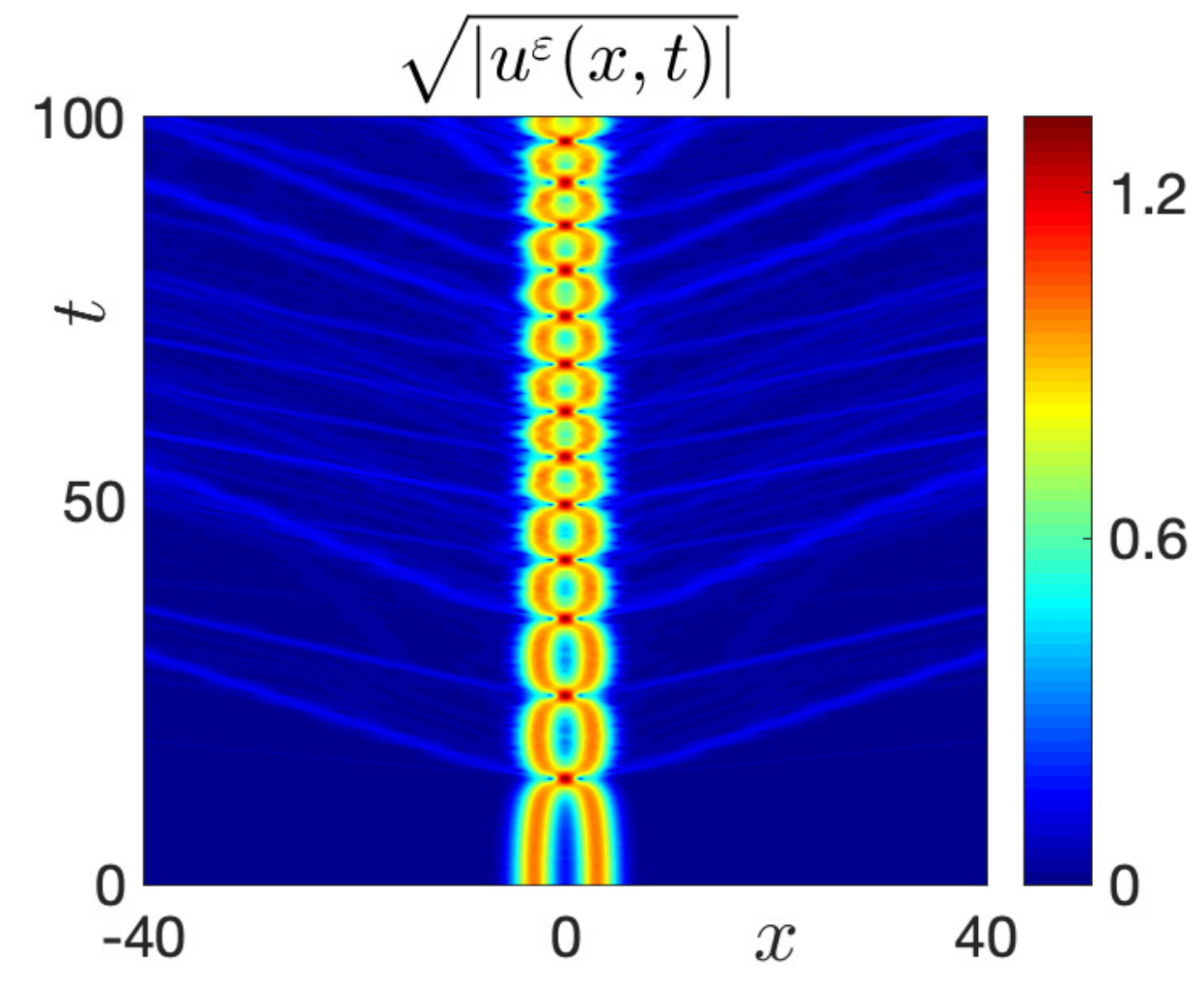}
\hspace{0.2cm}
\includegraphics[width=1.7in,height=1.55in]{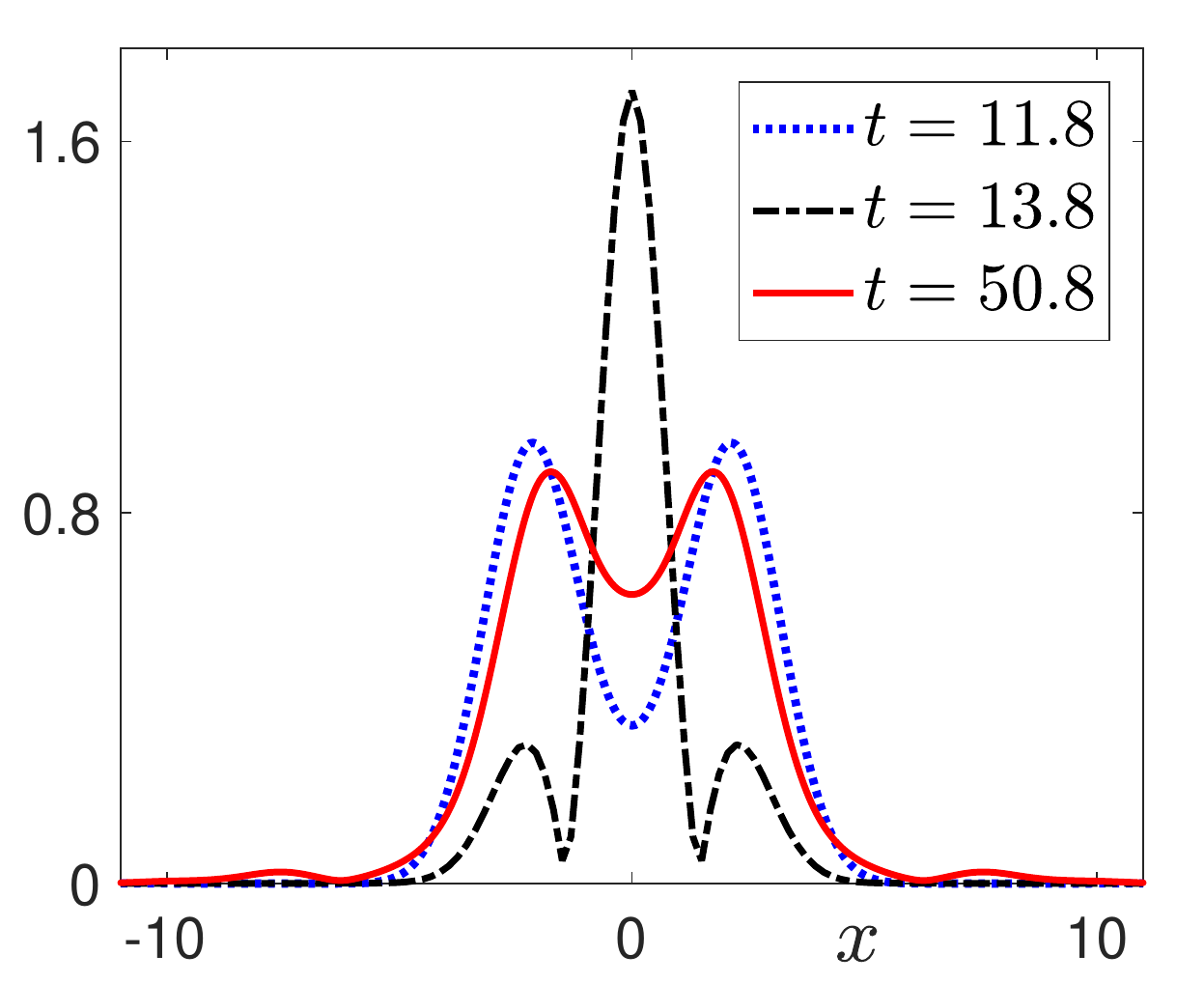}
\hspace{0.2cm}
\includegraphics[width=1.7in,height=1.55in]{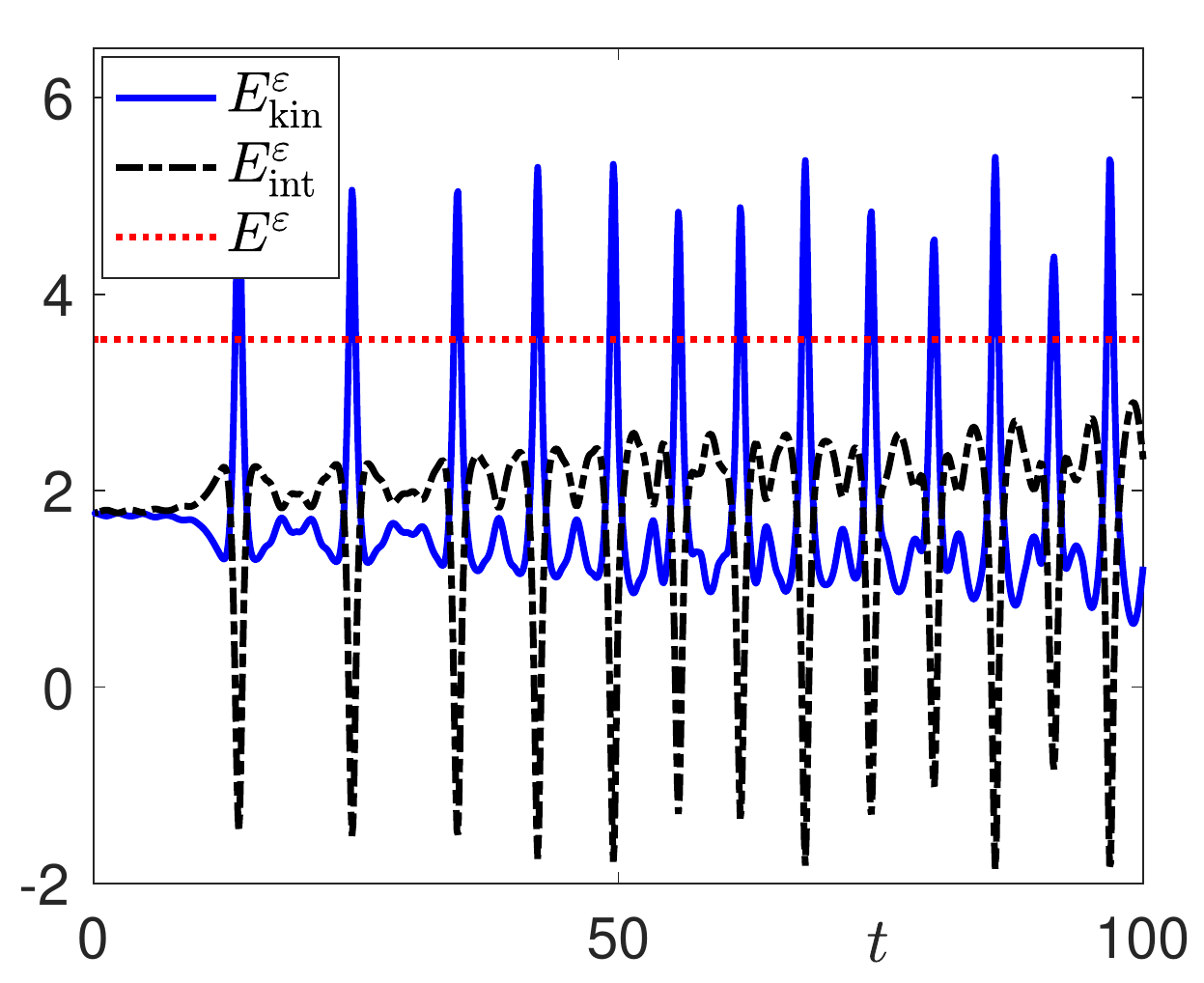}\\
\vspace{0.3cm}
\includegraphics[width=2.2in,height=1.7in]{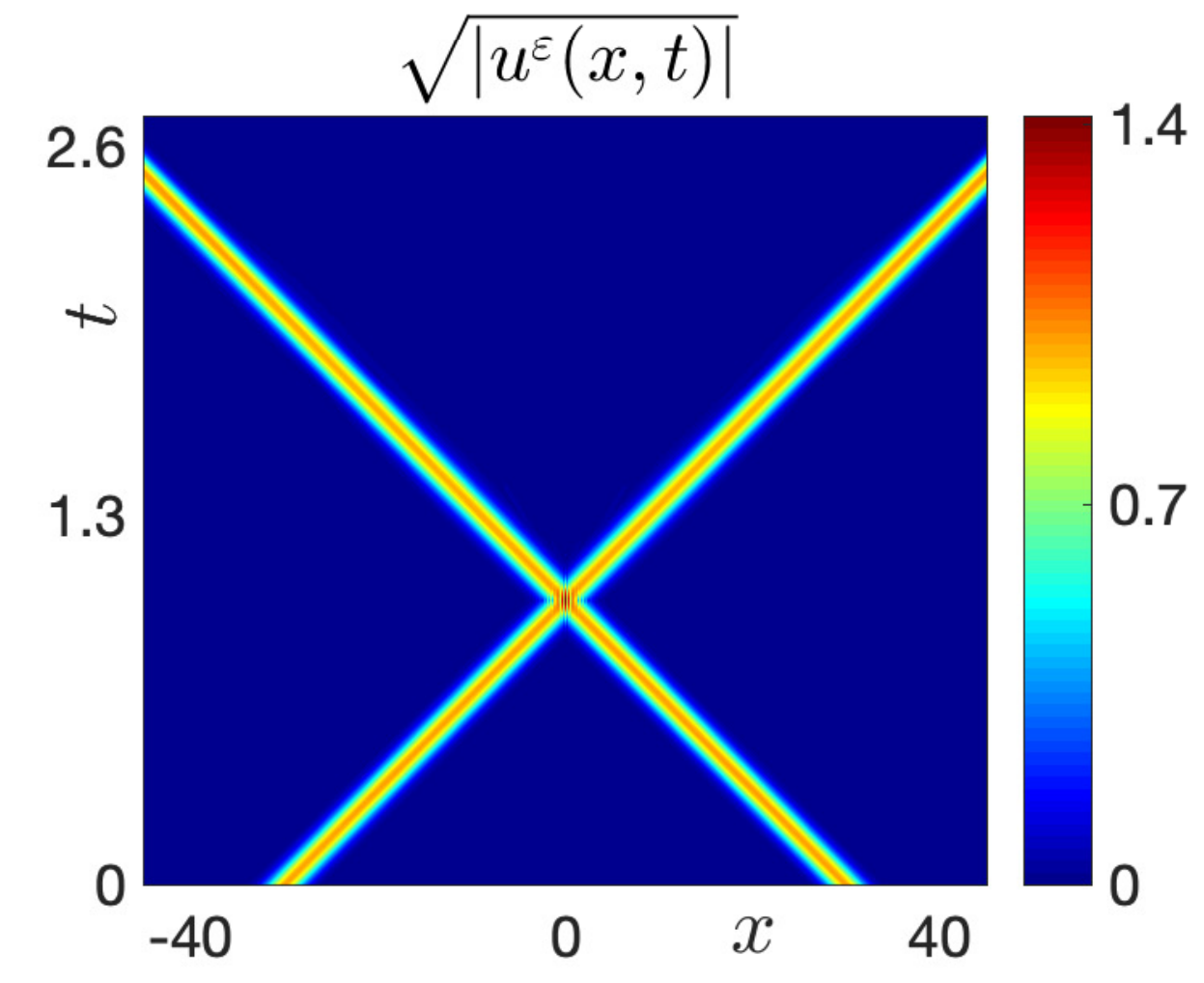}
\hspace{0.2cm}
\includegraphics[width=1.7in,height=1.55in]{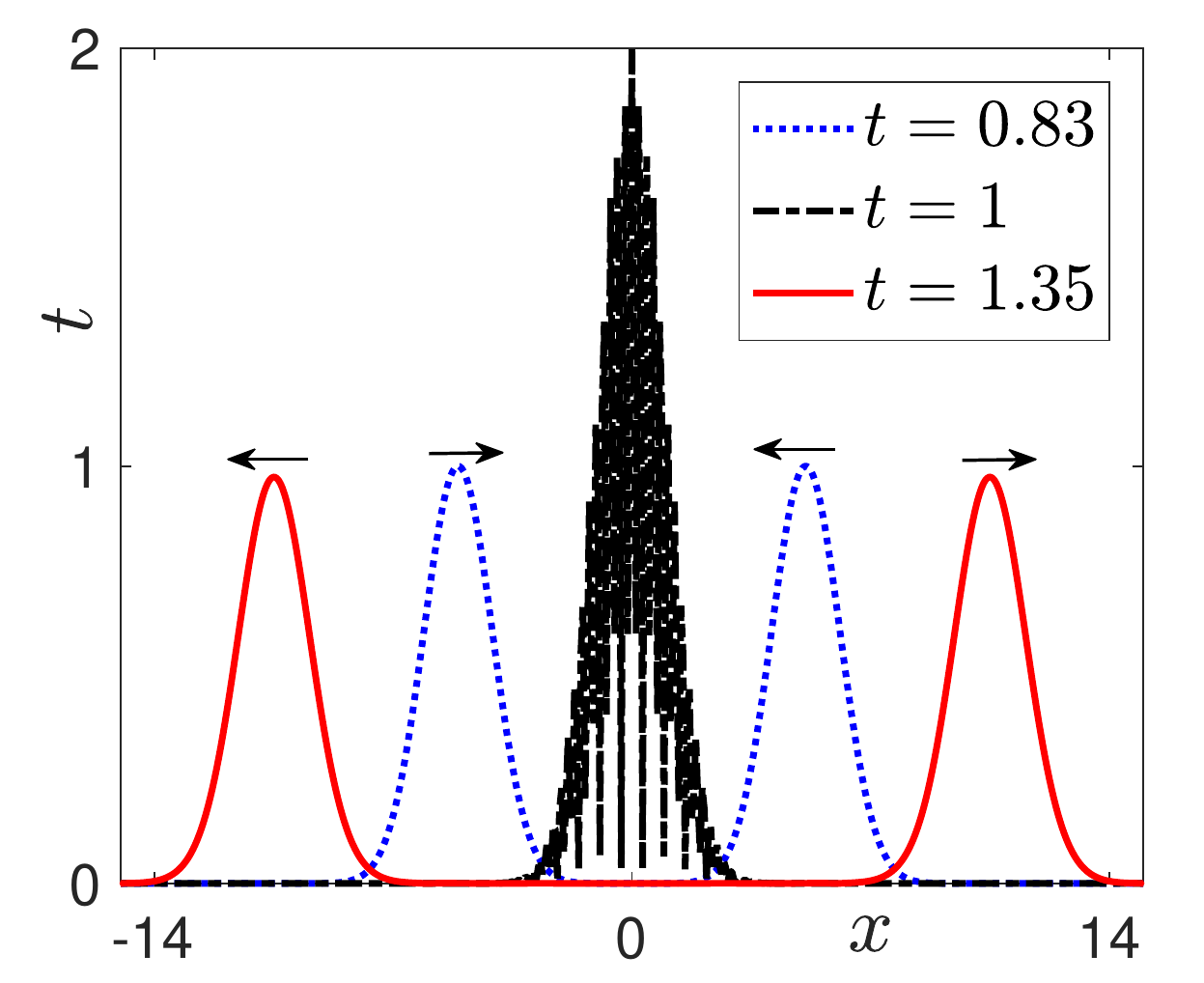}
\hspace{0.2cm}
\includegraphics[width=1.7in,height=1.55in]{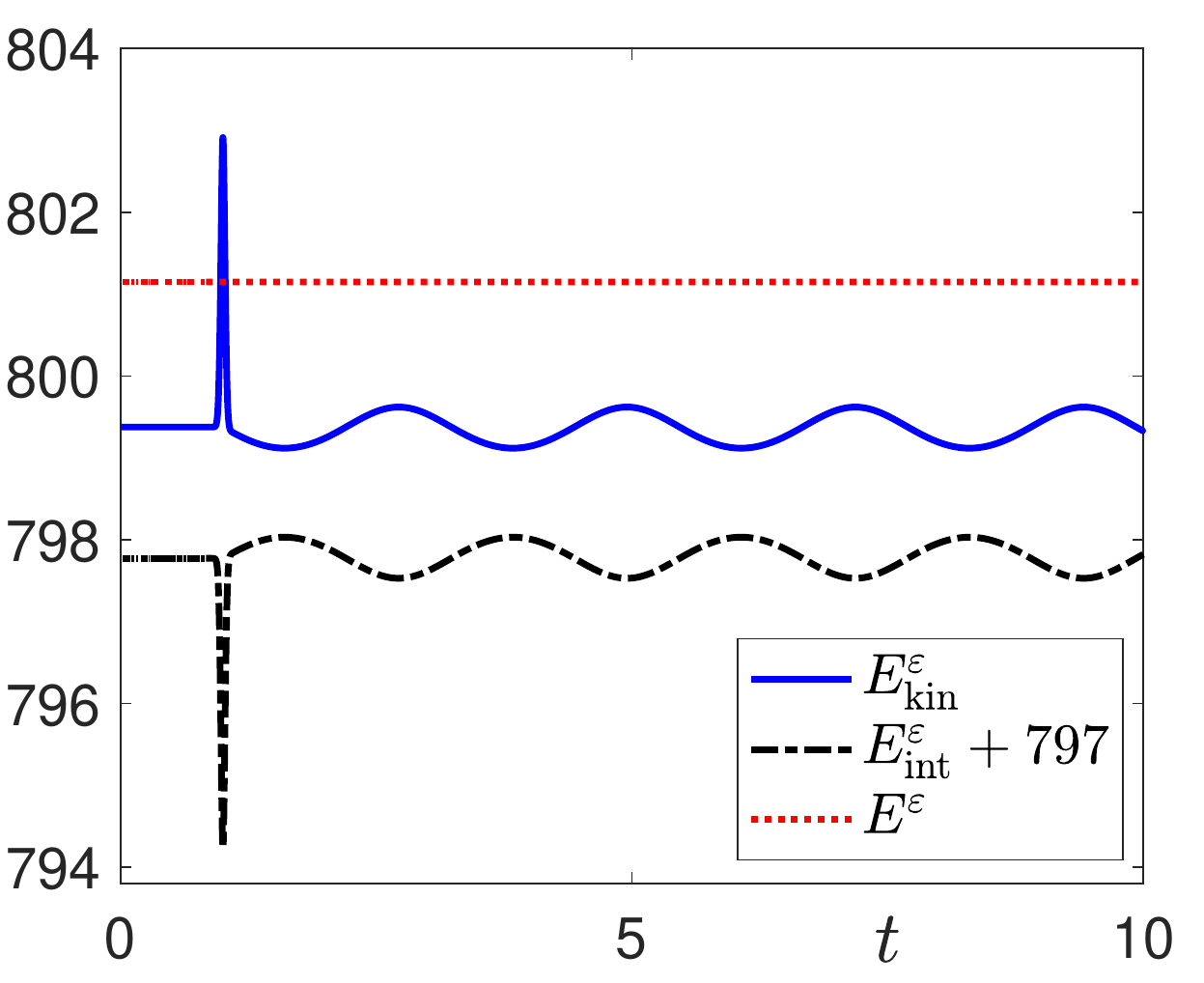}\\
\vspace{0.3cm}
\includegraphics[width=2.2in,height=1.7in]{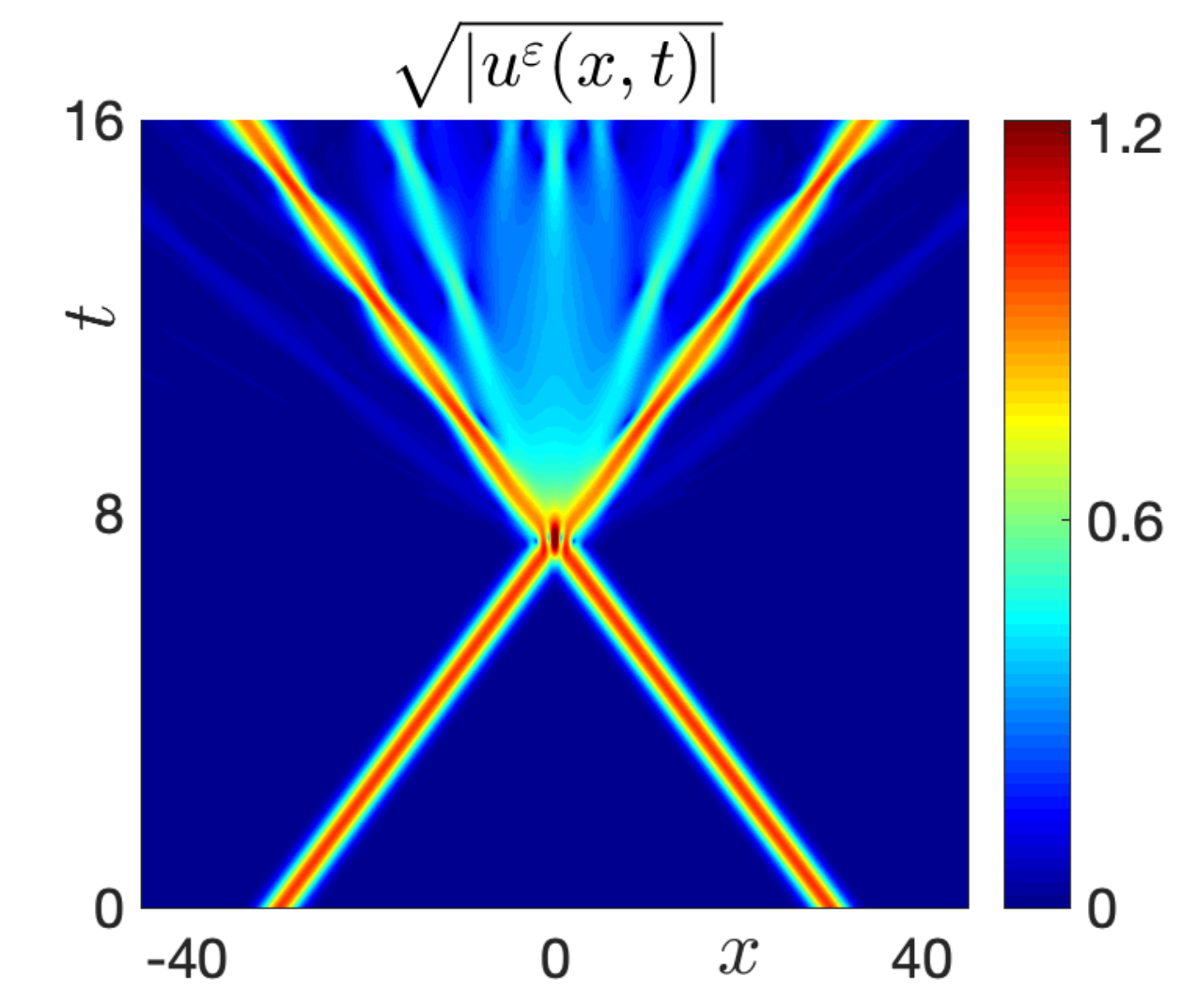}
\hspace{0.2cm}
\includegraphics[width=1.7in,height=1.55in]{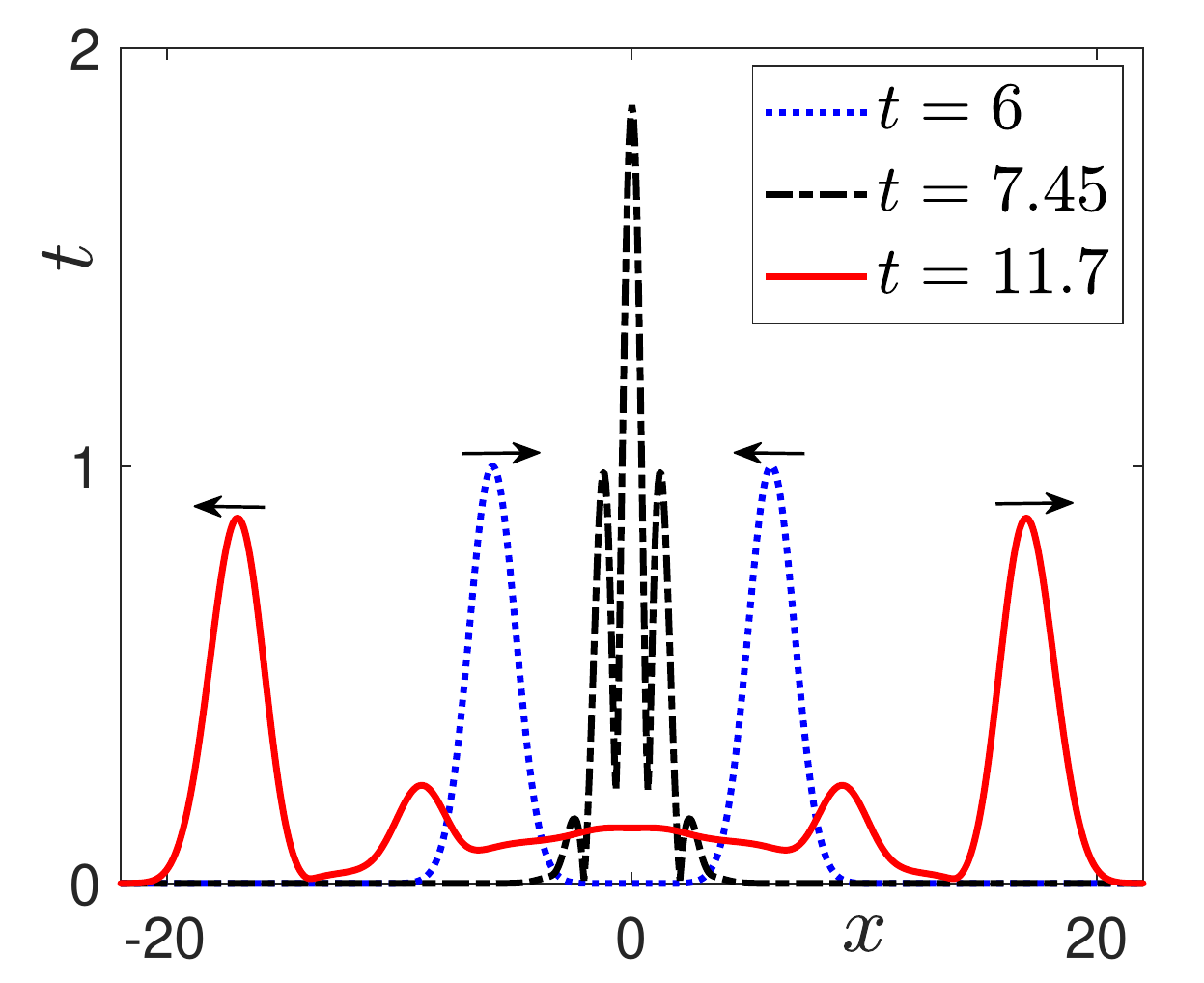}
\hspace{0.2cm}
\includegraphics[width=1.7in,height=1.55in]{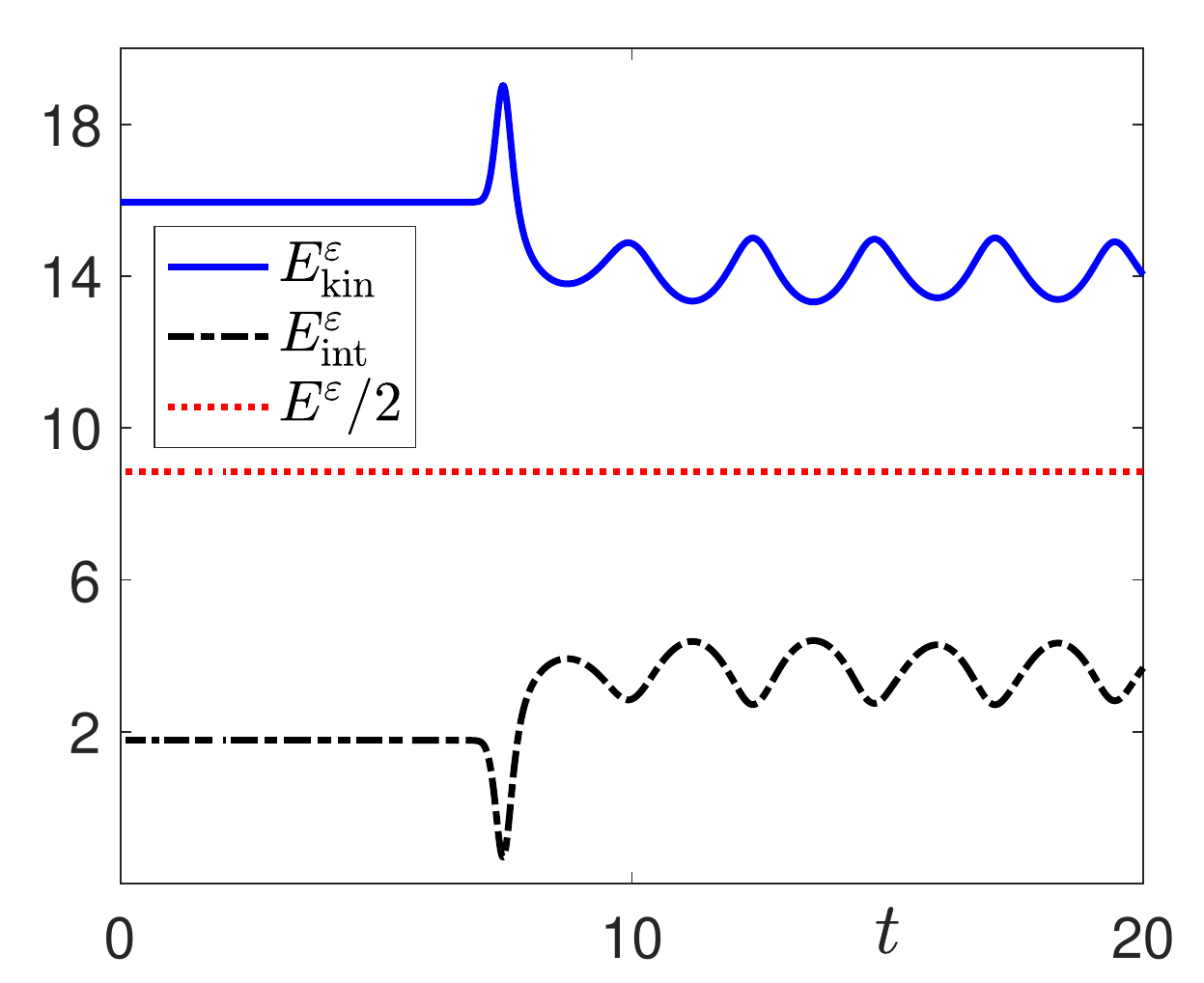}
\end{center}
\caption{Plots of $\sqrt{|u^\ep(x,t)|}$ (first column), $|u^\ep(x,t)|$ at different time (second column) and evolution of the energies (third column) for
 different parameters in {\bf Example 3}:  Case i--Case iv (from top to bottom).  }
\label{fig:ex3-casei-iv}
\end{figure}

\begin{figure}[htbp]
\begin{center}
\includegraphics[width=2.2in,height=1.7in]{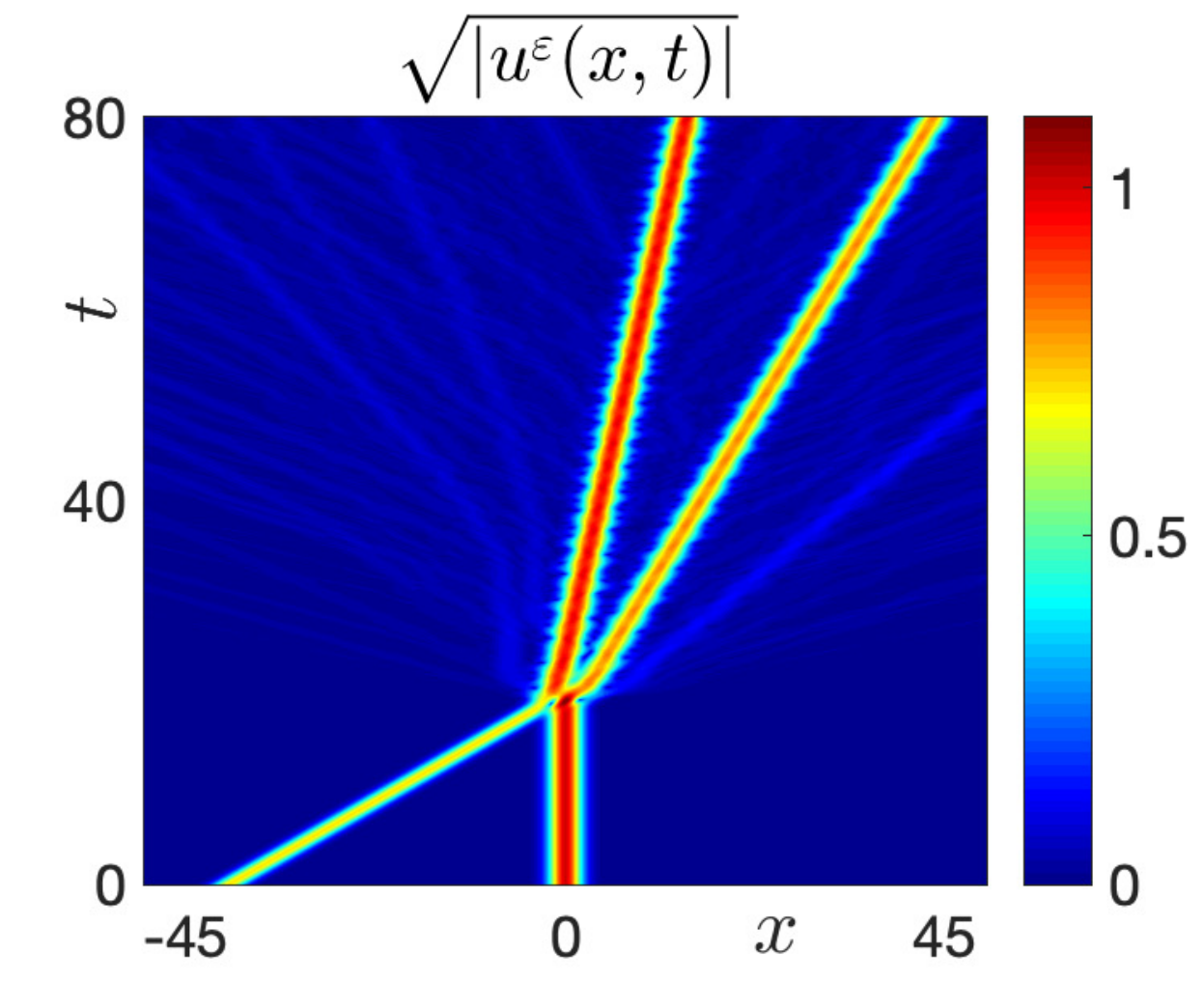}
\hspace{0.2cm}
\includegraphics[width=1.7in,height=1.55in]{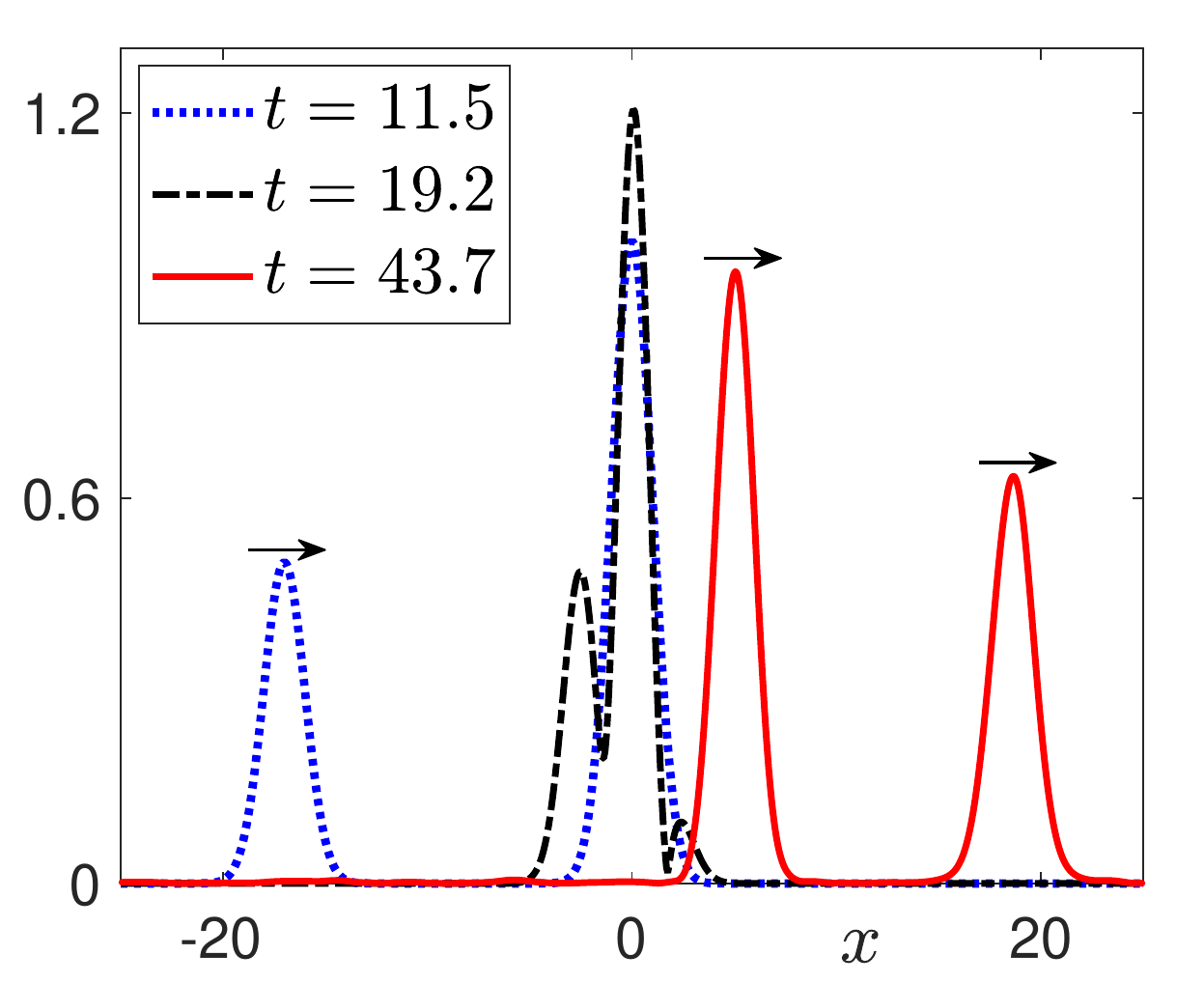}
\hspace{0.2cm}
\includegraphics[width=1.7in,height=1.55in]{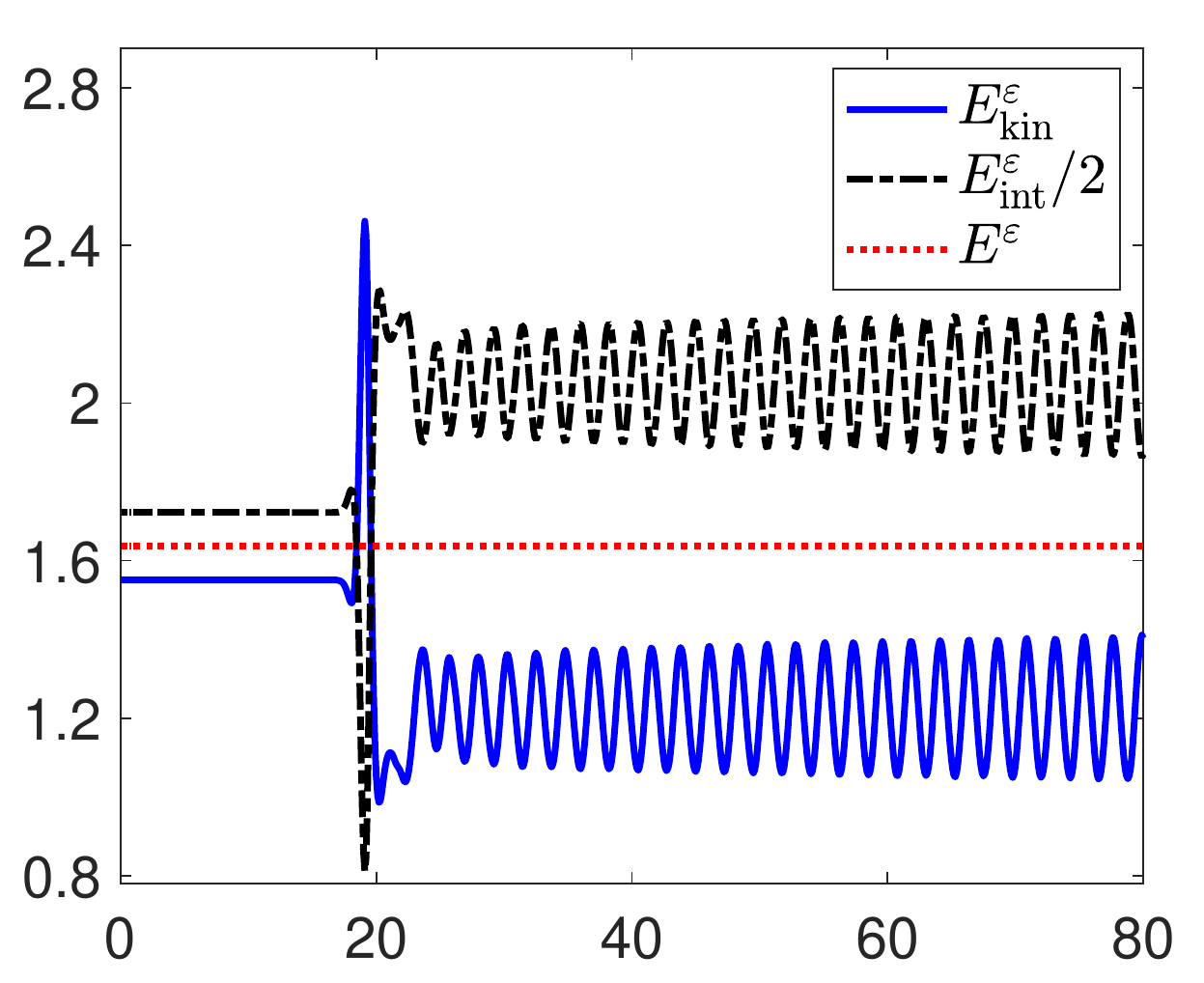}\\
\vspace{0.2cm}
\includegraphics[width=2.2in,height=1.7in]{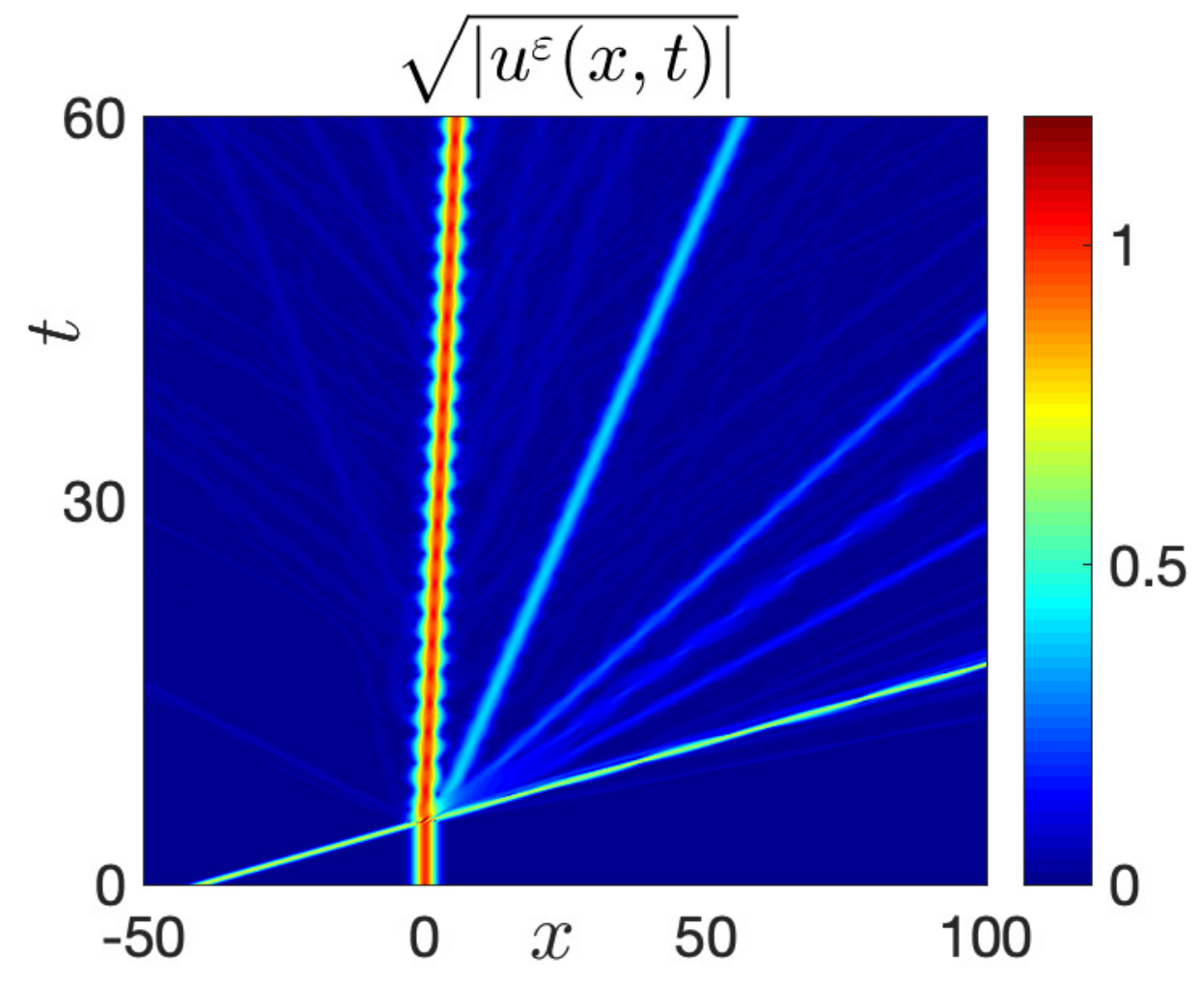}
\hspace{0.2cm}
\includegraphics[width=1.7in,height=1.55in]{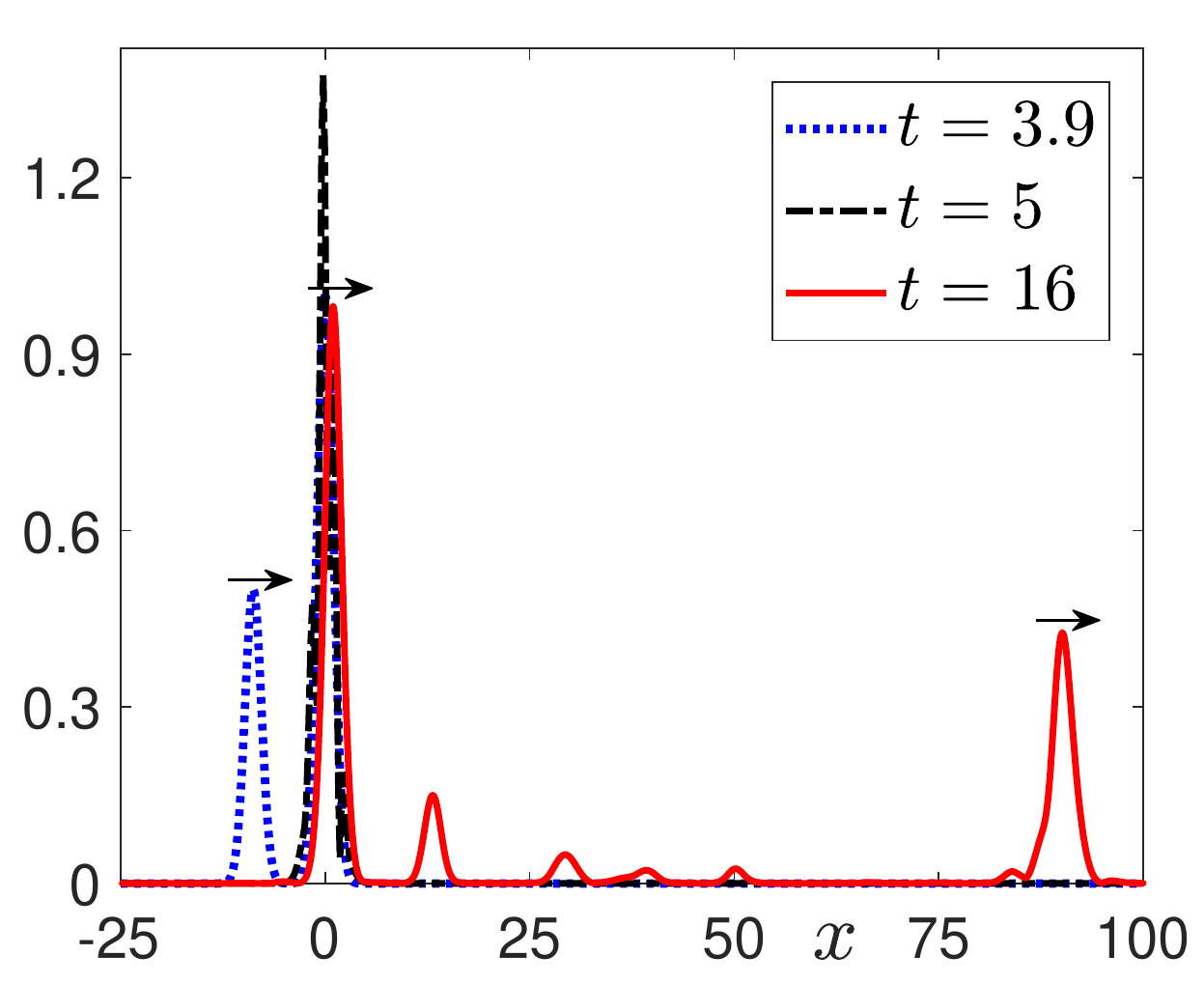}
\hspace{0.2cm}
\includegraphics[width=1.7in,height=1.55in]{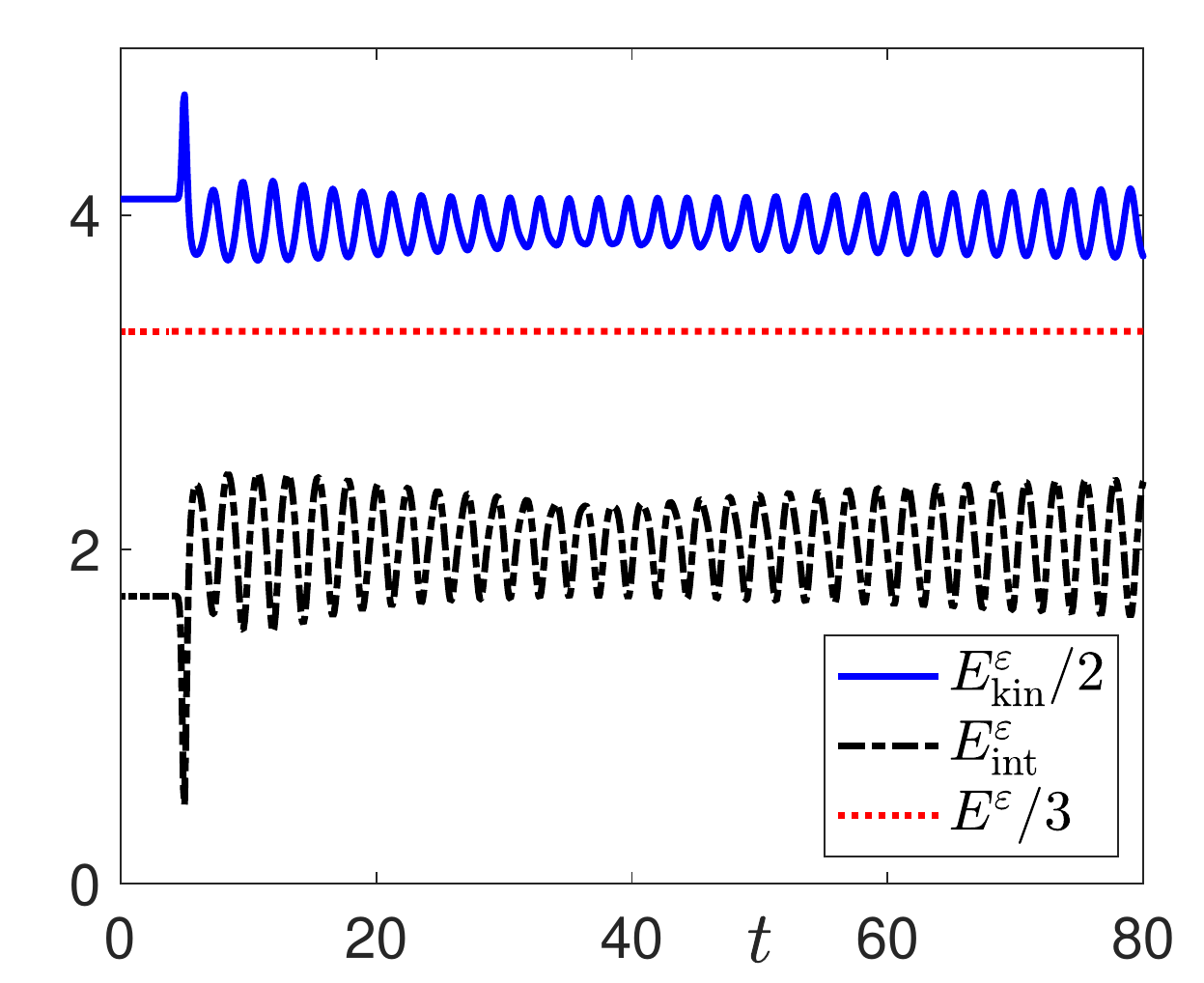}\\
\vspace{0.2cm}
\includegraphics[width=2.2in,height=1.7in]{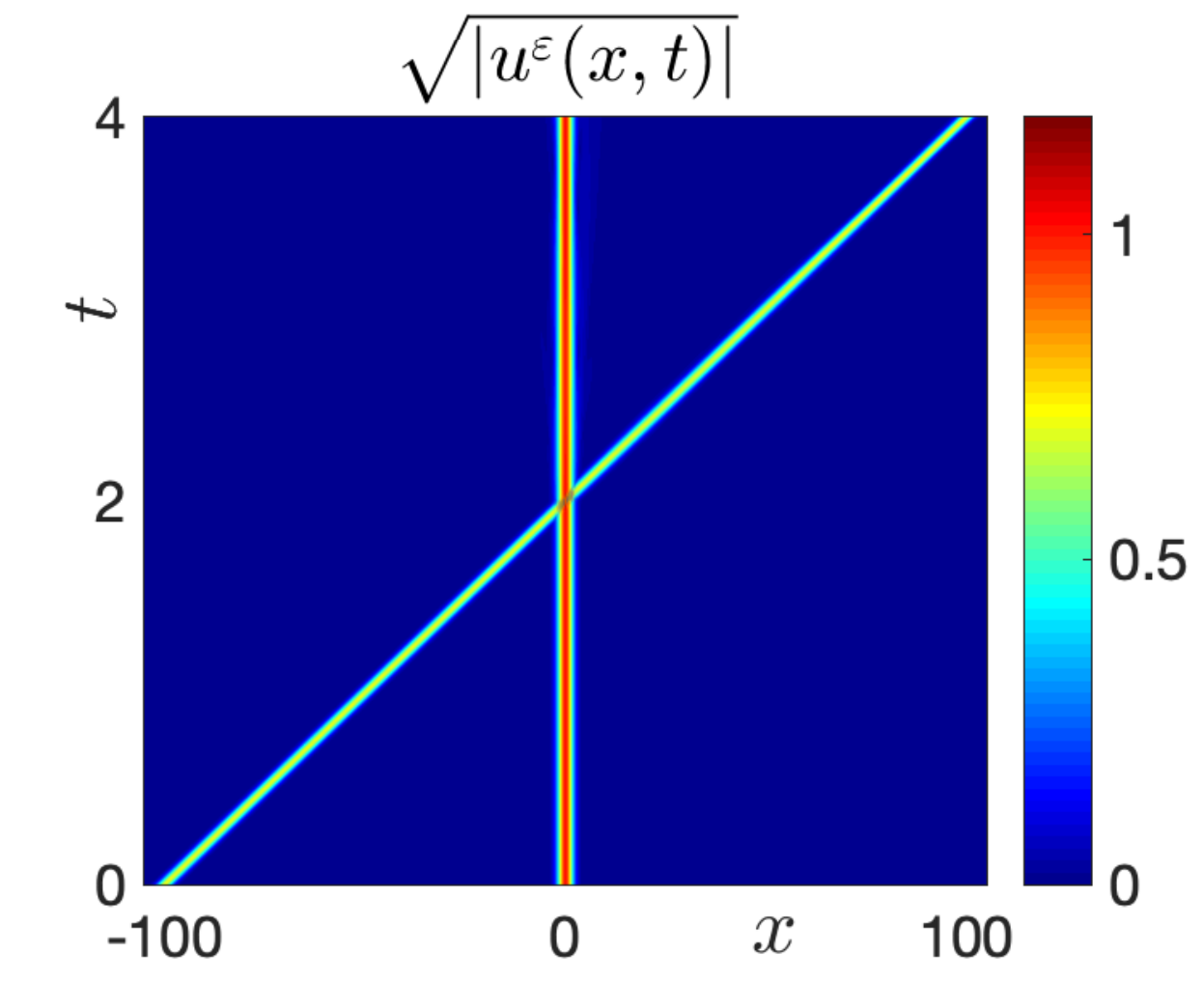}
\hspace{0.2cm}
\includegraphics[width=1.7in,height=1.55in]{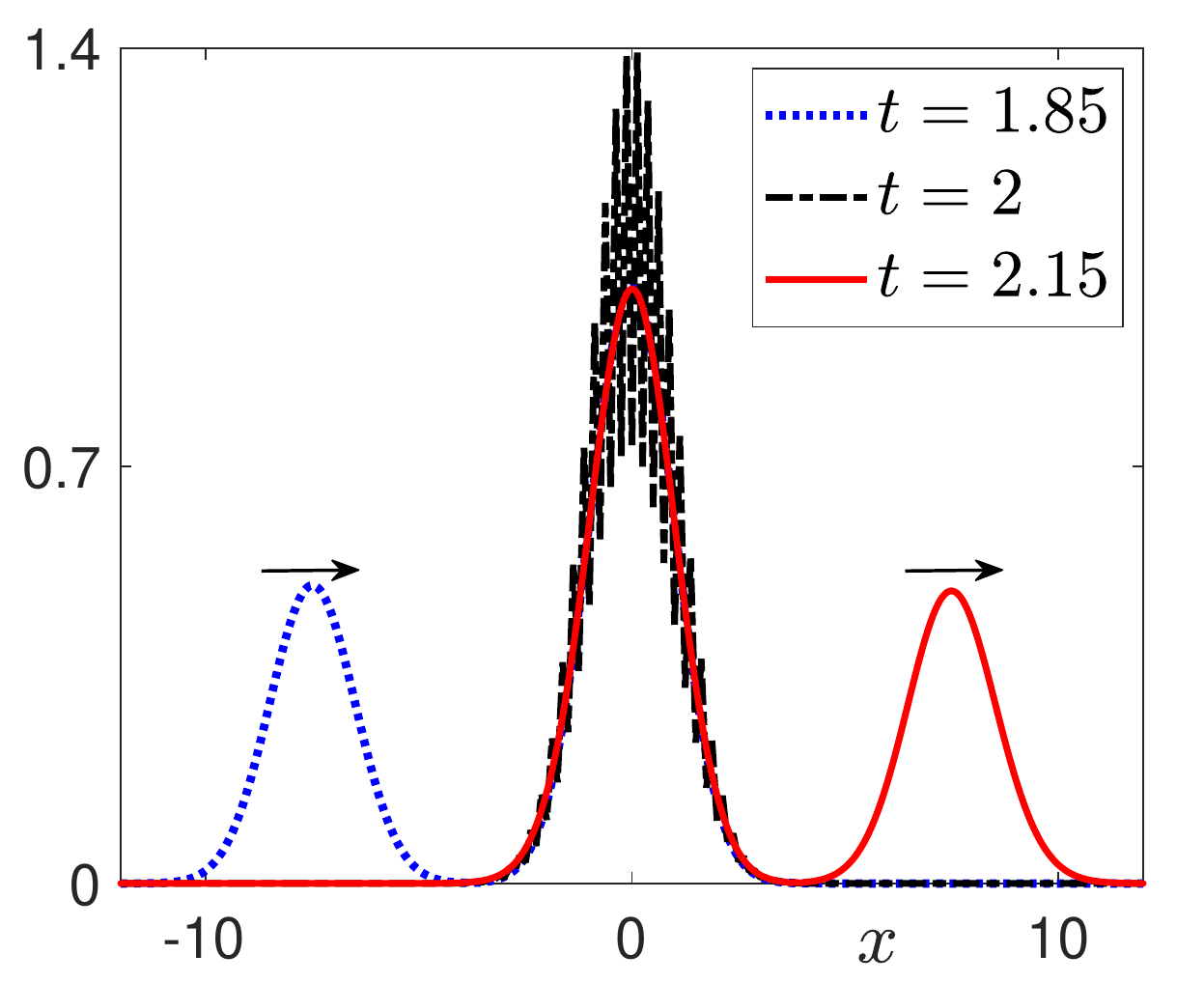}
\hspace{0.2cm}
\includegraphics[width=1.7in,height=1.55in]{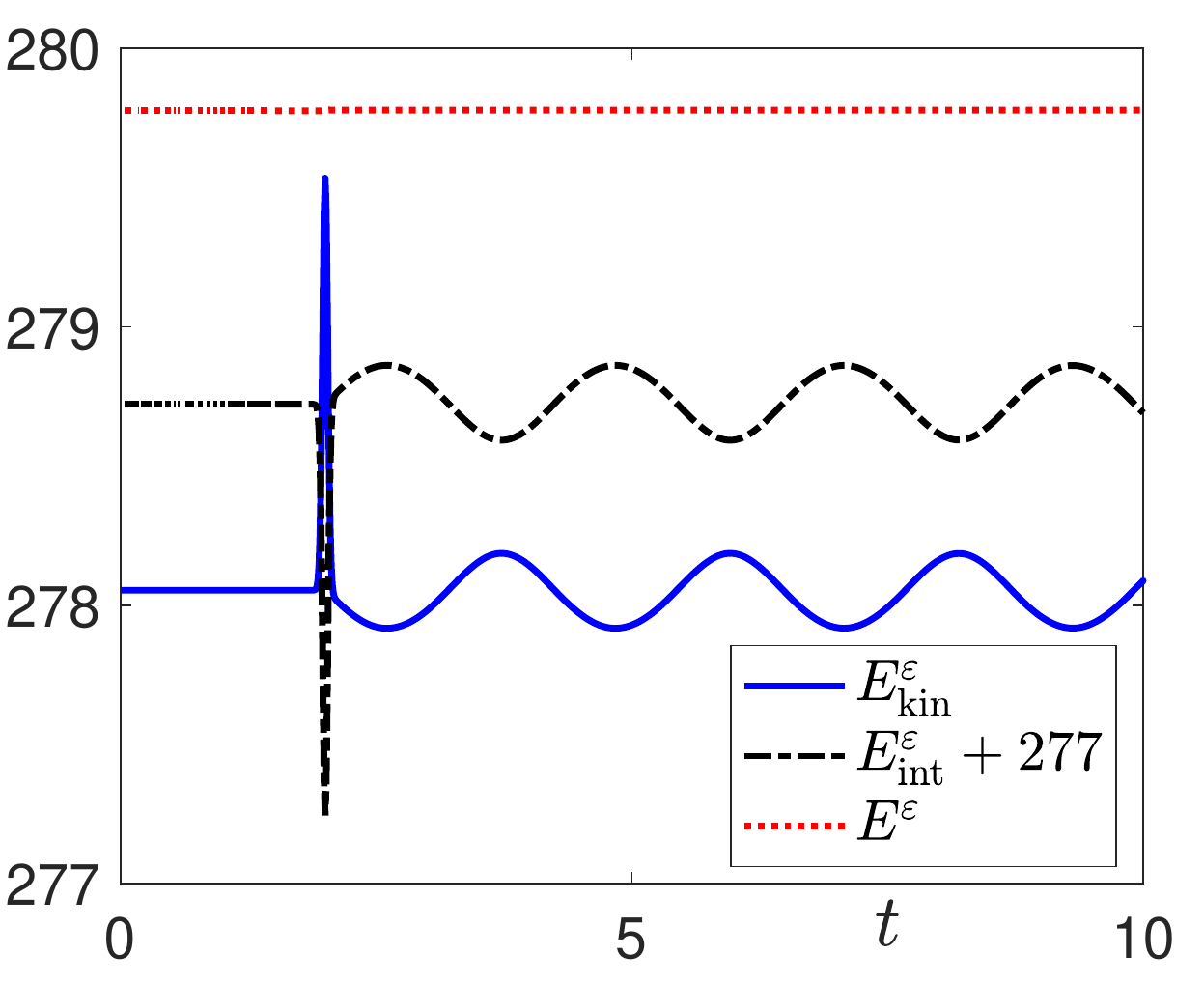}\\
\vspace{0.2cm}
\includegraphics[width=2.2in,height=1.7in]{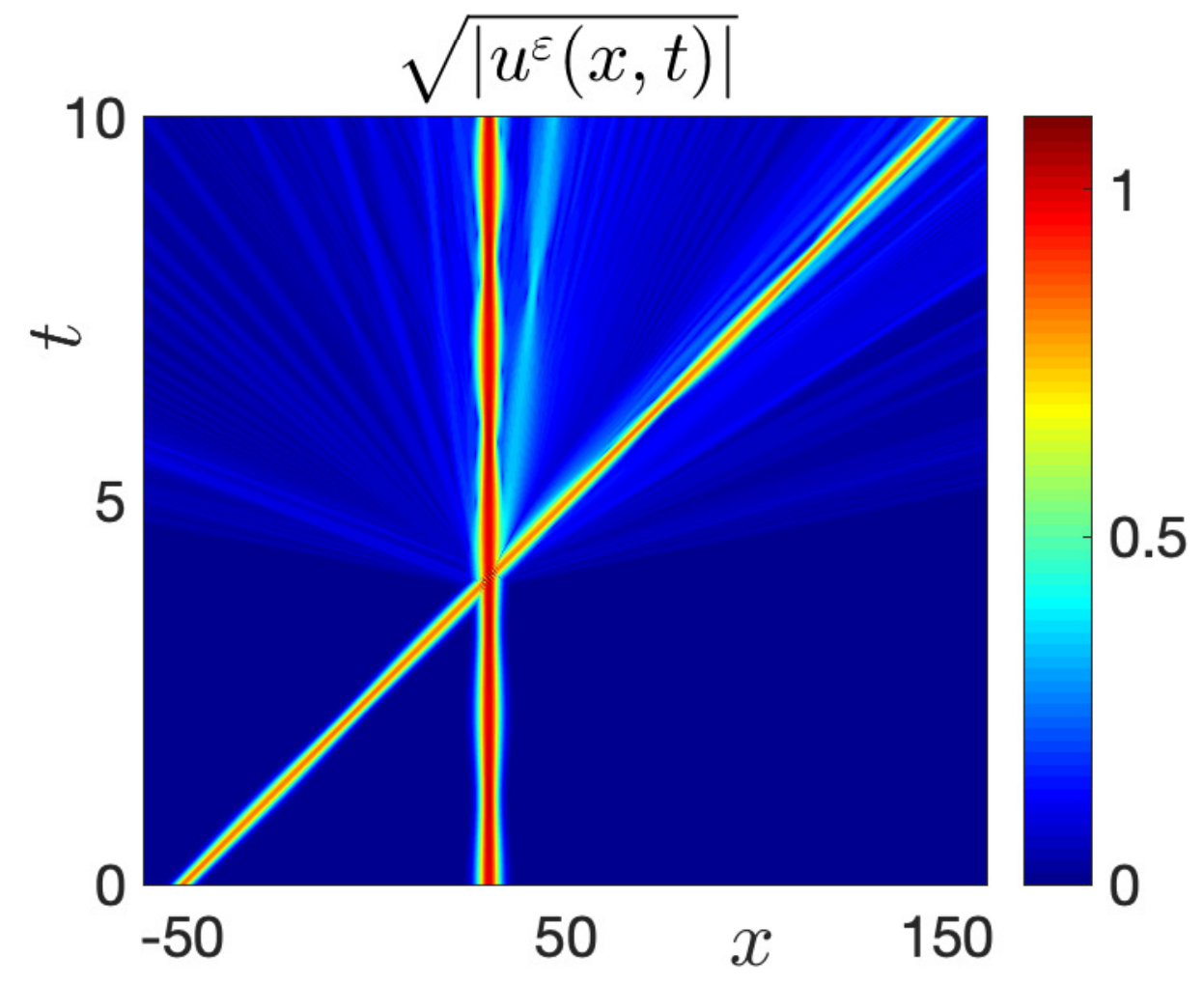}
\hspace{0.2cm}
\includegraphics[width=1.7in,height=1.55in]{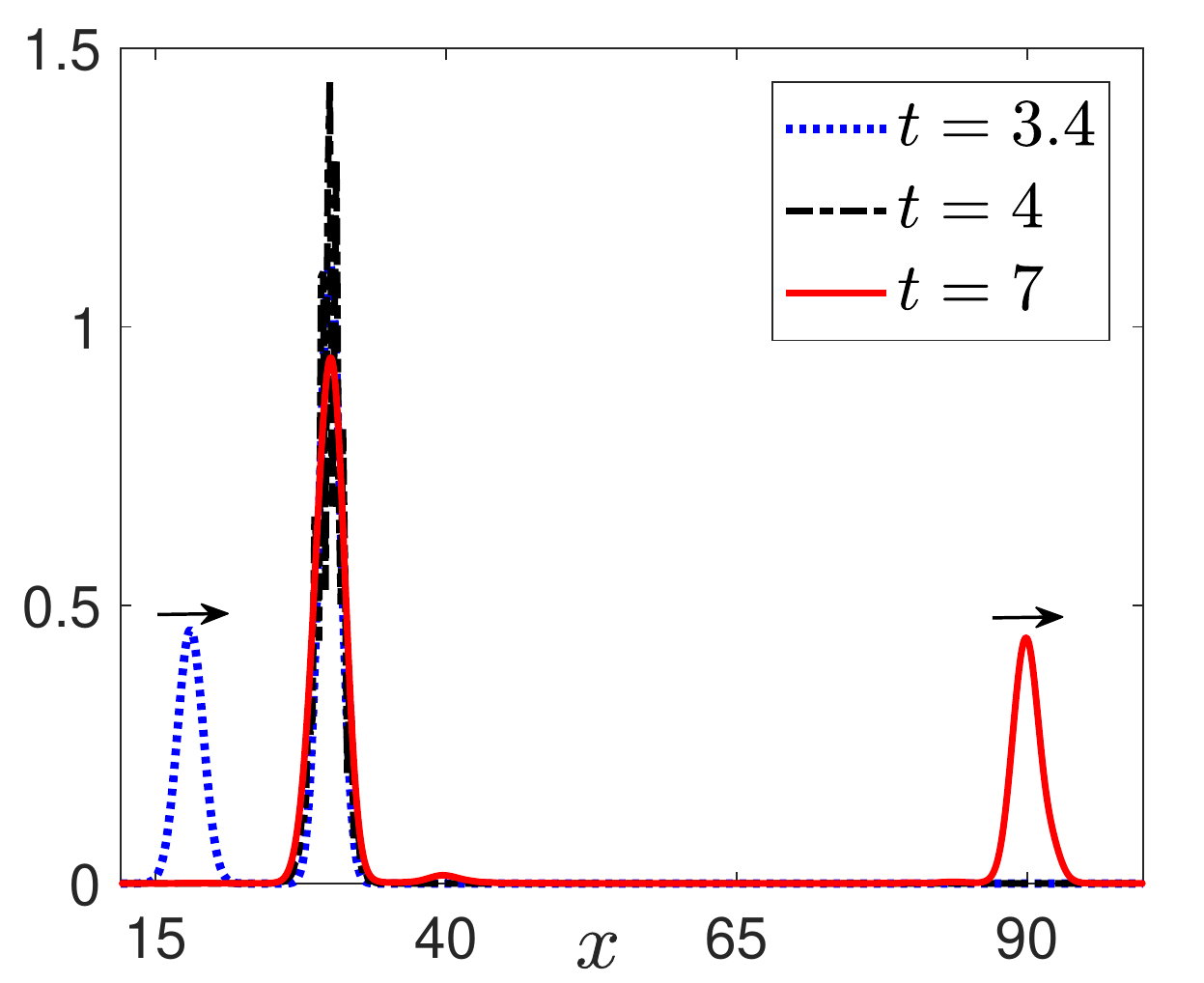}
\hspace{0.2cm}
\includegraphics[width=1.7in,height=1.55in]{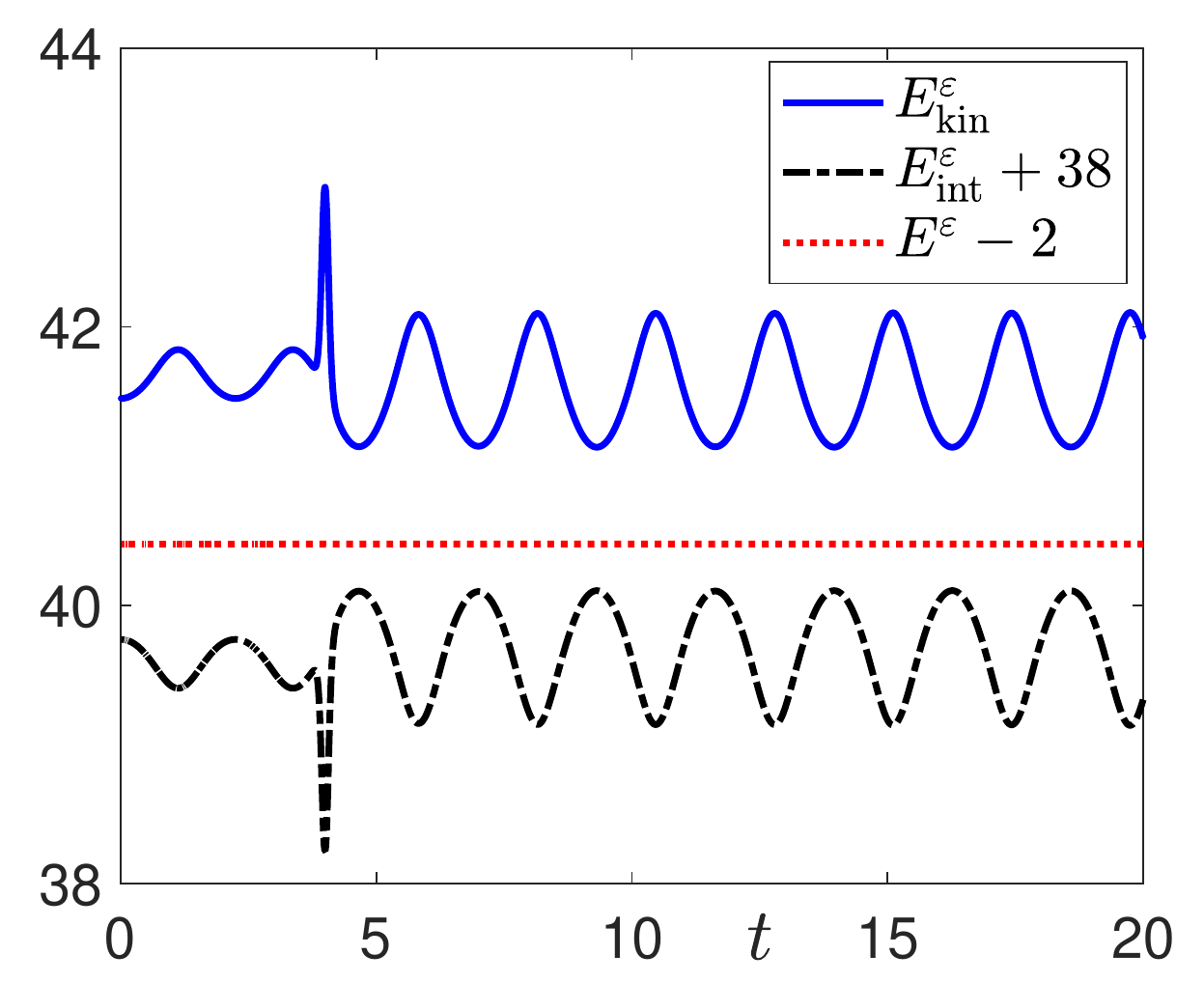}
\end{center}
\caption{Plots of $\sqrt{|u^\ep(x,t)|}$ (first column), $|u^\ep(x,t)|$ at different time (second column) and evolution of the energies (third column)
for different parameters in {\bf Example 3}:  Case v--Case viii (from top to bottom).  }
\label{fig:ex3-casev-viii}
\end{figure}

\bigskip

{\bf Example 4}. Here, we let $\lambda=1$ and $L=10000$. We consider the following three cases of parameters in \eqref{ini_set}:

\begin{itemize}
\item[](ix). $N=1$, $v_1=10$,   $x_1=10$,  $b_1=1$,  $a_1=1$;
\item[](x).  $N=2$, $v_1=10$, $v_2=0$,    $x_1=-100$, $x_2=0$,  $b_1=2$, $b_2=1$,  $a_1=a_2=1$;
\item[](xi). $N=2$, $v_1=20$, $v_2=0$,    $x_1=-100$, $x_2=0$,  $b_1=2$, $b_2=1$,  $a_1=a_2=1$.
\end{itemize}

\noindent Fig. \ref{fig:ex4-caseix-xi} illustrates the evolution of $\sqrt{|u^\ep(x,t)|}$, $E^\varepsilon_{\rm kin}$,
 $E^\varepsilon_{\rm int}$ and $E^\varepsilon$ as well as the plot of $|u^\ep(x,t)|$ at different time for
 Cases ix-xi.  We would conclude  from these figures and other numerical experiments not shown here for
brevity that:
(1) The total energy is conserved well.
(2) Unlike the case of $\lambda<0$ where the Gaussons behave like solitary waves,
the Gaussians in the case  $\lambda>0$  move and spread  out (cf. Fig. \ref{fig:ex4-caseix-xi}).
In fact, for a single Gaussian, the analytical solution $u(x,t)$ is given in \eqref{Gaus}. Fig. \ref{fig:ex4-caseix-error} shows the errors of $e(t):=u^\varepsilon(\cdot,t)-u(\cdot,t)$ measured in different norms, which again evidence  the accuracy of the STSP scheme.
In addition, the rate of dispersion of the Gaussians could indeed be estimated for large time dynamics in \cite{CaGa-p}.
(3) The dynamics and interaction of two moving Gaussians depend on the
relative velocity.  They will be separated completely and no solitary
waves are emitted if the relative velocity is large enough
(cf. Fig. \ref{fig:ex4-caseix-xi} Case xi). While if  the relative
velocity is not large enough, i.e., when they move more slowly than
the  speed they spread out, the Gaussians will be  partially twisted
together. Oscillation is  created and always there
(cf. Fig. \ref{fig:ex4-caseix-xi} Case x). This is consistent with the
fact that the convergence to a universal Gaussian profile (leaving out
the oscillatory aspects, which are not described in general) is very
slow, as established in \cite{CaGa-p} (logarithmic
convergence in time).

\begin{figure}[htbp]
\begin{center}
\includegraphics[width=2.2in,height=1.7in]{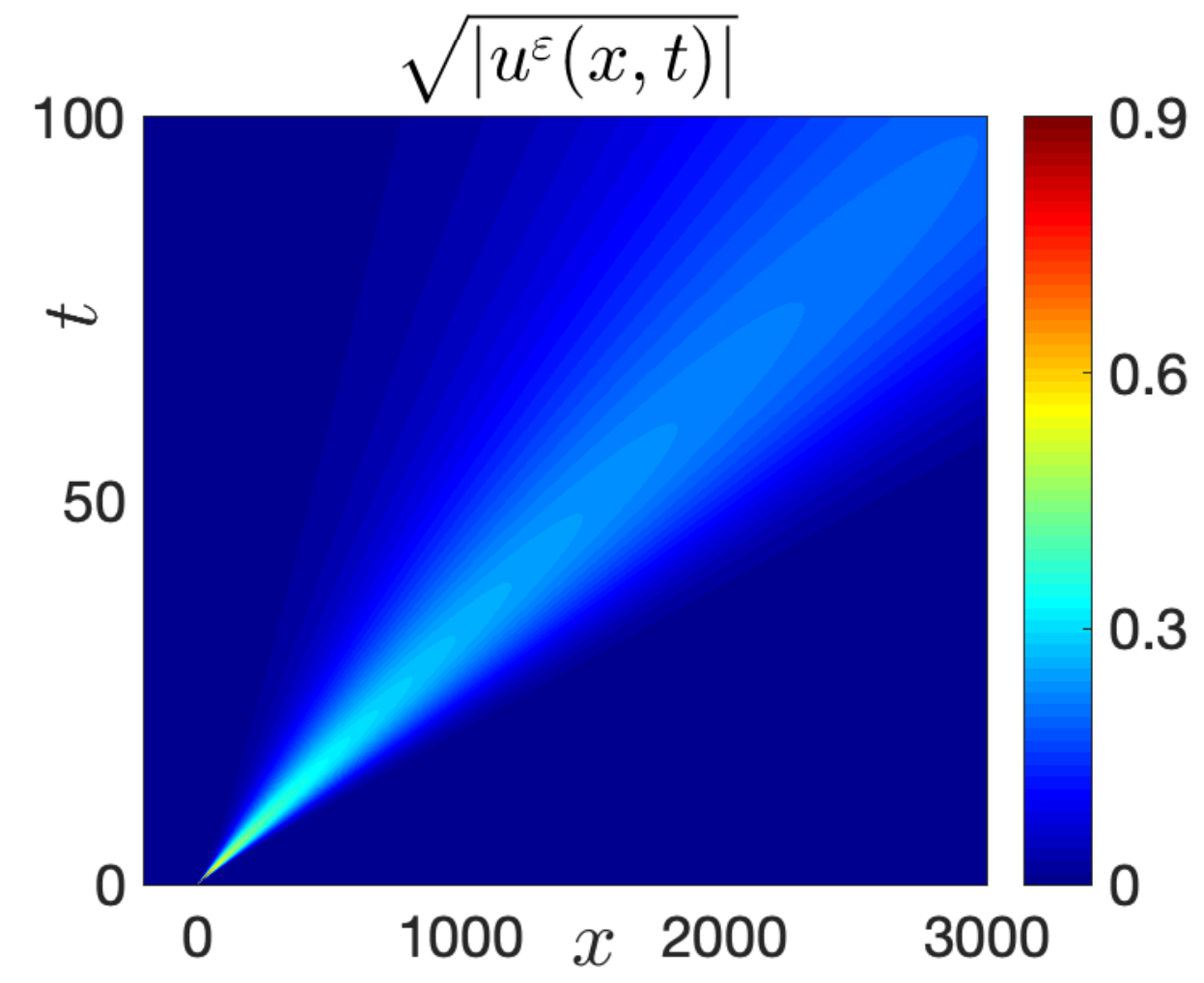}
\hspace{0.2cm}
\includegraphics[width=1.7in,height=1.55in]{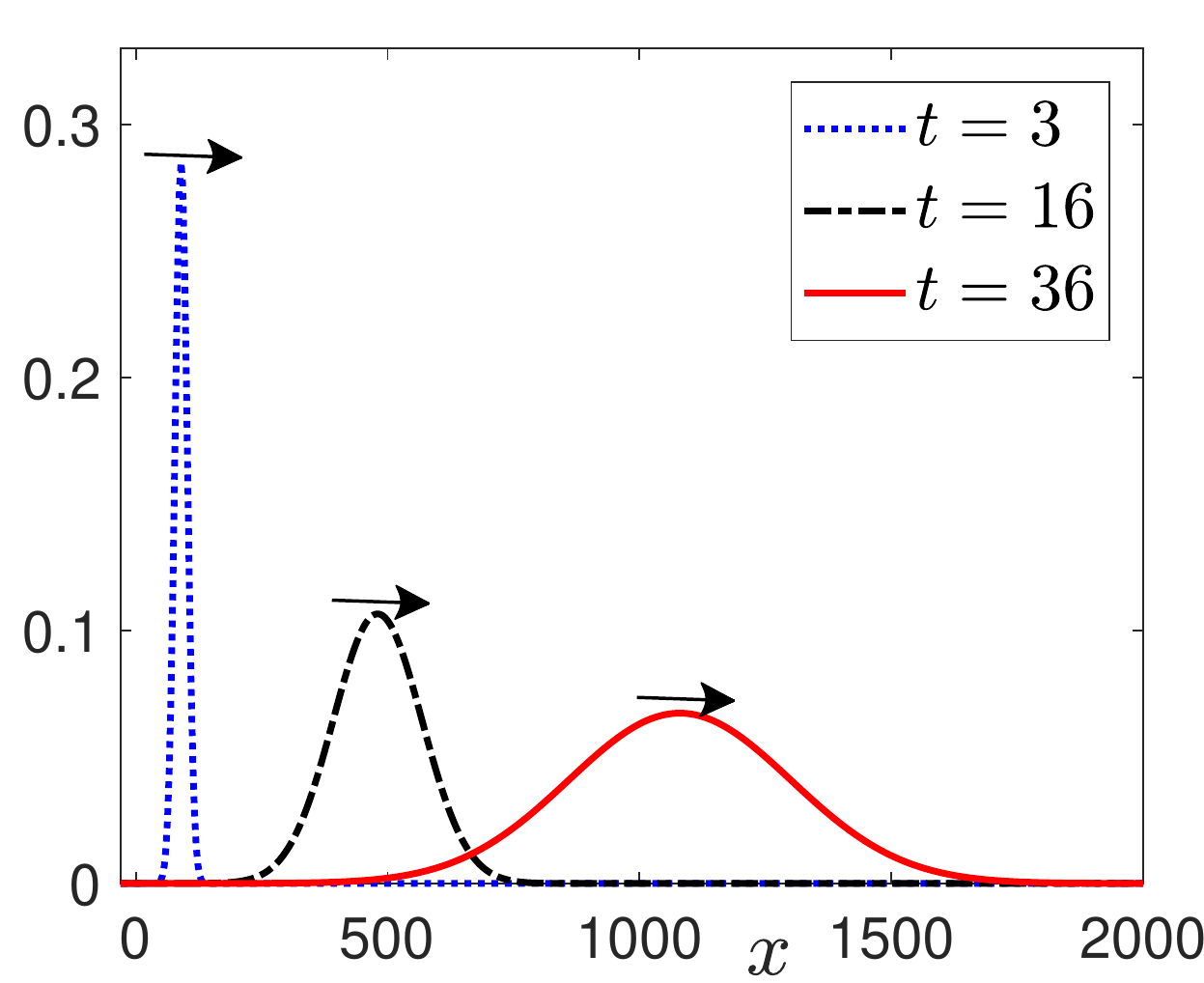}
\hspace{0.2cm}
\includegraphics[width=1.7in,height=1.55in]{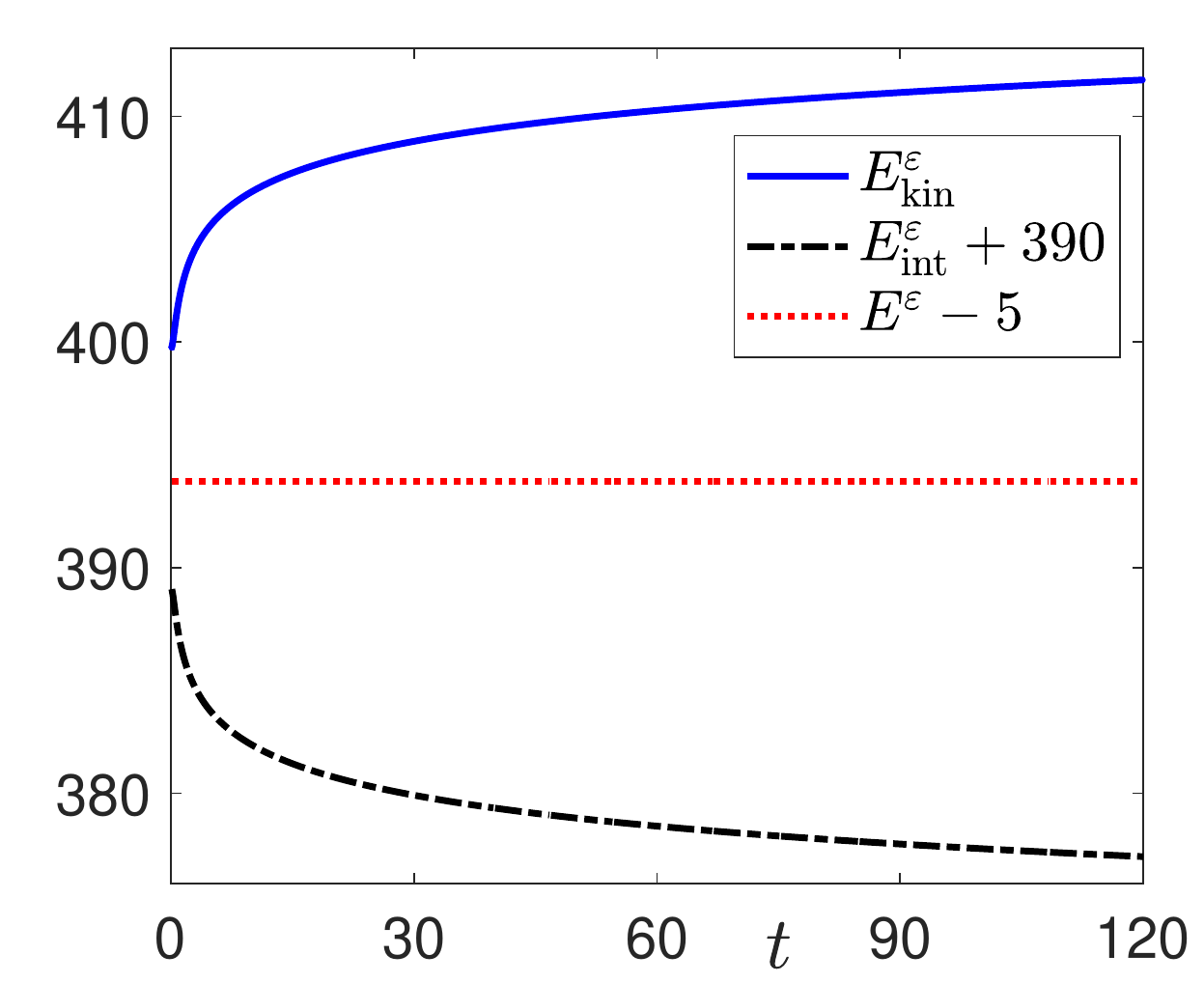}\\
\vspace{0.2cm}
\includegraphics[width=2.2in,height=1.7in]{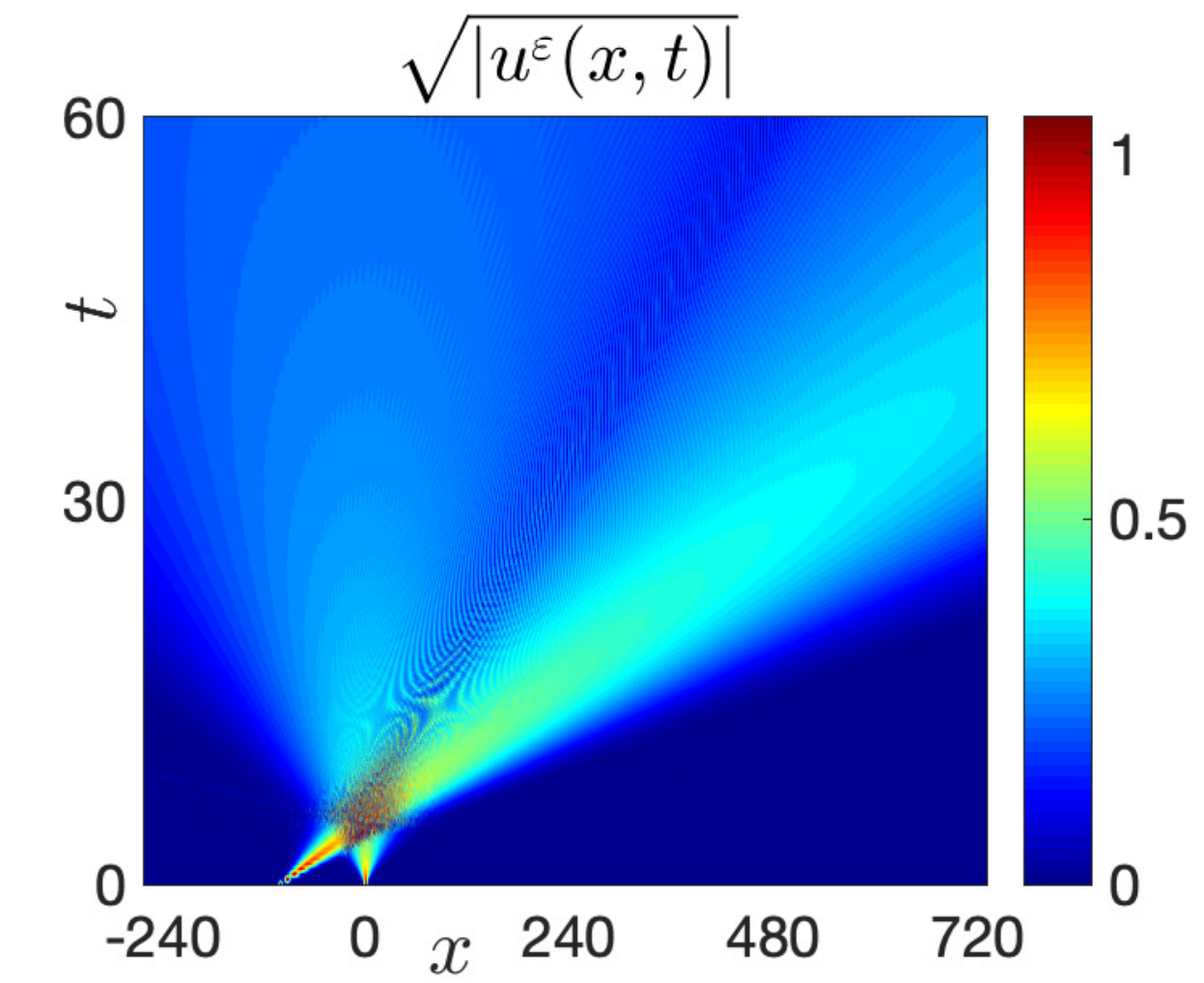}
\hspace{0.2cm}
\includegraphics[width=1.7in,height=1.55in]{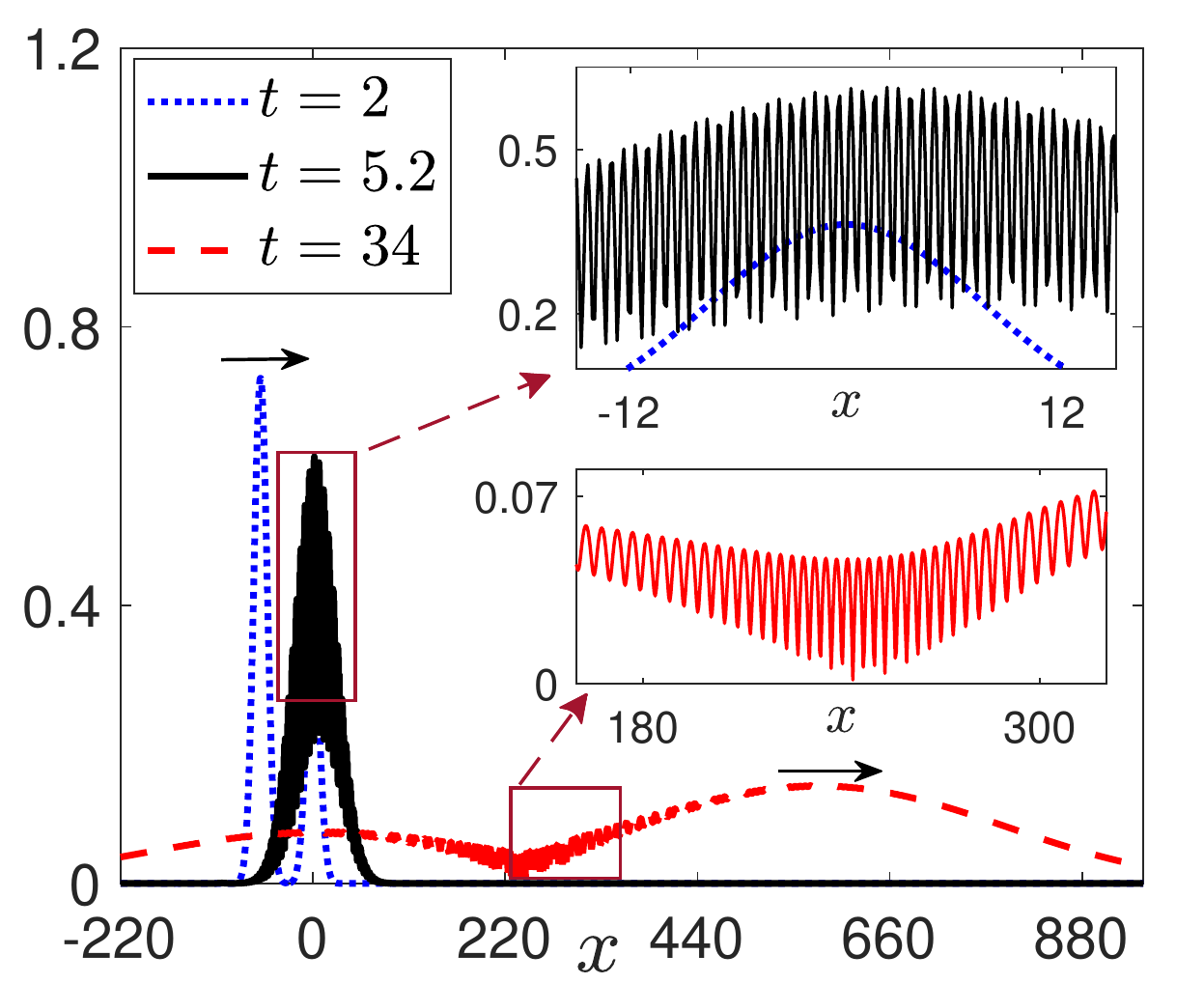}
\hspace{0.2cm}
\includegraphics[width=1.7in,height=1.55in]{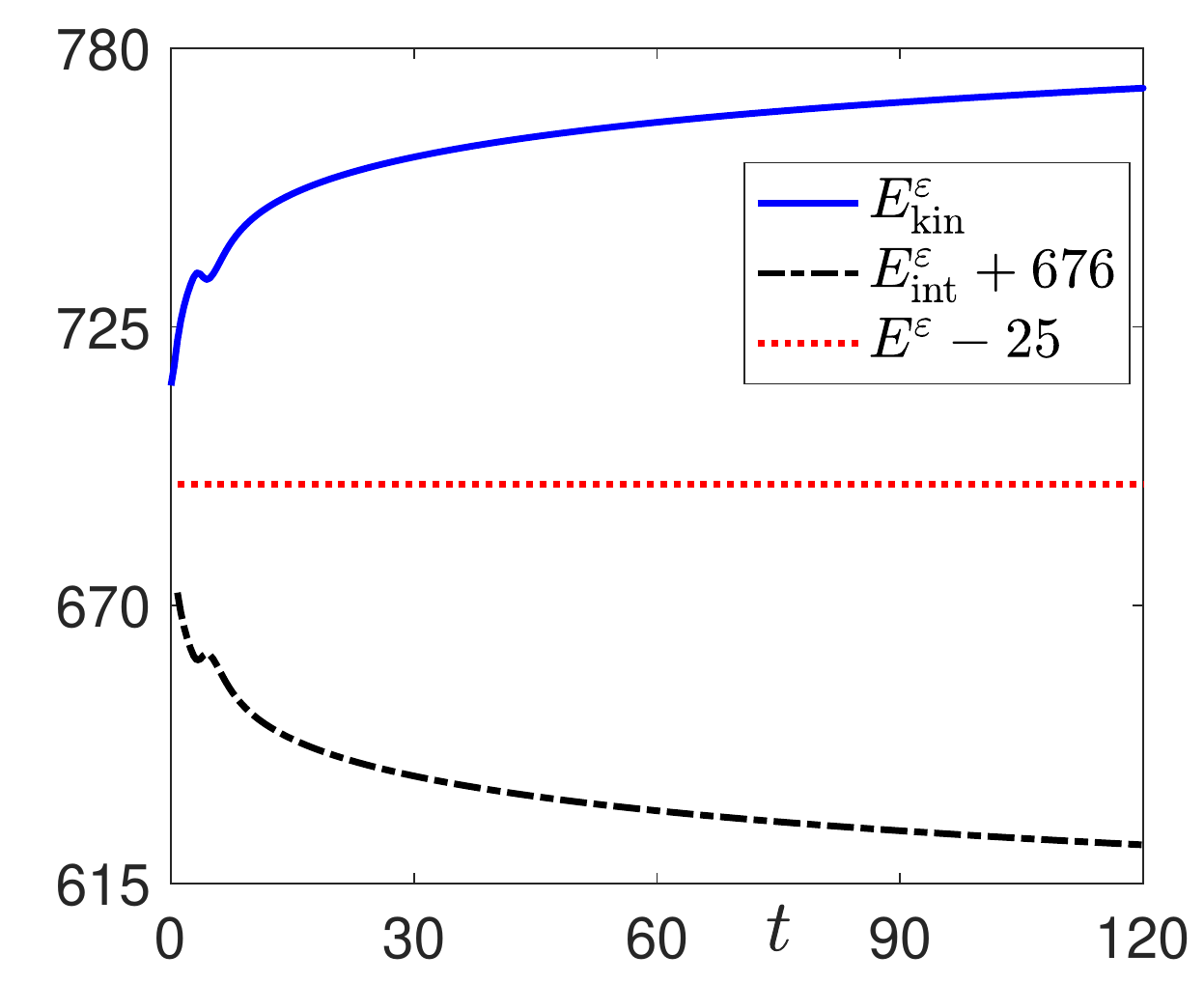}\\
\vspace{0.2cm}
\includegraphics[width=2.2in,height=1.7in]{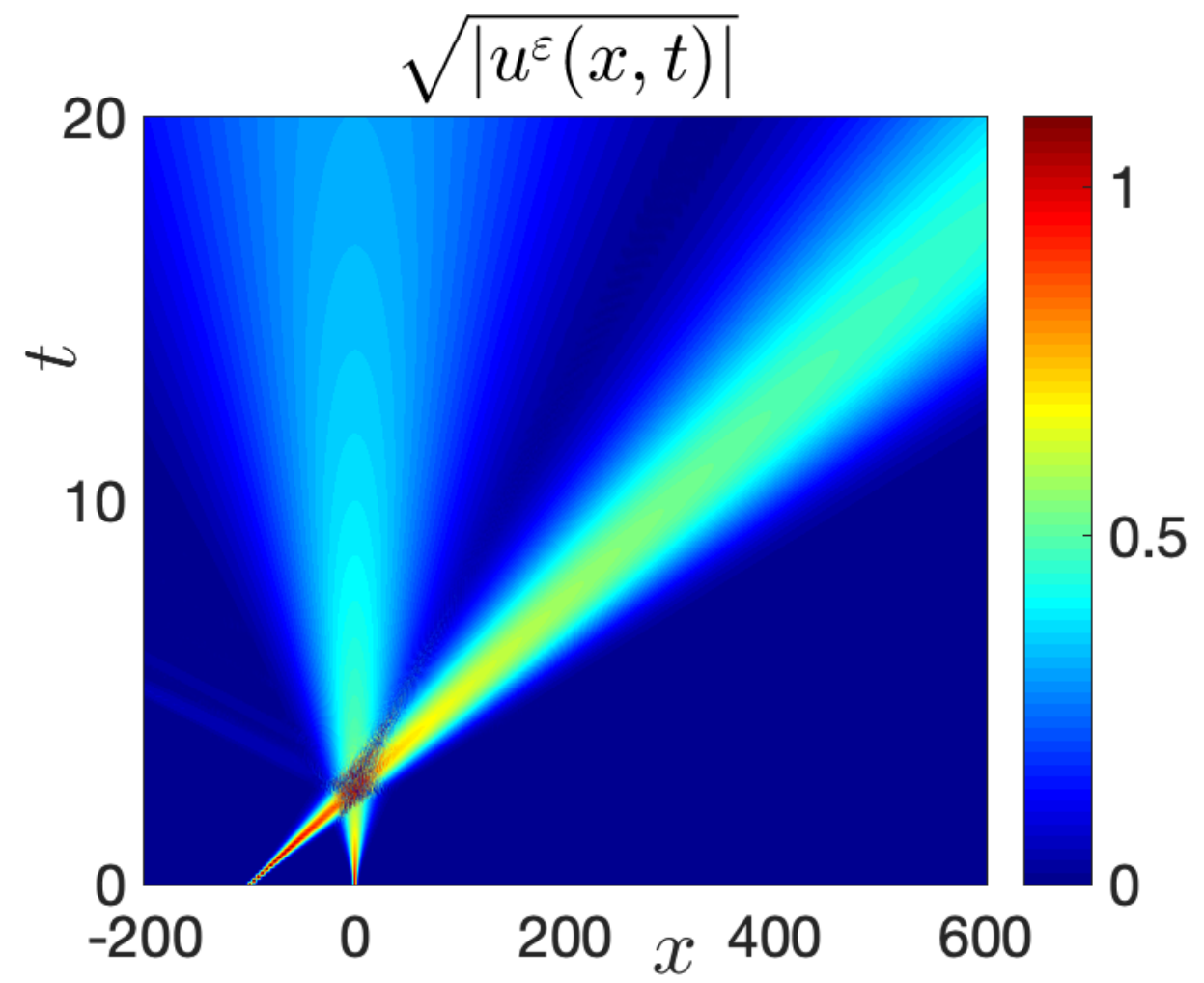}
\hspace{0.2cm}
\includegraphics[width=1.7in,height=1.55in]{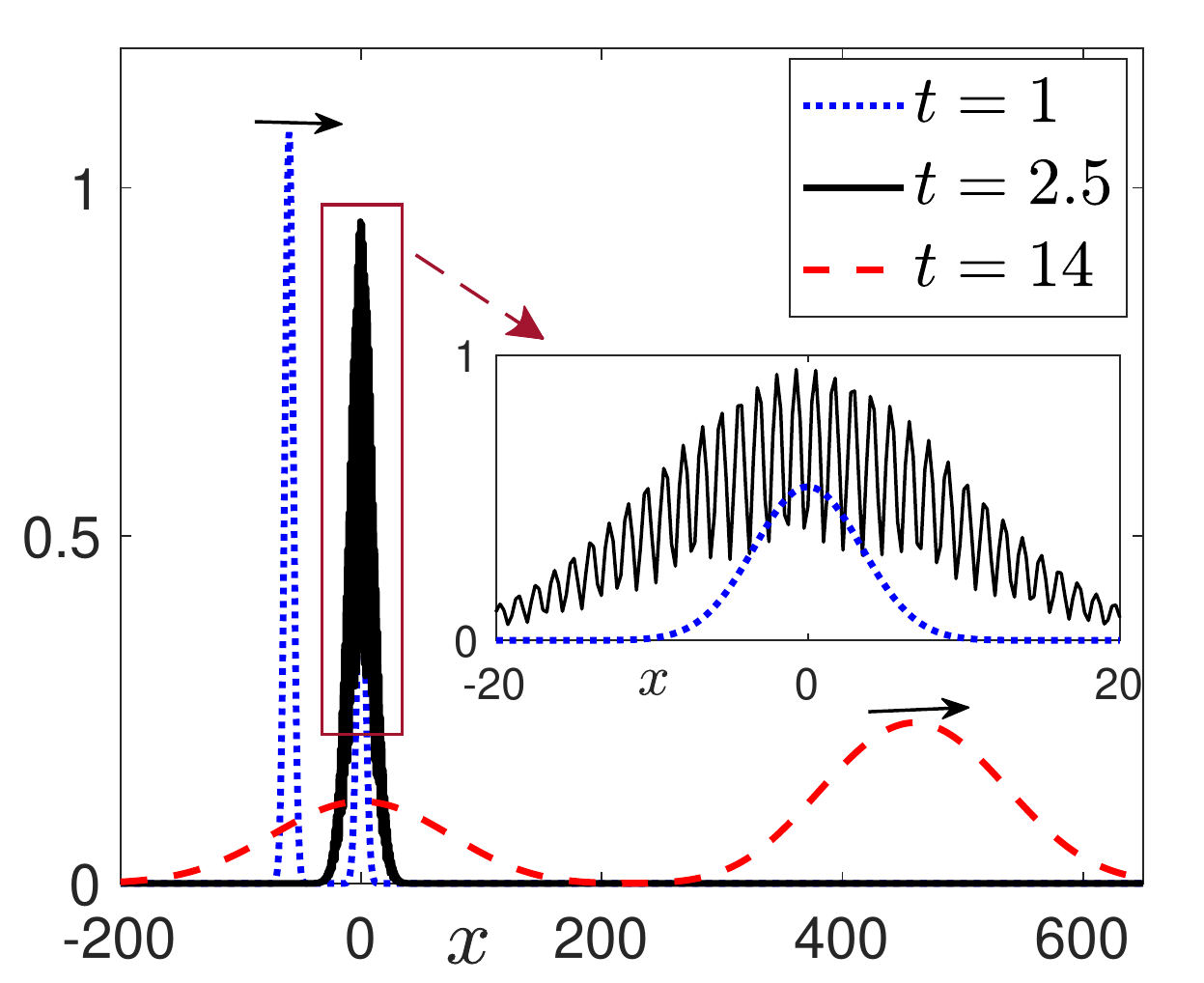}
\hspace{0.2cm}
\includegraphics[width=1.7in,height=1.55in]{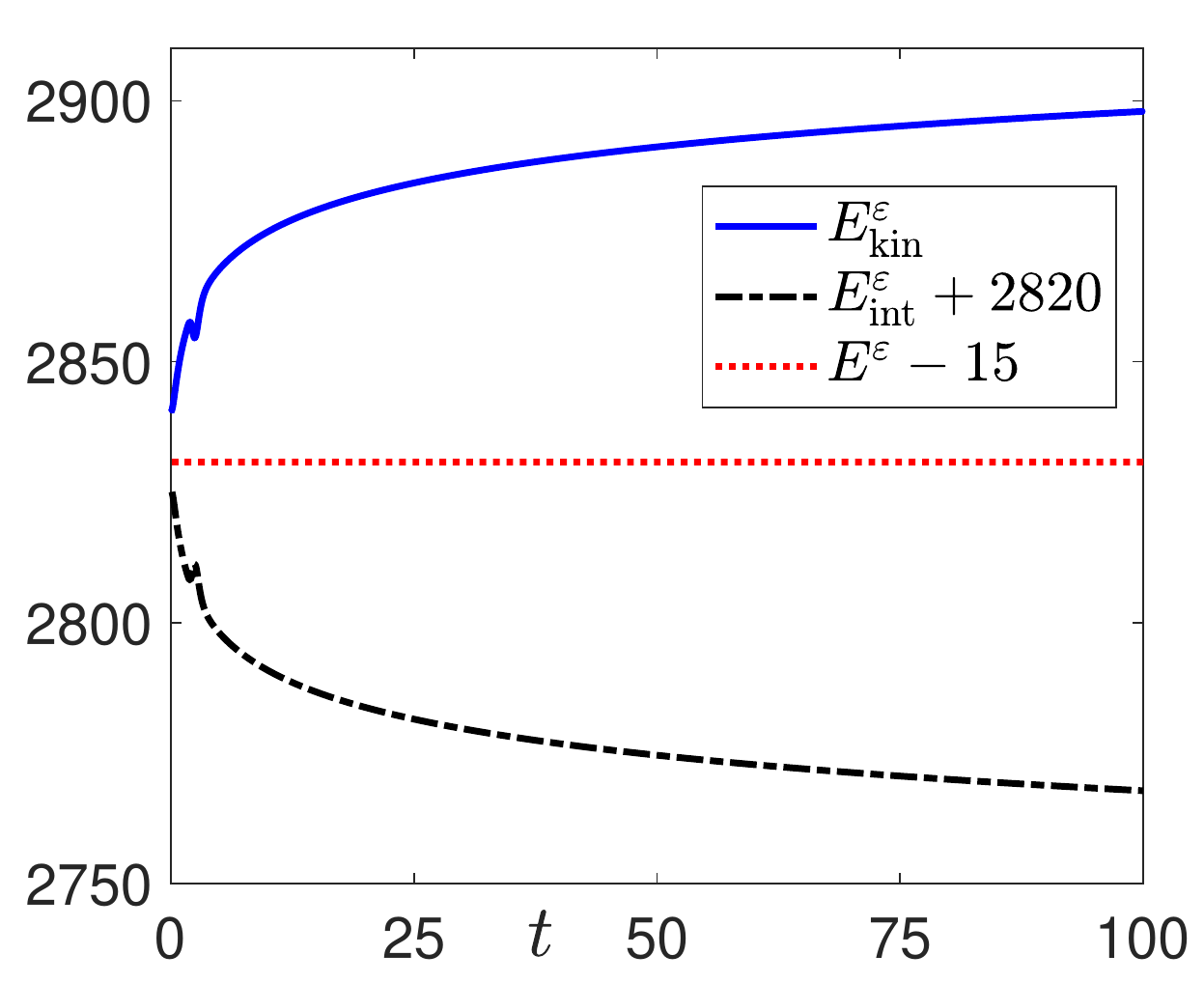}
\end{center}
\caption{Plots of $\sqrt{|u^\ep(x,t)|}$ (first column), $|u^\ep(x,t)|$ at different time (second column) and evolution of the energies (third column) for
different parameters in {\bf Example 4}:  Cases ix--xi (from top to bottom).  }
\label{fig:ex4-caseix-xi}
\end{figure}

\begin{figure}[htbp]
\begin{center}
\includegraphics[width=4in,height=2.5in]{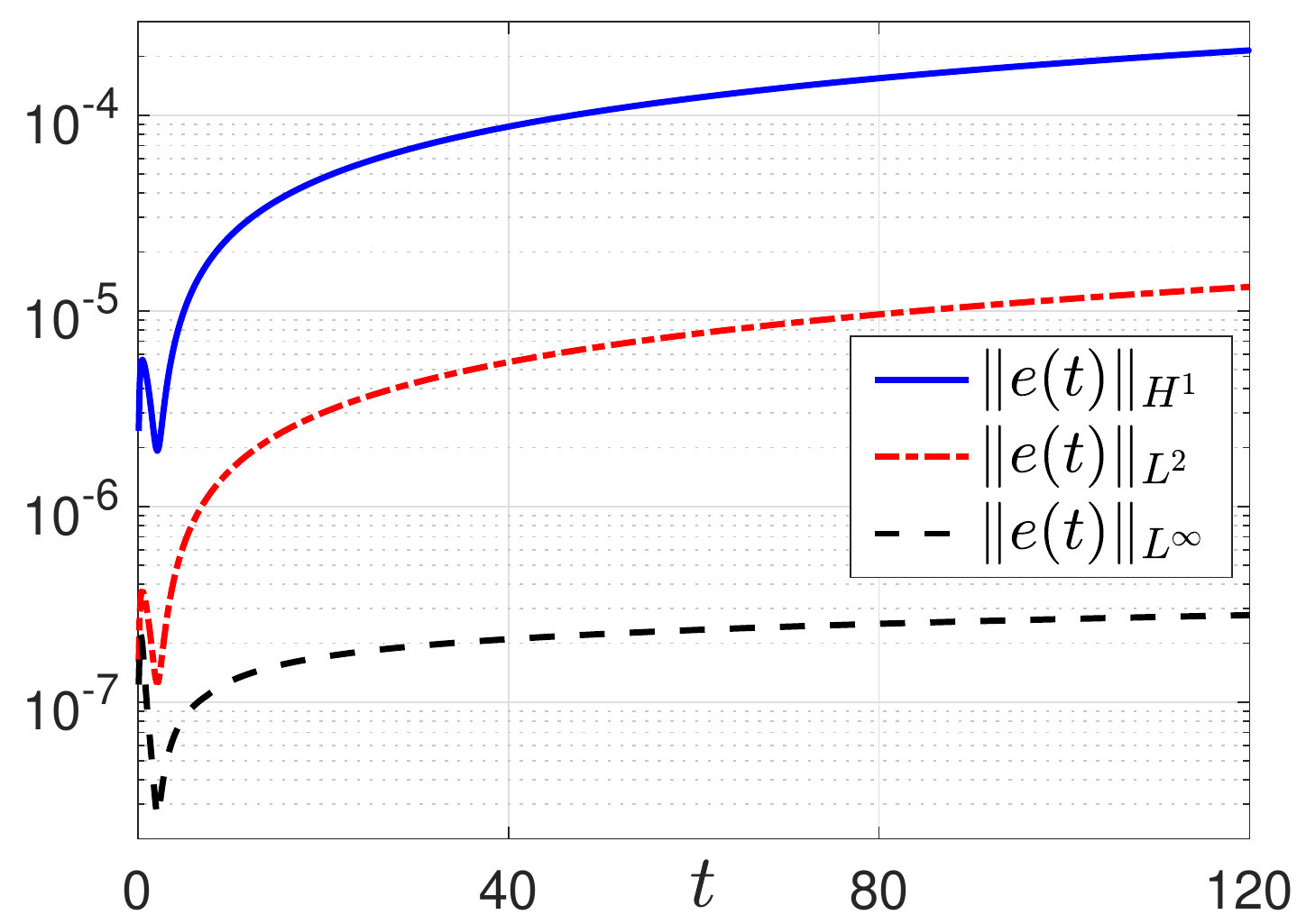}
\end{center}
\caption{Evolution of errors $\|e^\varepsilon(t)\|_{H^2}$, $\|e^\varepsilon(t)\|_{L^2}$  and
$\|e^\varepsilon(t)\|_{L^\infty}$  for Case (ix) in Example 4. }
\label{fig:ex4-caseix-error}
\end{figure}

\section{Conclusion}
We proposed and analyzed the regularized splitting methods and a regularized conservative finite difference method (CNFD) to solve the logarithmic Schr\"odinger equation (LogSE).
 For less regular initial setups, reduction of the standard order of accuracy for these methods in temporal direction is proved  theoretically
for the regularized Lie-Trotter splitting scheme, while also numerically observed for the regularized Strang splitting and CNFD schemes.
The method combining the regularized Strang-splitting scheme in time and spectral discretization in space is then applied to investigate
the long time dynamics of Gaussians for both positive and negative $\lambda$. It turns out that  the interaction of Gaussons in the LogSE is quantitatively  similar as the interaction of bright solitons in the cubic nonlinear Schr\"odinger equation.  However, there are also some qualitatively  different phenomena such as the breather-like dynamics and the spreading-out behavior when $\lambda<0$ and when $\lambda>0$ in the LogSE, respectively.
Our numerical results demonstrate rich and complicated dynamical phenomena
in the LogSE.

\bibliography{biblio}{}
\bibliographystyle{siam}
\end{document}